\documentclass[a4paper,11pt]{amsart}

\usepackage{amssymb} 
\usepackage{xcolor} 
\usepackage{fancybox}
\usepackage{dashbox}
\newboolean{showcomments}
\setboolean{showcomments}{true}
\ifthenelse{\boolean{showcomments}}
{ \newcommand{\mynote}[3]{
  \fbox{\bfseries\sffamily\scriptsize#1}
  {\small$\blacktriangleright$\textsf{\emph{\color{#3}{#2}}}$\blacktriangleleft$}}}
{ \newcommand{\mynote}[3]{}}
\definecolor{asparagus}{rgb}{0.53, 0.66, 0.42}

\usepackage{mathrsfs}

  \usepackage{float}

  \usepackage{pgffor}
  \usepackage{tikz}
  \usetikzlibrary{calc,3d,arrows,positioning,shapes.misc, plotmarks, shapes}

  \usepackage[english]{babel}
  \usepackage{amsthm}
  \usepackage{amsmath}
  \usepackage{amssymb}
  \usepackage{mathtools}
  \usepackage[latin1]{inputenc}
  \usepackage{setspace}
  \usepackage{enumerate}
  \usepackage{stmaryrd}
  \usepackage{algpseudocode}
  \usepackage{algorithm}
  \usepackage{multirow}

\definecolor{phase1}{HTML}{377EB8}
\definecolor{phase2}{HTML}{FF7F00}
\definecolor{phase3}{HTML}{4DAF4A}

  \usepackage[colorlinks=true,linkcolor=black,citecolor=black,filecolor=black,urlcolor=black,menucolor=black]{hyperref}

  \usepackage{array}
  \usepackage{graphicx}
  \usepackage{xcolor}
  \usepackage{wrapfig}
  \usepackage[babel]{csquotes}
  \usepackage{geometry}
  \usepackage{bbm}
  \usepackage{stmaryrd}
  \usepackage[all]{xy}
  \usepackage{mathrsfs}
  \usepackage{subcaption}
  \theoremstyle{plain}
  \newtheorem{theorem}{Theorem}
  \newtheorem*{theorem*}{Theorem}
  \newtheorem{proposition}{Proposition}
  \newtheorem{corollary}{Corollary}
  
  \newtheorem{lemma}{Lemma}
  \newtheorem{assumption}{Assumption}
  \theoremstyle{definition}
  \newtheorem{definition}{Definition}

  \newtheorem{rem}{Remark}

	\allowdisplaybreaks[3]

  \def \N{\mathbb{N}}
  \def \R{\mathbb{R}}
  
  \def \smedian{\left(m_\mu^\ell - \frac{1}{P}\mathbbold{1}\right)}

  \newcommand {\sign} {\mathop \mathrm{sign}} 
  
  \newcommand {\inj} {\mathop \textup{inj}(M)}

  \newcommand {\supp} {\mathop \textup{supp}}
  
 \renewcommand {\div} {\mathop \textup{div}}
 \newcommand {\divrho} {\textup{div}_\mu}
 
 \newcommand {\dist} {\textup{dist}}
 
 \newcommand {\Hess} {\mathop \textup{Hess}}
 \newcommand {\diag} {\textup{diag}}

  \newcommand {\eps} {\varepsilon}

\newcommand{\normxi}[1]{\left\lVert#1\right\rVert}
 
 \DeclareMathOperator*{\argmax}{\mathop \textup{arg\,max}}
 \DeclareMathOperator*{\Area}{\mathop \textup{Area}}
\DeclareMathOperator*{\argmin}{\mathop \textup{arg\,min}}
\DeclareMathOperator*{\Lip}{Lip}
\DeclareMathOperator*{\Vol}{Vol}
 
\DeclareMathAlphabet{\mathbbold}{U}{bbold}{m}{n}
 
\makeatletter
 \def\namedlabel#1#2{\begingroup
    #2%
    \def\@currentlabel{#2}%
    \phantomsection\label{#1}\endgroup
}
\makeatother

\title[Efficient volume-preserving MBO for clustering and classification]{An efficient volume-preserving MBO scheme for data clustering and classification}
\author{Fabius Kr{\"a}mer}

\author{Tim Laux}
\address{Fakult{\"a}t f{\"ur} Mathematik, Universit{\"a}t Regensburg, Universit{\"a}tsstra{\ss}e 31, 93053 Regensburg, Germany (\texttt{{fabius.kraemer,tim.laux}@ur.de})}
    \date{\today}

\begin{document}

    \begin{abstract}
	
We propose and study a novel efficient algorithm for clustering and classification tasks based on the famous MBO scheme. On the one hand, inspired by Jacobs et al.~[J.\ Comp.\ Phys.\ 2018], we introduce constraints on the size of clusters leading to a linear integer problem. We prove that the solution to this problem is induced by a novel order statistic. This viewpoint allows us to develop exact and highly efficient algorithms to solve such constrained integer problems. 
On the other hand, we prove an estimate of the computational complexity of our scheme, which is better than any available provable bounds for the state of the art. This rigorous analysis is based on a variational viewpoint that connects this scheme to volume-preserving mean curvature flow in the big data and small time-step limit.

	\medskip
    
  \noindent \textbf{Keywords:} Clustering, Classification, semi-supervised learning, mean curvature flow, gradient flows.

  \medskip

\noindent \textbf{Mathematical Subject Classification (MSC2020)}:
68Q25; 
90C10; 
53E10 (Primary); 
58J35; 
53Z50; 
49Q20; 
49Q05 (Secondary). 
  \end{abstract}
\maketitle


\tableofcontents

\section{Introduction}

\subsection{Motivation}
Many modern machine learning methods, such as classification via neural networks, rely on large amounts of labeled data. 
However, for many applications, only little information is known about the data: 
Only a few data are labeled, or only average information, such as the rough size of each class, is known.
Therefore, it is crucial to develop efficient methods that do not require many or any labeled data at all. 
In this semi- or unsupervised learning regime, it is paramount to understand and exploit the geometry of the whole data cloud.
A successful avenue for this is graph-based learning, which equips the dataset with a graph structure by connecting data points according to some similarity measure.
Then, one aims to partition this similarity graph into clusters so that only few edges lead from one cluster to another while many edges connect points within each cluster~\cite{MR1265105}.

\medskip

Finding these minimal cuts (with hard or soft constraints on the sizes of clusters) in large graphs gives rise to NP-hard problems~\cite{MR1265105}. 
Therefore, efficient algorithms are needed that solve these problems approximately. 
One of the most efficient methods to do this is the MBO scheme, which translates the original scheme of Merriman, Bence, and Osher~\cite{MBO94, levelset} to the setting of graphs. 
It was proposed by Bertozzi et al.~\cite{MR3115457, 6714564, van2014mean} and has received continuous attention since then, see e.g.~\cite{jacobsvoronoi, jacobs2018auction}. 
The scheme uses the geometry of the dataset to improve an initial guess for a clustering or classification task (for example starting from random assignments or linear methods such as k-Means) by propagating label information from nearby data points. 
More precisely, it iteratively improves this guess by alternating between linear diffusion of the labels and pointwise thresholding. 
One drawback is that clusters may shrink or even disappear. 

\medskip

In the present paper, we propose a \emph{provably efficient} and \emph{exact} method to extend the MBO scheme to the case of volume constraints which in particular prevent this shrinkage.
The key algorithmic challenge is that the trivial thresholding step must be replaced by a linear integer program. 
Our main contributions are (a) the development of efficient algorithms for a class of such linear integer programs and (b) a novel efficiency analysis that rigorously proves the observed speed-up of these algorithms in the volume-constrained MBO scheme.
Our complexity estimate relies on the variational structure of the algorithm as a minimizing movements scheme for Volume-Preserving Mean Curvature Flow.

\subsection{Our New Algorithm}

The volume-preserving MBO scheme alternates between linear diffusion and pointwise thresholding, now at threshold values that are to be determined: Given an initial partition of the point cloud into clusters~$X_N= \Omega_1^0\cup \ldots \cup \Omega_P^0$, diffuse each cluster independently for a short time using the heat semi-group $e^{-t \Delta_N}$ associated to a graph Laplacian $\Delta_N$; then find suitable threshold values that lead to a partition of the same volumes. 
More precisely, given a time-step size $h>0$ for $\ell=1,2, \ldots$ 
\begin{enumerate}[1.]
\item  Diffusion: $ u^{\ell-1}_i := e^{-h\Delta_N} \chi_{\Omega_i^{\ell-1}} \quad (1\leq i \leq P)$;
\item Thresholding: 
$\Omega_i^\ell := \{ u^\ell_i - m^\ell_i > u^\ell_j - m^\ell_j \text{ for all } i \neq j\}$,
where the point $m^\ell=(m^\ell_i)_{1\leq i\leq P}$ needs to be chosen such that the number of points in each cluster $\# \Omega^\ell_i=V_i$ is preserved.
\end{enumerate} 
Following Esedo\u{g}lu--Otto~\cite{esedog2015threshold} and Jacobs--Merkurjev--Esedo\u{g}lu~\cite{jacobs2018auction}, the second step can be formulated as a linear integer program with constraints:
\begin{align}\label{alg:volumeMBO}
\chi^\ell \in \argmax_{\chi:X_N \rightarrow \{0,1\}^{P}} &\sum_{i=1}^P \sum_{x \in X_N} \chi_i(x) u_i(x)\\
\text{s.t. } &\sum_{i=1}^P \chi_i(x) = 1 \quad \hspace{3pt} \text{ for all } x \in X_N, \nonumber\\
\hspace{20.5pt}&\!\!\sum_{x \in X_N}\!\! \chi_i(x) = V_i \quad \text{ for all } i = 1,\dots,P, \nonumber
\end{align}
where $u = u^{\ell - 1}$ is given by Step 1. Jacobs--Merkurjev--Esedo\u{g}lu~\cite{jacobs2018auction} interpret this optimization problem as an auction in which each cluster bids for data points; then they apply a black-box auction algorithm by Bertsekas~\cite{bertsekas1979distributed} to solve this assignment problem approximately.

We propose a new algorithm for solving~\eqref{alg:volumeMBO} that is particularly suited for the parameter-regime and distribution of the values $u^{\ell-1}$ in the MBO scheme. 
Our new algorithm (Algorithm~\ref{alg:median} below) solves this problem~\eqref{alg:volumeMBO} exactly and so efficiently that it is faster than most approaches for the diffusion step (which is a mere matrix-vector multiplication), see Theorem~\ref{the:runningtime} and Theorem~\ref{the:improvedRunning} below. 
We present several approximations for computing the diffusion step and point out a family of sparse matrices for which the computational complexity of the matrix-vector multiplication matches the one of our algorithm for~\eqref{alg:volumeMBO}.

\medskip

Although our algorithm was designed for the specific situation of the MBO scheme, it also performs very well on arbitrary data distributions $u$.
In particular, the extension to inequality constraints 
\begin{align}\label{alg:inequalityMBO}
\chi^\ell \in \argmax_{\chi:X_N \rightarrow \{0,1\}^P} &\sum_{i=1}^P \sum_{x \in X_N} \chi_i(x) u_i(x)\\
\text{s.t. } &\sum_{i=1}^P \chi_i(x) = 1 \quad \hspace{3pt} \text{for all } x \in X_N, \nonumber\\
\hspace{20.5pt} L_i \leq & \!\!\sum_{x \in X_N} \!\!\chi_i(x) \leq U_i \quad \text{for all } i = 1,\dots,P. \nonumber
\end{align}
in Algorithm~\ref{alg:lower_upper} has a provable complexity of $O(N\log N)$ for any given $u$ while the state of the art needs $O(N^2)$ operations \cite{jacobs2018auction}.

\medskip

A first intersting insight already appears in the scalar case when the dataset is partitioned into only two clusters, say, of equal size: 
This is commonly done by sorting the points according to their diffused label value (e.g.\ using quick sort) at a cost of~$\mathcal{O}(N \log N)$ for~$N$ data points. 
After this sorting procedure, the first half of the data points are assigned to the first cluster and the second half to the second one.
We claim that this whole procedure can be replaced by one that only needs~$\mathcal{O}(N)$ operations. 
Indeed, it is sufficient to simply find the median of the diffused label values and then point-by-point threshold at this value. 
Both steps, the median finding and the thresholding now only take~$\mathcal{O}(N)$ operations. Of course, the vectorial case is much more delicate, but the geometric interpretation of~$m$ as a vectorial median turns out to be an effective guiding principle for this task. 

\medskip

To describe our new algorithm, let us first focus on the simplest nontrivial case of three clusters and exactly prescribed sizes of the clusters (as opposed to more intricate inequality constraints described later, see Algorithm~\ref{alg:median}) and that all clusters should be of equal size (which can easily be adapted to arbitrary assigned sizes). 
In this case it is instructive to picture the standard simplex~$\Sigma = \mathrm{conv} (e_1,e_2,e_3) $ in~$\mathbb{R}^3$ and think of each vertex representing one cluster. 
(Indeed, this is precisely the image of the vector field of indicator functions~$(\chi_1,\chi_2,\chi_3)$.)
During the diffusion, the labels $u(x,t)=(e^{-t\Delta}\chi_1(x),e^{-t\Delta}\chi_2(x),e^{-t\Delta}\chi_3(x))$ move away from the vertices towards the interior of the simplex, see Fig. \ref{fig:diffusion_over_time}. 
Now, after a short time $h$, we want to find a point~$m$ in~$\Sigma$ which is a vectorial median of the diffused vectorial labels~$\{u(x,h) \colon x\in X\}$ in the sense that the $Y$-shaped union of line segments through $m$ in Fig.~\ref{fig:3median} below divides the simplex $\Sigma$ into three parts, each containing the same amount of points.
This can be achieved by the following iterative combinatorial process based on an elementary geometric intuition.

Starting from an initial guess for~$m$ for which, say, the first cluster is too small, 
we want to move~$m$ in order to increase the size of this cluster. 
As~$m$ passes points of a cluster with too many points, we simply move~$m$ past those points so that these points are now counted towards the previously too-small first cluster. 
However, if~$m$ passes a point $u(x)$, say, in the second cluster and that cluster does not have too many points assigned to it, we do not want to reassign points between these two clusters as this may result in an endless loop of reassigning points. 
Instead, we stop~$m$ here and move it into a different direction that leaves this point exactly at the border between Clusters~1 and~2. 
We continue to move~$m$ until we hit a point of the third cluster. 
As we only have three clusters, we know that Cluster~3 has too many points and will gladly give the point at the border to either Cluster~1 or~2, whoever is on the other side of the border. 
In case this is Cluster~1, we achieved our goal: We increased the size of Cluster~1 by one data point without decreasing any other too-small cluster. 
In case Cluster~2 is on the other side of this data point, we make a three-way swap: Cluster~3 gives the data point on its border to Cluster~2 which in turn gives the data point on its border to Cluster~1. 
Iterating until all clusters have equal size gives us an~$m$ at which we can threshold and preserve the volume of each cluster. 
Just as in the scalar case, this point~$m$ can be viewed as some sort of vectorial median of the diffused labels of the data points.

\subsection{Estimating the Computational Complexity Via Gradient Flow Techniques}
\begin{wrapfigure}{r}{0.4\textwidth} 
    \centering
    \includegraphics[width=0.4\textwidth]{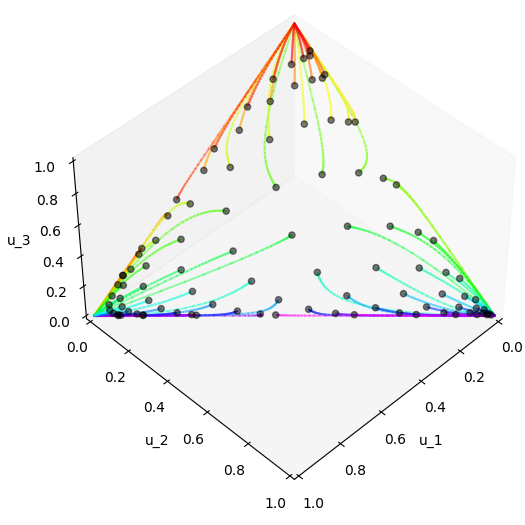}
    \caption{Diffusion of labels $e_1,e_2,e_3$ over time $h$.}
    \label{fig:diffusion_over_time}
\end{wrapfigure}
This new simple algorithm  is surprisingly efficient. One can easily see that it terminates in~$\mathcal{O}(N \log N)$ operations. Furthermore, we will prove that we only need~$\mathcal{O}(\sqrt{h} N \log N)$ operations to find the median $m$ when starting the median search in the MBO scheme from the center of the simplex or from the median of the previous step, in a typical step.
The idea is that in case of the MBO scheme, most label values will remain close to the pure cluster points~$e_1, e_2, e_3$ and only few data points lay close the decision boundary and need to be considered in the above loop (see Fig. \ref{fig:diffusion_over_time}). 
Here,~$\sqrt{h}$ corresponds to the width of the heat kernel. 
We show this rigorously for sufficiently large point clouds under the so-called manifold assumption stating that the data points are distributed according to some probability measure supported on some manifold. 
Under these assumptions, it is known that in the large-data limit the unconstrained MBO scheme behaves like its natural mean-field limit---the MBO scheme on a weighted manifold~\cite{laux2021large, laux2023large, ullrich2024medianfiltermethodmean}. 
We will see easily that this is also true in the case of our volume-constrained MBO scheme. 
Our rigorous proof of the speed-up by the factor $\sqrt{h}$ then consists of three steps. 
\begin{enumerate}[1.]
	\item\label{intro_L2} We prove that in a typical iteration, the vectorial median of the mean-field MBO scheme is~$\sqrt{h}$ close to the center of the simplex. This is the most technical part of our proof and requires a careful $L^2$ estimate of the Lagrange multiplier of the minimizing movements.
	\item\label{intro_median}	We show that the discrete medians converge to their mean-field counterparts in the large-data limit. This follows from a careful analysis of the variational interpretation of our vectorial median, see Section~\ref{sec:convergence}.
	\item \label{intro_speedup} Finally, we prove that since by~\ref{intro_L2}.\	 and~\ref{intro_median}., in a typical iteration, the discrete vectorial median is~$\sqrt{h}$-close to the center of the simplex, the running time of the median find algorithm is sped up by the factor~$\sqrt{h}$.
\end{enumerate}
 The analytically most challenging part is the $L^2$-estimate in Step \ref{intro_L2}.
 This result in particular generalizes previous work of Swartz and one of the authors~\cite{laux2017convergence} to~(a) the vectorial case and~(b) the case of (weighted) manifolds. 
 We are convinced that the techniques developed here will also lead to a convergence proof of the volume-preserving MBO scheme and the vectorial MBO scheme on weighted manifolds, which are both still open questions.

\medskip

The efficiency of our median find algorithm gives rise to the question of whether the diffusion step can be carried out more efficiently so that both complexities are balanced. 
The most efficient version we find is an approximation of the diffusion semi-group by a sparse matrix.  
In addition, we show that the difference between two consecutive time steps of the scheme is sparse. 
All in all, this turns the diffusion step into the multiplication of an~$N \log N$-sparse~$(N \times N)$-matrix with an~$\eps N$-sparse~$N$-vector (where $\eps$ is the length scale in constructing the similarity graph). This can be carried out at the cost of~$\mathcal{O}(\eps N \log N)$, matching the cost of our median find algorithm when using the classical parabolic scaling regime $\eps \sim \sqrt{h}$.
The~$\eps N$-sparsity of the increment is the result of an a-priori estimate and a simple interpolation, see Lemma~\ref{the:l1estimate} which is inspired by~\cite{MR3556529}.

\subsection{Context}\label{sec:intro_context}
The scheme and our results have connections to different seemingly disconnected communities. There is a lot of literature that connects the MBO scheme with graph learning and data science tasks. 
The first to introduce the MBO scheme on graphs were Bertozzi et al. \cite{MR3115457, 6714564, merkurjev2014diffuse} where they used the scheme also for imaging tasks. Afterwards there have been several papers that modified the scheme like Jacobs \cite{jacobsvoronoi} to a wide range of diffusion kernels, Jacobs-Merkurjev-Esedoglu \cite{jacobs2018auction} to volume constraints, and Calder-Cook-Thorpe-Slep{\v c}ev PoissonMBO \cite{calder2020poisson} making a connection to the Poisson equation and a different approach to impose volume constraints.
We also note that the influential work on diffusion maps by Coifman and Lafon et al.~\cite{MR2238665, MR2438821, coifman2005geometric} already introduced a one-step version of the MBO scheme for classification in which the thresholding step is interpreted as a maximum a posteriori estimator; our algorithm can also be used in their context to preserve the sizes of classes.

\medskip

Recently, there has also been a rising development in Big Data limits. This started with  Garc{\'i}a Trillos and Slep{\v c}ev who introduced the $TL^p$ metric as framework for comparing discrete with continuous schemes. $\Gamma$-convergence of the thresholding energies, via the weak $TL^2$ topology as first done by \cite{laux2021large}, is also one of the building blocks to show our mean field limit of the scheme.  To the best of our knowledge the present work is the first to use information on the Big Data limit to obtain a computational complexity estimate.

\medskip

We want to point out that our volume constrained MBO scheme also has a less obvious connection to the well established method of spectral clustering \cite{MR2409803}.
Indeed, writing the heat semigroup applied to some vector $\chi:X_N \rightarrow \{0,1\}$ by its well known spectral decomposition
\begin{equation}\label{eq:spectral_decomposition}
e^{-h \Delta_N} \chi = \sum_{i = 1}^N e^{-h \lambda_i} \langle \chi, \phi_i \rangle \phi_i
\end{equation}
where $\lambda_i$ and $\phi_i$ are the eigenvalues and eigenvectors of a graph Laplacian $\Delta_N$ we consider the limit $h \rightarrow \infty$. Assume for simplicity that the graph is connected and that all eigenvalues are simple. Then $0 = \lambda_1 < \dots < \lambda_N$ such that the prefactor $e^{-h \lambda_1} = 1$ is independent of $h$ and $e^{-h \lambda_2}$ is the biggest prefactor dependent on $h$. Thus one has
\begin{align*}
e^{-h \Delta_N} \chi = \langle \chi, \mathbbold{1} \rangle \mathbbold{1} + e^{-h \lambda_2} \langle \chi, \phi_2 \rangle \phi_2 + h.o.t.
\end{align*}
Take the MBO scheme without volume constraints for two phases then the next clustering is $\{\chi = 1\} = \{x \in X_N: \phi_2(x) > 0\}$ provided $ \langle \chi, \mathbbold{1} \rangle \mathbbold{1} = \frac{1}{2}$. This is also known under spectral clustering (see for example section 5.1 in \cite{MR2409803}). The second eigenvector is of interest for clustering as it solves the relaxed problem of minimizing the RatioCut \cite{MR2409803}.  However, generically $ \langle \chi, \mathbbold{1} \rangle \mathbbold{1} \neq~\frac{1}{2}$ and the $0$-th order term dominates for $h$ large enough such that every point will be assigned to the same phase, i.e. $\chi = \mathbbold{1}$ or $\chi = 0$; this is a bad clustering. This does not happen for our volume constrained scheme: In this case, the produced clustering for $h$ large is more meaningful. As the constant summand $\langle \chi, \mathbbold{1} \rangle \mathbbold{1}$ does not matter for the thresholding under volume constraints, the clustering will be given by the $V$ biggest elements of $\phi_2$. 

\medskip

For three phases with volumes $V_1, V_2$ and $V_3$, one gets for example the $V_1$ biggest, $V_3$ smallest and $V_2$ middle points in the respective phases if $\langle \chi_1, \phi_2 \rangle \geq \langle \chi_2, \phi_2 \rangle \geq \langle \chi_3, \phi_2 \rangle$ holds; and similarly for more phases. One observes on Fig. \ref{fig:four_images} that those are indeed good clusterings. In practice one would use for spectral clustering not only the second eigenvector but rather the first $k$ eigenvectors \cite{MR2409803} which corresponds to having a big but finite $h$. This is one reason why we use the in frequency cut-off matrix $\sum_{i = 1}^{\log N} e^{-h \lambda_i} \langle \chi, \phi_i \rangle \phi_i$ as one of the previously mentioned cheap approximations of the heat kernel. This matrix is not sparse but due to its low rank, the previous mentioned matrix-vector multiplication can still be computed in $O(N \log N)$ matching our worst case running time of the thresholding step.

\medskip

In the connection to spectral clustering above the limit $h \rightarrow \infty$ was easy to see. 
But to recover Volume-Preserving Mean Curvature Flow as the limit $h \rightarrow 0$ (while $N \rightarrow \infty$) is a more delicate procedure. 
Indeed, this is a purely geometric flow and therefore has geometric invariances that lead to degeneracies in the underlying partial differential equations.
Volume-Preserving Mean Curvature Flow first appeared in the differential geometry literature: 
Gage~\cite{MR848933} used it to continuously deform convex curves into circles enclosing the same area, and Huisken~\cite{MR921165} introduced the flow in arbitrary dimensions to flow convex surfaces into round spheres that enclose the same volume. 
While the convergence of our scheme to Volume-Preserving Mean Curvature Flow is not proven yet, the convergence in this joint limit to the standard (weighted) Mean Curvature Flow was shown in \cite{laux2023large} in the case of two clusters and no volume constraints using a geometric comparison principle. In the case of more than two clusters no such comparison principle is available. However when starting from the mean-field limit and assuming constant density and flat geometry, the convergence to weak solutions can be shown based on the gradient-flow structure of Mean Curvature Flow \cite{MR3556529, MR4385030, MR4056816, laux2023large}.

\medskip

We work under the aforementioned manifold assumption. Thus, our analysis has many natural connections to differential geometry that are not present in \cite{laux2017convergence, MR3556529, MR4385030, MR4056816, laux2023large}. Two examples that we want to highlight in this context are the following: (i) We adapt in Lemma \ref{lem:bound_by_grad} the Segment Inequality of Cheeger and Colding \cite{MR1320384} to the context of the heat kernel $p(h,x,y)$ on a (weighted) manifold. At the base of the proof the problem is that in contrast to the Euclidean space on a manifold one does not have translation invariance of geodesics. The idea of Cheeger and Colding is to use polar coordinates around a point to align the coordinate system with the geodesics. Therefore, one can shift the geodesics along the polar coordinate of a fixed angle. 
(ii) In our $L^2$-estimate of Proposition \ref{the:l2estimate} cancellation  effects of the gradient of the heat kernel play a crucial role. Again, denoting by $p(h,x,y)$ the heat kernel at time $h$, in the Euclidean space one has the crucial antisymmetry $\nabla_x p + \nabla_y p = 0$, which does not make sense geometrically as these two vectors live in different tangent spaces. On a manifold, denoting by $\tilde{p}(h,x,y) = \frac{1}{(4\pi h)^{d/2}} \exp(-\frac{\dist^2_M(x,y)}{4h})$ the approximate heat kernel at time $h$, the gradients do not cancel each other completely but one can still see a cancellation when testing with a vector field $\xi$:
\begin{equation*}
\left|\langle \xi(x), \nabla_x \tilde{p}(h,x,y)\rangle_x + \langle \xi(y), \nabla_y \tilde{p}(h,x,y)\rangle_y\right| \leq \frac{\dist^2_M(x,y)}{2h} \Lip(\xi) \tilde{p}(h,x,y).
\end{equation*}
\color{black}

\begin{figure}[h!]
    \centering
    \begin{subfigure}[b]{0.24\textwidth}
        \includegraphics[width=\textwidth]{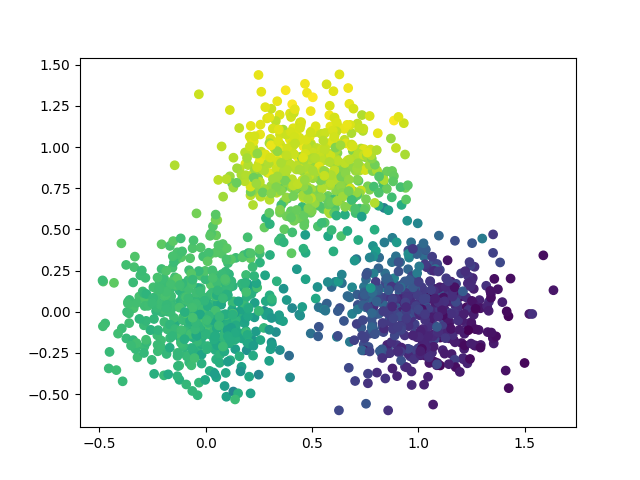}
        \caption{Second eigenvector of the graph Laplacian.}
        \label{fig:image1}
    \end{subfigure}
    \begin{subfigure}[b]{0.24\textwidth}
        \includegraphics[width=\textwidth]{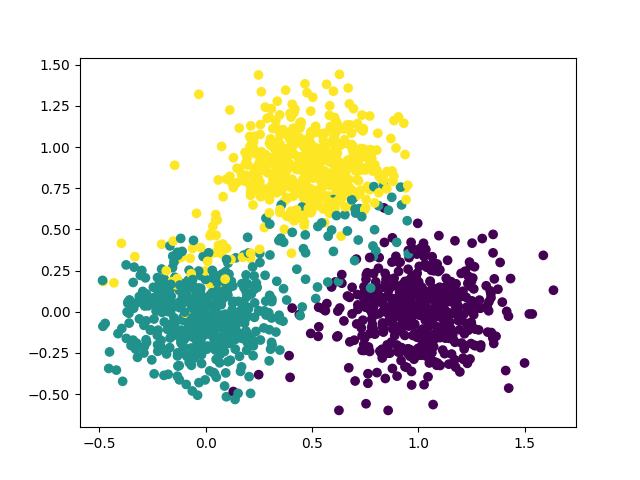}
        \caption{Clustering resulting from the eigenvector.}
        \label{fig:image2}
    \end{subfigure}
    \begin{subfigure}[b]{0.24\textwidth}
        \includegraphics[width=\textwidth]{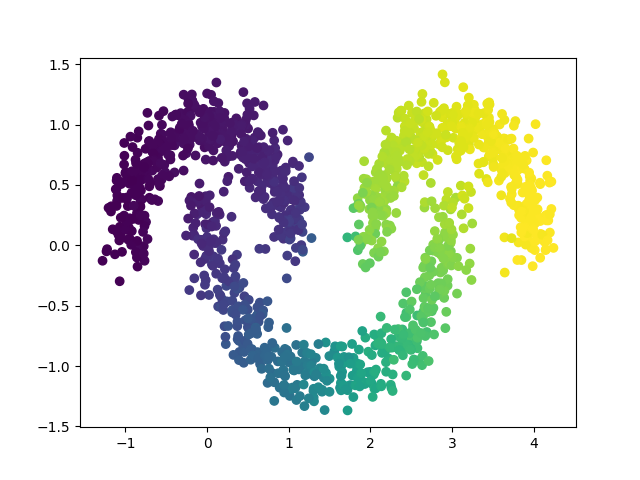}
        \caption{Second eigenvector of the graph Laplacian.}
        \label{fig:image3}
    \end{subfigure}
    \begin{subfigure}[b]{0.24\textwidth}
        \includegraphics[width=\textwidth]{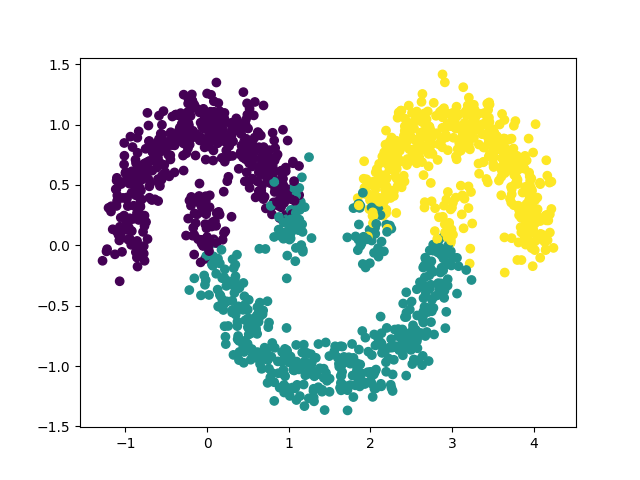}
        \caption{Clustering resulting from the eigenvector.}
        \label{fig:image4}
    \end{subfigure}
    \caption{Toy examples for the spectral clustering achieved as limit of the volume constrained MBO scheme.}
    \label{fig:four_images}
\end{figure}

\section{Main Results}

\subsection{The New Algorithm}\label{sec:new_alg}
The main idea behind the new algorithm is a geometric interpretation of the linear integer problem \eqref{alg:volumeMBO}. We claim that from a geometric viewpoint the solution is a vectorial version of the well known order statistic.

For a given finite set $U \subset \R$ the $k$-th order statistic $m$ is the point such that there are exactly $k$ of the numbers smaller than $m$, i.e. $\#\{u \in U| u < m\} = k$. Indeed our claim holds in the case of only two clusters $P=2$: Denote by $X$ the set of datapoints and by $\mathbbold{1} = (1,1)^\intercal$ the vectors consisting of ones.  The vectors $\{u(x)\}_{x \in X} \subset \R^2$ appearing in the volume constrained MBO scheme \eqref{alg:volumeMBO} lie on the line $u \cdot \mathbbold{1} = 1$ such that the optimal solution is induced by separating the line at the $V_1$-th point $m \in \R^2$ (see Fig. \ref{fig:2median}). Thus, the order statistic to $U = \{u_1(x)\}_{x \in X}$ is exactly $m_1$. The point $m$ agrees with the price vector used in \cite{jacobs2018auction}.

From an algorithmic point of view one should not compute the order statistic by sorting the points along the line and taking the $V_1$ smallest or biggest. But one should rather think of the dividing point as the median along the line for the equisized case $V_1 = V_2$ or in general of the $V_1$-th order statistic. The median or order statistic can be computed in $\mathcal{O}(N)$ while sorting takes $\mathcal{O}(N \log N)$ as for example shown in Theorem 17.3 in \cite{korte}.

It is still possible to give the price vector a geometric interpretation as a multivalued order statistic for $P>2$. This interpretation is visualized in the cases $P=3$ and $P=4$ in Fig. \ref{fig:3median}, Fig. \ref{fig:4median} and we explain in the following where it comes from.

\begin{figure}[!htb]
\begin{subfigure}{0.32\textwidth}
  \includegraphics[width=\linewidth]{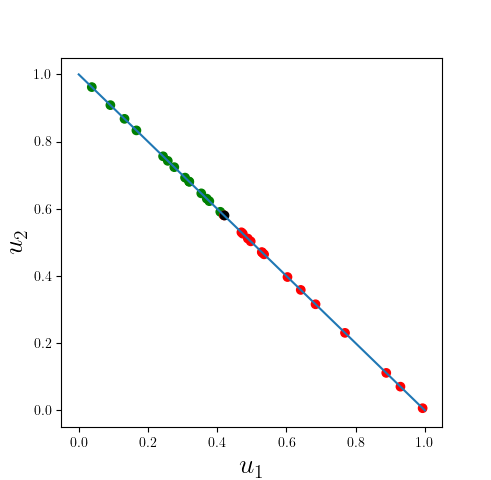}
  \caption{}\label{fig:2median}
\end{subfigure}\hfill
\begin{subfigure}{0.32\textwidth}
  \includegraphics[width=\linewidth]{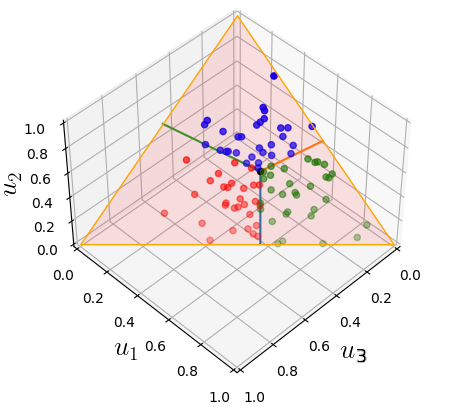}
    \caption{}\label{fig:3median}

\end{subfigure}\hfill
\begin{subfigure}{0.32\textwidth}%
  \includegraphics[width=\linewidth]{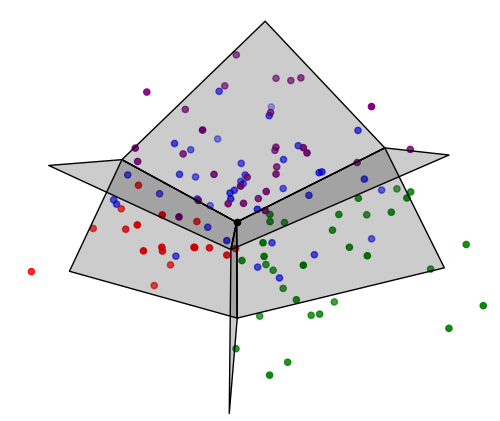}
    \caption{}\label{fig:4median}

\end{subfigure}
\caption{Order statistic (black point) for two (a), three (b) and four (c) clusters and the induced clustering into red, green, blue and purple points.}
\end{figure}

The geometric interpretation is closely connected to the equilibrium price vector. For a given $m \in \R^P$, we call the maximizers $\chi^m$ in the not volume constrained problem
\begin{align}\label{alg:m_induced_ILP}
\chi^\ell \in \argmax_{\chi:X_N \rightarrow \{0,1\}^P} &\sum_{i=1}^P \sum_{x \in X_N} \chi_i(x) (u^{\ell-1}_i(x) - m_i)\\
\text{s.t. } &\sum_{i=1}^P \chi_i(x) = 1 \quad \hspace{3pt} \quad \text{ for all } x \in X_N, \nonumber
\end{align}
an $m$-induced clustering. They are characterized by the following property: For every $x \in X_N$ there exists a phase 
\begin{align}\label{eq:price_vector}
i^*(x) \in \argmax_{1 \leq i \leq N} u_i(x)- m_i
\end{align}
such that with the Kronecker delta $\delta_{i,j}$ it holds
\begin{align}\label{eq:m_induced_cluster}
\chi^m_j(x) = \delta_{j,i^*(x)} \quad \text{ for all } j \in \{1, \dots, P\}.
\end{align}
In \cite{jacobs2018auction} an equilibrium price vector is then a vector $m^* \in \mathbb{R}^P$ such that there exists an $m^*$-induced clustering $\chi^{m^*}$ that is feasible in $\eqref{alg:volumeMBO}$, i.e. the volume constraints are satisfied by $\chi^{m^*}$. We give $m^*$ the new name $V$-order statistic that emphasizes its geometric properties; see Definition~\ref{def:median} for more information. The feasible solution $\chi^{m^*}$ can be shown to be optimal, see Corollary~\ref{cor:opt_order_statistic}. For $m \in \R^P$ denote by $\alpha$ the angle between the $e_i -e_j$ and $u(x) - m$. Reformulation \eqref{eq:price_vector} yields the geometric statement
\begin{align*}
\cos(\alpha) = \frac{(u(x) - m)\cdot (e_i - e_j)}{|u(x) - m||e_i - e_j|} \geq 0 \quad \text{ for all } j\neq i
\end{align*}
meaning that there exists a hyperplane $H_{ij}(m)$ going through $m$ with normal $(e_i - e_j)/\sqrt{2}$ that separates the points assigned to Cluster $i$ from the points assigned to Cluster $j$ for all $j \neq i$. Those hyperplanes are visualized in Fig.~\ref{fig:3median} and Fig.~\ref{fig:4median}.

To find the $V$-order statistic $m^*$ (or equibilibrium price vector) we propose Algorithm \ref{alg:median} which is designed to monotonously decrease the total error in the volume constraints of the $m$-induced clustering $\chi^m$ from \eqref{eq:m_induced_cluster}
\begin{equation}\label{eq:alg_energy}
E(m) := \sum_{i=1}^P \left|\sum_{x \in X_N} \chi^m_i(x) - V_i \right|.
\end{equation}
Note that $m$ is a $V$-order statistic if and only if $E(m) = 0$.

The monotonicity is ensured by carefully translating $m$. Imagine a cluster $i$ with volume less than $V_i$. Then, an algorithmically useful direction to translate $m$ to is 
\begin{equation}\label{eq:first_direction}
d_i := \frac{1}{P-1}\mathbbold{1} - \left(1 + \frac{1}{P-1}\right) e_i = \left(\frac{1}{P-1},\dots, \frac{1}{P-1},-1,\frac{1}{P-1}, \dots, \frac{1}{P-1}\right).
\end{equation} 
This direction is chosen as it increases the volume of Phase $i$ by moving the hyperplanes $H_{ij}(m)$ for all $j\neq i$ while the hyperplanes $H_{kj}(m)$ don't move for all $k,j\neq i$. Thus one reaches a point $u(x_{ij})$ that is on the hyperplane $H_{ij}(m)$ and belongs to Phase $j$, i.e. $\chi^m(x_{ij}) = e_j$. Assuming that the volume of Phase $j$ is greater than $V_j$ we can change the point $x_{ij}$ from Phase~$j$ to Phase $i$. This reduces $E(m)$ by 2.

Of course one can not always assume that Phase $j$ has volume larger than $V_j$. But one can iteratively adapt the directions such that in the end we find a path of $k$ points $x_{i_1, i_2}, x_{i_2, i_3}, \dots, x_{i_{k-1}, i_k}$ with the following properties: The image $u(x_{i_j, i_{j+1}})$ of such a point lays on the hyperplane $H_{i_j, i_{j+1}}(m)$ but belongs to Phase~$j+1$. One can also ensure that the first Phase $i_1$ does not have enough volume, the last Phase $i_k$ has too much volume and all intermediate phases $i_j$ have the correct volume for $j\in \{2, \dots, k-1\}$. In this way one can change the assignment of the points from Phase $j+1$ to Phase $j$ such that $i_1$ increases in volume, $i_k$ decreases in volume and all other phases keep their volume, yielding the desired monotonicity. For further details see Chapter \ref{sec:inner_workings}.

\begin{algorithm}
\begin{algorithmic}[1]
\Require Initial median guess $m$, number of clusters $P$ and points $U \subset \R^P$
\Ensure Computes the $V$-order statistic $m$ and the $m$-induced clustering $\chi^m$ solving \eqref{alg:volumeMBO}.
\State Find an $m$ induced clustering $\chi^m$ \label{alg:line1}\Comment{$\chi^m$ as in Definition \ref{def:median}}
\While {$\mathcal{I}_{:)} \neq \{1, \dots , P\}$} \label{alg:line2}\Comment{$\mathcal{I}_{:)}$ as in \eqref{eq:index_sets3}}
\State Choose a cluster $i^* \in \mathcal{I}_-$ \label{alg:line3}\Comment{$\mathcal{I}_{-}$ as in \eqref{eq:index_sets2}}
\State $\mathcal{T} \leftarrow \{i^*\}$  \label{alg:line4}\Comment{$\mathcal{T}$ is the tree of growing phases}
\State $d \leftarrow \frac{1}{P-1}\cdot\mathbbold{1}- (1+\frac{1}{P-1})e_{i^*}$ \label{alg:line5}\Comment{For motivation of this direction see \eqref{eq:first_direction}}
\While {$i^* \notin \mathcal{I}_+$} \label{alg:line6}\Comment{$\mathcal{I}_{+}$ as in \eqref{eq:index_sets}}
\State move $m$ in direction $d$ until $\exists  i \in \mathcal{T}, j \notin \mathcal{T}$ s.t. the hyperplane $H_{i j}(m)$ contains a \newline \makebox[50pt] \ point $u_{ij}$ of the cluster $j$. \label{alg:line7}
\State $p_{j} \leftarrow i$  \label{alg:line8}\Comment{$p_j$ is the predecessor of j}
\State $i^* \leftarrow j$ \label{alg:line9}
\State $\mathcal{T} \leftarrow \mathcal{T} \cup \{j\}$ \label{alg:line10}
\State $d \leftarrow d+ \frac{1}{P-1}\cdot\mathbbold{1}- (1+\frac{1}{P-1})e_{j}$\label{alg:line11}
\EndWhile \label{alg:line12}
\While {$i^* \notin \mathcal{I}_-$} \label{alg:line13}\Comment{$\mathcal{I}_{-}$ as in \eqref{eq:index_sets2}}
\State Change the clustering $\chi^m$ of $u_{p_{i^*}i^*}$ from $i^*$ to $p_{i^*}$\label{alg:line14} \Comment{i.e. set $\chi^m_j(u_{p_{i^*}i^*}) = \delta_{j p_{i^*}}$}
\State $i^* \leftarrow p_{i^*}$
\EndWhile\label{alg:line16}
\EndWhile
\end{algorithmic}
\caption{Algorithm to find a $V$-order statistic}
\label{alg:median}
\end{algorithm}

\subsection{Extension to Inequality Volume Constraints}
Additionally, we adapt Algorithm \ref{alg:median} to the case of inequality volume constraints \eqref{alg:inequalityMBO} which was originally proposed by Jacobs, Merkurjev and Esedo\={g}lu \cite{jacobs2018auction}. An optimal solution to the problem with inequality constraints is also an optimal solution to equality constraints given by the volumes of the optimal solution. Thus the solution is again induced by a vectorial order statistic denoted by \textit{$(L,U)$-order statistic}.

The problem of finding the $(L,U)$-order statistic is more complex as the feasible set now contains solutions with different volumes. So it is necessary to drive the dynamics by the objective function 
\begin{equation}\label{eq:objective_value}
\sum_{i=1}^P \sum_{x \in X_N} \chi_i(x) u_i(x)
\end{equation}
instead of solely by the energy $E(m)$ given in \eqref{eq:alg_energy}.

To monotonously increase the objective function \eqref{eq:objective_value} in our algorithm we only change the phase from Phase $i$ to Phase $j$ of a point $u(x)$ lying on a hyperplane $H_{ij}(m)$ if $m_i < m_j$. Here again, $m$ denotes the candidate for the $\{L_i, U_i\}$-order statistic that is translated during the algorithm. This guarantees an increase in the objective function by the following calculation:

As $u(x)$ is on the hyperplane $H_{ij}(m)$ there is a vector $v$ such that $u(x) = m + v$ and $v \cdot (e_j - e_i) = 0$. Thus changing the assignment of $x$ from $i$ to $j$ increases the objective function by 
\begin{equation}\label{eq:change_in_m}
u_j(x) - u_i(x) = (m + v) \cdot (e_j - e_i)= m_j - m_i > 0.
\end{equation}
A similar statement (see Lemma \ref{lem:mononotonicity}) holds true when the change is made along a path as described previously in case of Algorithm \ref{alg:median}.

When now adapting Algorithm \ref{alg:median} to only make such monotonous changes one naturally arrives at the following stopping and optimality criterion. It states that the biggest coordinate of the phases that can still take more points is smaller than the smallest coordinate of the phases that can spare points.
\begin{theorem}\label{the:opt_criterium}
Assume $m_{i_1} \geq \dots \geq m_{i_P}$ and denote by $\chi^m$ the by $m$ induced clustering according to \eqref{eq:m_induced_cluster}. Then $\chi^m$ is optimal for \eqref{alg:inequalityMBO} if there exist $b, w \in \{1, \dots, P\}$ with $m_{i_b} \leq m_{i_w}$ such that
\begin{align*}
\sum_{x \in X_N} \chi^m_{i_k}(x) &= U_{i_k} \quad \quad \text{ for all } k \in \{1,\dots,P\} \text{ with } m_{i_k} > m_{i_b},\\
\sum_{x \in X_N} \chi^m_{i_k}(x) &= L_{i_k} \quad \quad \text{ for all }  k \in \{1,\dots,P\} \text{ with } m_{i_k} < m_{i_w}.
\end{align*}
\end{theorem}

\begin{algorithm}
\begin{algorithmic}[1]
\Require initial order statistic guess $m$, number of clusters $P$ and points $\{u(x)\}_{x \in X_N}$
\Ensure Computes the $(L,U)$-order statistic $m$ and the $m$-induced clustering $\chi^m$ solving \eqref{alg:inequalityMBO}.
\State Find a feasible $m$ induced clustering $\chi^m$ with Algorithm \ref{alg:median}\Comment{$\chi^m$ is feasible in \eqref{alg:inequalityMBO}} \label{alg2:line1}
\While {$\exists b \in \mathcal{J}_-, w \in \mathcal{J}_+$ s.t. $m_b > m_w$} \Comment{$\mathcal{J}_-$ and  $\mathcal{J}_+$ as in \eqref{eq:def_jplus}}\label{alg2:line2}
\State $\mathcal{T} := \argmax_{p \in \mathcal{J}_-} m_p$\label{alg2:line3} \Comment{$\mathcal{T}$ is the tree of growing phases}
\State $d_\mathcal{T} := \frac{\# \mathcal{T}}{P-1}\cdot\mathbbold{1}- (1+\frac{1}{P-1})\sum_{i \in \mathcal{T}} e_i$ \label{alg2:line4}\Comment{For motivation of this direction see \eqref{eq:first_direction}}
\While {$\forall w \in \mathcal{T}\cap \mathcal{J}_+: m_w > \max_{p \in \mathcal{J}_-} m_p$}\label{alg2:line5}
\State move $m$ in direction $d_\mathcal{T}$ until $\exists  i \in \mathcal{T}, j \notin \mathcal{T}$ s.t. either
\begin{align*}
\begin{cases}
\text{a) the hyperplane } H_{i j}(m) \text{ contains a point } u_{ij} \text{ of the cluster } j,\\
\text{b) } j \in \argmax_{p \in \mathcal{J}_-} m_p.
\end{cases}
\end{align*}\label{alg2:line6}
\State In case a) set $p_j = i$. \Comment{$p_j$ is the predecessor of j}\label{alg2:line7}
\State $\mathcal{T} = \mathcal{T} \cup \{j\}$\label{alg2:line8}
\State $d_\mathcal{T} := \frac{\# \mathcal{T}}{P-1}\mathbbold{1}- (1+\frac{1}{P-1})\sum_{i \in \mathcal{T}} e_i$\label{alg2:line9}
\EndWhile\label{alg2:line10}
\State $i = \argmin_{p \in  \mathcal{T} \cap \mathcal{J}_+} m_p$\label{alg2:line11}
\While {$i \notin \argmax_{p \in \mathcal{J}_-} m_p$}\label{alg2:line12}
\State Change the clustering $\chi^m$ of $u_{p_{i}i}$ from $i$ to $p_{i}$\label{alg2:line13}
\State $i = p_{i}$\label{alg2:line14}
\EndWhile\label{alg2:line15}
\EndWhile
\end{algorithmic}
\caption{Algorithm to find a $(L,U)$-order statistic}
\label{alg:lower_upper}
\end{algorithm}

\subsection{Efficiency Analysis}\label{sec:main_efficiency}
For the running time analysis one observes that in every iteration (lines 2 to 17) of Algorithm \ref{alg:median} the energy \eqref{eq:alg_energy} reduces by two. Thereby the asymptotic running time is given by $\frac{E(m)}{2} \cdot \mathcal{O}((\log(N) +P)P^2)$ where $m$ is the initial guess for the order statistic. A priori one only has the worst case estimate $E(m) \leq N$. Our goal is to find a good initial guess that leads to a better estimate of $E(m)$.

So what is a good initial guess for $m$? This depends on the distribution of the data $\{u(x)\}_{x \in X_N}$. In our case the data is given by the MBO scheme, or more precisely by the convolution of the heat kernel $p(h,x,y)$ with the characteristic function $\chi$. At least heuristically, the heat kernel only changes the values on a scale $\mathcal{O}(\sqrt{h})$ around the interface between the phases given by $\chi$. Thus we expect that also the order statistic moves according to this length scale $\mathcal{O}(\sqrt{h})$ which we make rigorous by the following strategy. We interpret the order statistic as a Lagrange multiplier by observing that property \eqref{eq:m_induced_cluster} is equivalent to
\begin{align*}
\chi^\ell \in \argmin_{\chi \in \{0,1\}^P} &\sum_{x \in X_N} \sum_{i=1}^P \chi_i(x) (u_i(x) + m^*_i) \quad \text{ s.t. } \sum_{i = 1}^P\ \chi_i(x) = 1.
\end{align*}
By using $u = e^{-h \Delta_N}\chi^{\ell -1}$ and completing the square as in \cite{esedog2015threshold, laux2017convergence, jacobs2018auction} (see Lemma \ref{lem:lagrangeMulti}) it follows that this is equivalent to
\begin{align*}
\chi^\ell \in \argmin_{\chi} E_{h,N}(\chi) + \frac{1}{2h} d_{h,N}^2(\chi,\chi^{\ell-1}) - \frac{2}{\sqrt{h}} m^* \cdot \frac{1}{N} \sum_{x \in X_N} \chi(x),
\end{align*} 
where $E_{h,N}$ and $d_{h,N}$ are the thresholding energy \eqref{eq:discrete_thresholding_energy} and distance \eqref{eq:discrete_thresholding_distance}. Thus $2m^*/\sqrt{h}$ is the Lagrange multiplier to the minimizing movement interpretation of the MBO scheme. It is hard to estimate the Lagrange multiplier directly in our discrete setting. We therefore take a different approach: Starting with the paper of Garc{\'i}a Trillos and Slep{\v c}ev \cite{MR3458162} a lot of big data limits to continuous models have been proven recently.  We prove the convergence of the order statistic in this big data limit which yields together with the convergence of the MBO scheme by Lelmi and one of the authors \cite{laux2021large} the convergence of our volume constrained MBO scheme (see Proposition \ref{the:conv_os}). The clou is that one can use variational techniques that are only available in the continuous setting to prove good bounds for the limiting Lagrange multiplier. To be more precise we improve an $L^2$-estimate of the continuous Lagrange multiplier of Swartz and one of the authors \cite{laux2017convergence} to the multi-phase and spatially inhomogeneous setting (see Proposition \ref{the:l2estimate}). As we are anyways interested in the running time for big $N$ we thus can conclude from the convergence and the bound in the continuum that the Lagrange multiplier is also bounded in the discrete setting for big $N$.

To make this idea rigorous some assumptions on the distribution of the data $X_N$ are necessary. The well established manifold assumption states that $X_N$ are i.i.d. samples over a low-dimensional Riemannian manifold. This low-dimensionality is justified as data points might live in a high dimensional space but have within a cluster only few degrees of freedom such that they can actually be seen to be in a lower dimensional space. For example an image of a handwritten digit $1$ of the MNIST dataset is in $\R^{28\times 28}$ but can be roughly seen as a line with only the two parameters of length and angle. One also assumes that similarity between datapoints is expressed by the extrinsic Euclidean distance between points relative to some length scale $\eps > 0$. The exact assumptions are given in Assumption \ref{ass:manifold}. We can then answer the question of the choice of the initial guess $m^0$ by stating a running time improvement of the factor $\sqrt{h}$ if the initial guess is chosen as result of the previous iteration or as multiple of the vector $\mathbbold{1}$ with only $1$s as entries:

\begin{theorem*}[Informal version of Theorem \ref{the:improvedRunning}]
Under the manifold assumption there is a scaling regime
\begin{align*}
h \ll 1, \quad N \gg 1, \quad \eps \ll 1,
\end{align*}
such that in a typical iteration of the volume constrained MBO scheme, the computational complexity of Algorithm \ref{alg:median} is
\begin{align*}
O(\sqrt{h}\, N \log N)
\end{align*}
when starting from the center $m^0 = \frac{1}{P} \mathbbold{1}$ or the previous order statistic $m^0 = m_N^{\ell - 1}$.
\end{theorem*}

For the above theorem also an assumption on the well preparedness of the initial clusterings is necessary. Out assumption is minimal and states that the initial cuts converge, in the correct topology, to a continuous initial clustering, with finite cut area as the number of data points goes to infinity. This assumption is for example satisfied for the standard initialization with $k$-means or Voronoi cells.

\subsection{Empirical Tests and Comparison to State-of-the-Art}
In practice, the computation of the heat semigroup is expensive and thus not feasible for big datasets. We discuss and compare several possibilities to approximate the heat semigroup on graphs which leads to kernels that are cheaper to compute. One kernel will be motivated by Lelmi and one of the authors \cite{laux2023large} where they show that one only needs the first $\log^q N$ many summands for some scalar $q > 0$ in the spectral decomposition of the heat kernel \eqref{eq:spectral_decomposition} to get convergence of the MBO scheme to Mean Curvature Flow. The other type of approximations are connected to Taylor expansions of the exponential function. They have the advantage of being sparse which leads to fast computation.  One can even improve this with a neat little trick by computing for an arbitrary approximation $A$ of the heat semigroup the product with $\chi^\ell$ by
\begin{align*}
A \chi^\ell = A (\chi^\ell - \chi^{\ell - 1}) + A \chi^{\ell -1 }.
\end{align*}
This is of interest as $A \chi^{\ell -1 }$ can be stored from the previous iteration and $A (\chi^\ell - \chi^{\ell - 1})$ can be computed more efficiently as $\chi^\ell - \chi^{\ell -1}$ is sparse by our Lemma \ref{the:l1estimate}.

Our results on three common datasets for semi supervised learning show that there is not the one approximation that is overall preferable but they all have there pros and cons on the different datasets. Our results support the choice of \cite{jacobs2018auction} where they used the squared weight matrix as approximation. We also compare briefly the results to other graph learning methods which shows that our method is comparable in accuracy while being computationally very efficient.

\medskip

The rest of the paper is now structured as follows. Section \ref{sec:inner_workings} proves the correctness and running time of Algorithm \ref{alg:median} and Algorithm \ref{alg:lower_upper} which is done by showing that the algorithms find the respective order statistic. Section \ref{sec:improved} improves then the running time analysis of Algorithm \ref{alg:median} by making connections to the underlying gradient flow structure of the limiting geometric evolution equation. Section \ref{sec:empiric} introduces several approximations of the heat kernel and compares them empirically on three standard datasets.

\section{The Inner Workings of the New Algorithm}\label{sec:inner_workings}
The aim of this chapter is to prove the correctness and a first estimate on the running time of Algorithm \ref{alg:median} and Algorithm \ref{alg:lower_upper}. To this end, Subsection \ref{sec:exis_opt} introduces rigorously the already in the main results mentioned order statistic and proves Theorem \ref{the:opt_criterium} on the optimality of the by the order statistic induced clustering in the optimization problem \eqref{alg:volumeMBO}. Subsection \ref{sec:fast_alg} explains then how to algorithmically find the order statistic which leads to Theorem \ref{the:correctness} and Theorem \ref{the:runningtime} on the correctness and running time estimate of Algorithm \ref{alg:median}. The results will be even improved in Subsection \ref{sec:alg_inequality} to the more general setting of upper and lower volume constraints, see optimization problem \eqref{alg:inequalityMBO}.  This will be done by adapting Algorithm \ref{alg:median} to greedily optimize the objective function of \eqref{alg:inequalityMBO} which leads to Algorithm \ref{alg:lower_upper}. The proofs of correctness (Theorem \ref{the:correctness_inequality}) and running time estimate (Theorem \ref{the:running_time_inequality}) of Algorithm \ref{alg:lower_upper} are surprising because they show that this greedy procedure finds a global maximizer (which is a priori not clear for a greedy algorithm) and improve the state of the art worst case running time $\mathcal{O}(N^2)$ of \cite{jacobs2018auction} to $\mathcal{O}(N \log (N))$. 

\subsection{Properties of the $V$- and $(L, U)$-Order Statistic} \label{sec:exis_opt}
First we denote by 
\begin{equation}
\mathcal{C} := \left\{\chi: X_N \rightarrow \{0,1\}^P \left| \sum_{i = 1}^P \chi_i(x) = 1~ \text{ for all } x \in X_N \right. \right\}
\end{equation}
the set of clusters. Given $V = \{V_i\}_{i=1}^P$ and $L = \{L_i\}_{i = 1}^P, U = \{U_i\}_{i = 1}^P$ the set of admissible clusters $\mathcal{A}(V)$ and $\mathcal{A}(L, U)$ in the optimization problems \eqref{alg:volumeMBO} and \eqref{alg:inequalityMBO}, respectively, are given by 
\begin{align*}
\mathcal{A}(V) &:= \left\{\chi \in \mathcal{C} \left| \sum_{x \in X_N} \chi_i(x) = V_i~ \text{ for all } i \in \{1, \dots, P\}\right. \right\},\\
\mathcal{A}(L,U) &:= \left\{\chi \in \mathcal{C} \left| L_i \leq \sum_{x \in X_N} \chi_i(x) \leq U_i~ \text{ for all } i \in \{1, \dots, P\}\right. \right\}.
\end{align*} 
Now we define $m$-induced clusterings as well as the $V$- and $(L, U)$-order statistic:

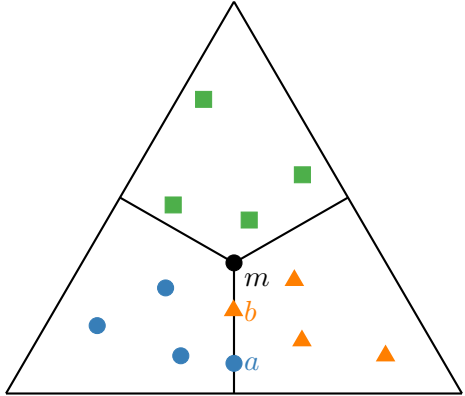
\begin{figure}

\begin{tikzpicture}
\draw[black, thick] (-3,0) -- (3,0);
\draw[black, thick] (3,0) -- (0,5.1961);
\draw[black, thick] (-3,0) -- (0,5.1961);
\filldraw[black] (0,1.7320) circle (3pt) node[anchor=north west]{$m$};
\draw[black, thick] (0,0) -- (0,1.7320);
\draw[black, thick] (-1.5,2.5980) -- (0,1.7320);
\draw[black, thick] (1.5,2.5980) -- (0,1.7320);

\filldraw[phase1] (0,0.4) circle (3pt) node[anchor=west]{$a$};
\node[fill=phase2,regular polygon, regular polygon sides=3,inner sep=1.5pt, anchor=west] at (-0.1,1.1) {};
\filldraw[phase2,anchor=west] (0,1.1) circle (0pt) node[anchor=west]{$b$};

\filldraw[phase1] (-0.9,1.4) circle (3pt) ;
\filldraw[phase1] (-1.8,0.9) circle (3pt) ;
\filldraw[phase1] (-0.7,0.5) circle (3pt) ;

\node[fill=phase2,regular polygon, regular polygon sides=3,inner sep=1.5pt, anchor=west] at (0.7,1.5) {};
\node[fill=phase2,regular polygon, regular polygon sides=3,inner sep=1.5pt, anchor=west] at (1.9,0.5)  {};
\node[fill=phase2,regular polygon, regular polygon sides=3,inner sep=1.5pt, anchor=west] at (0.8,0.7) {};

\filldraw[phase3] ([xshift=-3pt,yshift=-3pt]0.9,2.9) rectangle ++(6pt,6pt);
\filldraw[phase3] ([xshift=-3pt,yshift=-3pt]0.2,2.3) rectangle ++(6pt,6pt);
\filldraw[phase3] ([xshift=-3pt,yshift=-3pt]-0.4,3.9)  rectangle ++(6pt,6pt);
\filldraw[phase3] ([xshift=-3pt,yshift=-3pt]-0.8,2.5) rectangle ++(6pt,6pt);

\end{tikzpicture}
\caption{$\{4,4,4\}$-order statistic $m$ in \textbf{black}. The colors represent a $m$-induced clustering $\chi^m$. The points \color{phase1} $a$ \color{black} and \color{phase2} $b$ \color{black} on the hyperplane $H_{\;\color{phase1}{\leavevmode\put(0,2.4){\circle*{4.5}}}\, \color{phase2} \blacktriangle\color{black}}\!(m)$ could be either assigned to the \quad \color{phase1} \leavevmode\put(-7,3.4){\circle*{6.5}} \color{black} or \color{phase2}$\blacktriangle$ \color{black} phase.}
\end{figure}

\begin{definition}\label{def:median}
Given a point $m \in \mathbb{R}^P$, a clustering $\chi^m \in \mathcal{C}$ of $\{u(x)\}_{x \in X_N}$ is \textit{induced} by $m$ if for all $x \in X_N$ and $i = 1,\dots, P$
\begin{equation}
\chi_i(x) = 1 \Rightarrow (u(x_k)-m)\cdot(e_i-e_j) \geq 0 \quad \text{ for all } j \neq i.
\label{eq:sep_cond}
\end{equation}
Given volume constraints $V_i \in \N,\ i = 1, \dots, P$, we say $m$ is a \textit{$V$-order statistic} if there exists a by $m$ induced clustering $\chi^m$ such that $\chi^m \in \mathcal{A}(V)$.

If volume constraints $L_i, U_i,\ i = 1, \dots, P$, are given, we say $m$ is a \textit{$(L,U)$-order statistic} if there exists a by $m$ induced clustering $\chi^m$ with $\chi^m \in \mathcal{A}(L,U)$ such that when ordering the coordinates $m_{i_1} \geq \dots \geq m_{i_P}$ there exist $b, w \in \{1, \dots, P\}$ with $m_{i_b} \leq m_{i_w}$ and
\begin{align}
\sum_{x \in X_N} \chi^m_{i_k}(x) &= U_{i_k} \quad \quad \text{ for all } k \in \{1,\dots,P\} \text{ with } m_{i_k} > m_{i_b},\\
\sum_{x \in X_N} \chi^m_{i_k}(x) &= L_{i_k} \quad \quad \text{ for all }  k \in \{1,\dots,P\} \text{ with } m_{i_k} < m_{i_w}.
\end{align}
Note that this is precisely the condition of Theorem \ref{the:opt_criterium} and therefore any by an order statistic induced clustering is optimal in \eqref{alg:inequalityMBO}. 
\end{definition}
Furthermore we denote by 
\begin{equation}\label{eq:hyperplanes}
H_{ij}(m) := \left\{u \in \R^P\big|(u - m)\cdot (e_i - e_j) = 0\right\}
\end{equation}
the hyperplanes through an order statistic $m$ which separates the by $m$ induced Clusters $i$ and $j$. The first question to answer is whether there always exists such order statistics. The answer is yes and summarized in the following lemma.
\begin{lemma}
Let volume constraints $\{(L_i, U_i)\}_i$ be given with $\sum_{i=1}^P L_i \leq N \leq \sum_{i = 1}^P U_i$. Then there exists an $(L,U)$-order statistic. In particular the $V$-order statistic exists if $V$ is given with $\sum_{i=1}^P V_i = N$.
\end{lemma}
\begin{proof}
The result follows from Theorem \ref{the:correctness} where we show that Algorithm \ref{alg:lower_upper} finds an $(L,U)$-order statistic. To see that the $V$-order statistic is a special case of the $(L,U)$-order statistic use $b= N$ and $w = 1$ in the definition of the $(L,U)$-order statistic.
\end{proof}
Note that in general neither the order statistics nor the solution to the problems \eqref{alg:volumeMBO} and \eqref{alg:inequalityMBO} are unique. To see the former, one slightly translates (if possible) the order statistic such that the induced clustering doesn't change. For the latter there could be for example cases with two points on the same hyperplane $H_{ij}(m)$ belonging to different clusters. Swapping their belonging to the clusters won't change the objective function \eqref{eq:objective_value}.

Although the order statistic is not unique we can show the optimality of the induced clustering. This was already claimed in Theorem \ref{the:opt_criterium}. We give the proof here:
\begin{proof}[Proof of Theorem \ref{the:opt_criterium}]
Let $m$ and $\chi^m$ be as in the theorem. We want to show that the target value of $\chi^m$ is maximal among the admissible clusterings.  Let $\chi \in \mathcal{A}(L, U)$ be an arbitrary admissible clustering. We claim that 
\begin{align}
\sum_{i = 1}^P \sum_{x \in X_N} \chi_i^m(x)(u_i(x)-m_i) &\geq \sum_{i = 1}^P \sum_{x \in X_N} \chi_i(x)(u_i(x)-m_i), \label{eq:ineq_m_induced}\\
\sum_{i = 1}^P \sum_{x \in X_N} \chi_i^m(x) m_i &\geq \sum_{i = 1}^P \sum_{x \in X_N} \chi_i(x) m_i. \label{eq:ineq_order_statistic}
\end{align}
Then the theorem follows because 
\begin{align*}
\sum_{i = 1}^P \sum_{x \in X_N} \chi_i^m(x)u_i(x) &= \sum_{i = 1}^P \sum_{x \in X_N} \chi_i^m(x)(u_i(x)-m_i) + \sum_{i = 1}^P \sum_{x \in X_N} \chi_i^m(x) m_i \\
 &\geq \sum_{i = 1}^P \sum_{x \in X_N} \chi_i(x)(u_i(x)-m_i) + \sum_{i = 1}^P \sum_{x \in X_N} \chi_i(x) m_i \\
& = \sum_{i = 1}^P \sum_{x \in X_N} \chi_i(x)u_i(x).
\end{align*}
Indeed, the first inequality \eqref{eq:ineq_m_induced} follows by definition of $\chi^m$ (see \eqref{eq:sep_cond}). It remains to prove \eqref{eq:ineq_order_statistic}.
To show \eqref{eq:ineq_order_statistic} we first observe that
\begin{equation*}
\sum_{i = 1}^P \left( \sum_{x \in X_N} \chi_i^m(x) - \sum_{x \in X_N} \chi_i(x) \right) = \sum_{i = 1}^P ( V_i - V_i ) = 0.
\end{equation*}
W.l.o.g. assume that $m_1 \geq \dots \geq m_b \geq \dots \geq m_w \geq \dots \geq m_P$. For all $k \geq w$ we have 
\begin{align*}
\sum_{x \in X_N} \chi_k^m(x) m_k - \sum_{x \in X_N} \chi_k(x) m_k \geq m_w \left(\sum_{x \in X_N} \chi_k^m(x) - \sum_{x \in X_N} \chi_k(x)\right)
\end{align*}
as either $m_k = m_w$ or $m_k < m_w$ and by assumption $\sum_{x \in X_N} \chi_k^m(x) - \sum_{x \in X_N} \chi_k(x) = L_k - \sum_{x \in X_N}\chi_k(x) \leq 0$. Completely analogously we have for all $k < w$ 
\begin{align*}
\sum_{x \in X_N} \chi_k^m(x) m_k - \sum_{x \in X_N} \chi_k(x) m_k \geq m_w \left(\sum_{x \in X_N} \chi_k^m(x) - \sum_{x \in X_N} \chi_k(x) \right).
\end{align*}
All in all this yields
\begin{align*}
&\sum_{i = 1}^P \sum_{x \in X_N} \chi_i^m(x) m_i - \sum_{i = 1}^P \sum_{x \in X_N} \chi_i(x) m_i\\
&\geq m_w \sum_{k \geq w} \left( \sum_{x \in X_N} \chi_{k}^m(x)  - \sum_{x \in X_N} \chi_{k}(x) \right) + m_w \sum_{k < w} \left( \sum_{x \in X_N} \chi_{k}^m(x)  - \sum_{x \in X_N} \chi_{k}(x)\right)\\
&= m_w \sum_{i = 1}^P \left( \sum_{x \in X_N} \chi_i^m(x) - \sum_{x \in X_N} \chi_i(x) \right)\\
&= 0,
\end{align*}
which is precisely the  claimed inequality \eqref{eq:ineq_order_statistic} and finishes the proof.
\end{proof}

The optimality of the $V$-order statistic follows immediately.
\begin{corollary}\label{cor:opt_order_statistic}
Let $m$ be the $V$-order statistic. If $\chi^m$ is a $m$-induced clustering, then $\chi^m$ is optimal in \eqref{alg:volumeMBO}.
\end{corollary}
\begin{proof}
The proof follows from Theorem \ref{the:opt_criterium} by setting $b = N$ and $w = 1$.
\end{proof}

\subsection{A Fast Algorithm for the Volume Constrained MBO Scheme}\label{sec:fast_alg}
This section is dedicated to developing an efficient, exact algorithm to solve \eqref{alg:volumeMBO} by finding the $V$-order statistic. The final algorithm will have similarities to a discretized gradient flow where the respective energy $E(m)$ is given by \eqref{eq:alg_energy}.

Remember we want to minimize $E(m)$ as $E(m) = 0$ if and only if $m$ is $V$-order statistic. Thus solving \eqref{alg:volumeMBO} can be done by minimizing $E$. To get a feeling on how to minimize $E$ we first briefly look at the continuum counterpart of the energy $E$ and for simplicity the Euclidean case with constant density
\begin{equation}
E_{\mathcal{L}^d}(m) := \sum_{i = 1}^P \left|\int_M \chi^m_i(x) \,dx - V_i \right| = \sum_{i = 1}^P \left|u_{\#}\mathcal{L}^d(U_i(m)) - V_i \right|.
\end{equation}
Here $U_i(m) = \{u(x) \in u(M)|(u(x) - m )\cdot (e_i - e_j) \geq 0\ \forall j \neq i \}$ is the set of points in the $i$-th cluster and $u_{\#}\mathcal{L}^d$ denotes the pushforward by $u$ of the $d$-dimensional Lebesgue measure. Assume $u_{\#}\mathcal{L}^d$ is absolutely continuous with respect to the Lebesgue measure and denote by $\rho$ the density of $u_{\#}\mathcal{L}^d$. Now we compute for arbitrary $v \in \R^d$ using Reynold's transport theorem
\begin{equation}
\left. \frac{d}{ds}\right|_{s = 0}u_{\#}\mathcal{L}^d(U_i(m+s\cdot v)) = \left. \frac{d}{ds}\right|_{s = 0} \int_{U_i(m+sv)} \rho \,dx = \sum_{j \neq i} \int_{\partial U_i(m) \cap \partial U_j(m)} v \cdot \frac{e_i - e_j}{\sqrt{2}} \rho \,dx.
\end{equation}
Thereby the gradient of the energy is given by
\begin{equation}
\nabla E_{\mathcal{L}^d}(m) = 2 \sum_{i=1}^P \sign\left(\mathcal{L}^d(U_i(m)) - V_i \right) \sum_{j \neq i} \int_{\partial U_i(m) \cap \partial U_j(m)} \frac{e_i - e_j}{\sqrt{2}} \,dx.
\end{equation}
Thus, if we want to minimize $E_{\mathcal{L}^d}$ we need to move $m$ in the direction of the normals $e_i - e_j$ with $\sign$ given by $\sign(\mathcal{L} (U_i(m)) - V_i)$. For example if the cluster $U_i(m)$ has too little mass then one wants to move $m$ in direction $ \sum_{j \neq i} (e_i - e_j)$. Notice that directions cancel each other out in the sum over $i$ and $j$: If $U_j(m)$ has too little mass as well Cluster $j$ will move $m$ in direction $e_j - e_i = - (e_i - e_j)$ (and others). Thus this direction cancels out $e_i - e_j$ of $U_i(m)$.  This motivates to also move along the directions $e_i - e_j$ in the discrete setting and also to join directions of same signed clusters.

Therefore we denote by 
\begin{align}
\mathcal{I}_+ &:= \Big\{i \in \{1, \dots, P\} \Big| \sum_{x\in X_N} \chi_i(x) > V_i\Big\}, \label{eq:index_sets}\\
\mathcal{I}_- &:= \Big\{i  \in \{1, \dots, P\} \Big| \sum_{x\in X_N} \chi_i(x) < V_i \Big\}, \label{eq:index_sets2}\\
\mathcal{I}_{:)} &:= \Big\{i  \in \{1, \dots, P\} \Big| \sum_{x\in X_N} \chi_i(x) = V_i \Big\} \label{eq:index_sets3}
\end{align}
the indices of the clusters with positive, negative and neutral sign. The idea is now to construct an iterative algorithm.

\begin{figure}
\hspace{-60pt}
\begin{subfigure}[t]{0.15\linewidth}
\includegraphics[height = 1.55\linewidth]{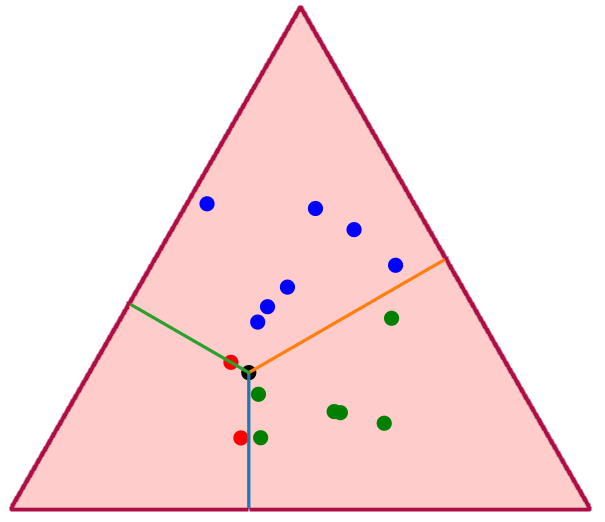}
\caption{Initial guess for the order statistic (black point).}
\label{fig:init_guess}
\end{subfigure}%
\hspace{10pt}
\begin{subfigure}[t]{0.15\linewidth}
\includegraphics[height = 1.55\linewidth]{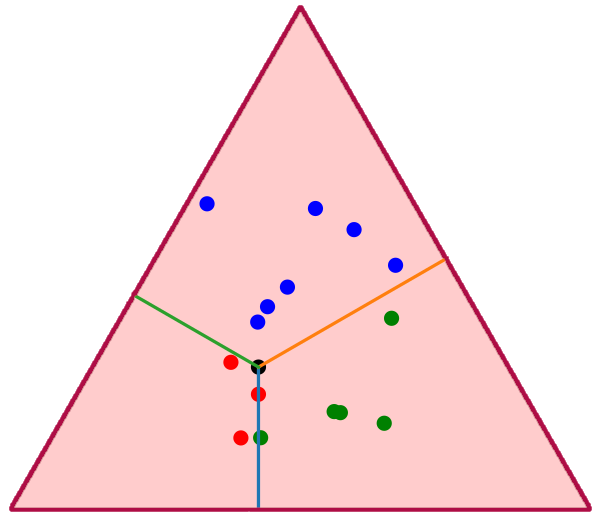}
\caption{Move order statistic to next point and change phase.}
\label{fig:move_and_change}
\end{subfigure}%
\hspace{10pt}
\begin{subfigure}[t]{0.15\linewidth}
\includegraphics[height = 1.55\linewidth]{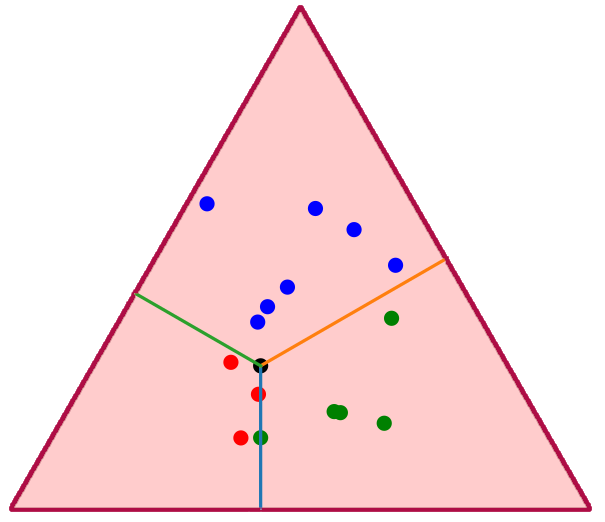}
\caption{Move order statistic until the green point lies on the hyperplane.}
\label{fig:move_and_stay}
\end{subfigure}%
\hspace{10pt}
\begin{subfigure}[t]{0.15\linewidth}
\includegraphics[height = 1.55\linewidth]{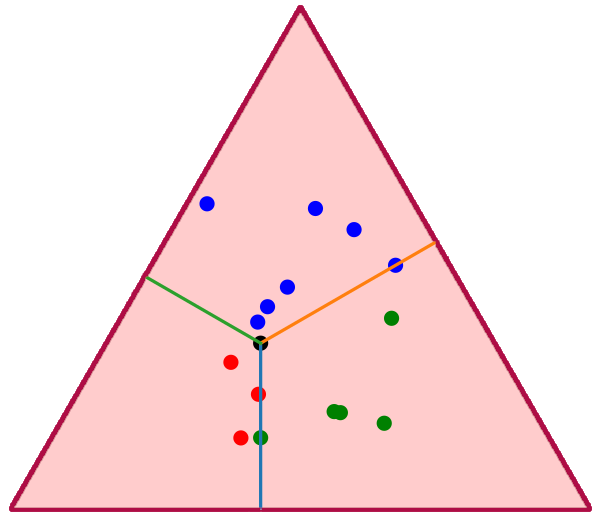}
\caption{The green phase has already the correct volume, thus move in other direction.}
\label{fig:join_phase}
\end{subfigure}%
\hspace{10pt}
\begin{subfigure}[t]{0.15\linewidth}
\includegraphics[height = 1.55\linewidth]{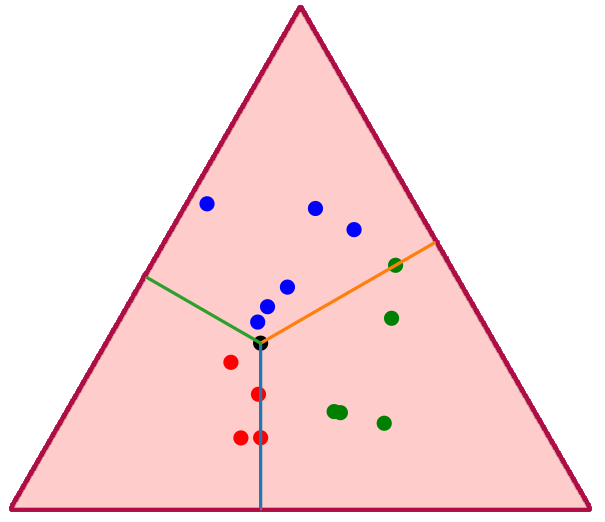}
\caption{Change the phases of the points on the hyperplanes.}
\label{fig:triple_change}
\end{subfigure}%
\caption{Visualization of first steps of Algorithm \ref{alg:median} for $V_1 = V_2 = V_3 = 5$.}
\label{fig:visualization_algo}

\vspace{10pt}
\hspace{-60pt}
\begin{subfigure}[t]{0.15\linewidth}
\begin{tikzpicture}
\filldraw[red] (0,0) circle (3pt);
\end{tikzpicture}
\caption{The tree $\mathcal{T}$ consist only of the root \quad \color{red} \leavevmode\put(-7,3.4){\circle*{6.5}} \color{black} $\!\in \!\mathcal{I}_-$.}
\label{fig:tree1}
\end{subfigure}%
\hspace{10pt}
\begin{subfigure}[t]{0.15\linewidth}
\begin{tikzpicture}
\draw[black, line width = 2pt] (0,0) -- (2,0);
\filldraw[red] (0,0) circle (3pt);
\filldraw[green] (2,0) circle (3pt);
\end{tikzpicture}\caption{Phase \quad \color{green} \leavevmode\put(-7,3.4){\circle*{6.5}} \color{black} is added to $\mathcal{T}$. The edge is added as there is a point on the hyperplane between the two phases.}
\label{fig:tree2}
\end{subfigure}%
\hspace{10pt}
\begin{subfigure}[t]{0.15\linewidth}
\begin{tikzpicture}
\filldraw[red] (0,0) circle (3pt);
\end{tikzpicture}
\caption{As \quad \color{green} \leavevmode\put(-7,3.4){\circle*{6.5}} \color{black} was in $\mathcal{I}_+$ the reassignment was done. Thus, the tree is restarted.}
\label{fig:tree3}
\end{subfigure}%
\hspace{10pt}
\begin{subfigure}[t]{0.15\linewidth}
\begin{tikzpicture}
\draw[black, line width = 2pt] (0,0) -- (2,0);
\filldraw[red] (0,0) circle (3pt);
\filldraw[green] (2,0) circle (3pt);
\end{tikzpicture}\caption{As in Figure \ref{fig:tree2}.}
\label{fig:tree4}
\end{subfigure}%
\hspace{10pt}
\begin{subfigure}[t]{0.15\linewidth}
\begin{tikzpicture}
\draw[black, line width = 2pt] (0,0) -- (2,0);
\draw[black, line width = 2pt] (1,1.81) -- (2,0);
\filldraw[red] (0,0) circle (3pt);
\filldraw[green] (2,0) circle (3pt);
\filldraw[blue] (1,1.81) circle (3pt);
\end{tikzpicture}\caption{Phase \quad \color{blue} \leavevmode\put(-7,3.4){\circle*{6.5}} \color{black} is added and connected to \quad \color{green} \leavevmode\put(-7,3.4){\circle*{6.5}} \color{black}.}
\label{fig:tree5}
\end{subfigure}%
\caption{Visualization of the to Figure \ref{fig:visualization_algo} according tree $\mathcal{T}$ constructed in Algorithm \ref{alg:median}.}
\label{fig:visualization_tree}
\end{figure}

In every step one wants to find a direction such that when $m$ is moved in this direction one point of a cluster with too much mass changes to a cluster with too little mass. So after at most $N/2$ iterations one would be done. Depending on the positions of the points $\{u(x)\}_{x \in X_N}$ such a direction can exist (compare Fig. \ref{fig:init_guess}, \ref{fig:move_and_change}) but does not necessarily have to exist. Imagine for example the three cluster case with one cluster $a \in \mathcal{I}_-$, one $ b \in \mathcal{I}_+$ and the last $c \in \mathcal{I}_{:)}$. It could happen that moving in the directions $e_b - e_a$ or $e_c - e_a$ one always hits a point of the neutral cluster $c$ (see Fig. \ref{fig:move_and_stay}). So after moving $m$ a bit further Cluster $c$ would not be in $\mathcal{I}_{:)}$ any more which means that no progress would have been made.

The solution to this problem is given by joining the directions of $a$ and $c$ when  a point $u(x_{ac})$ lies on the hyperplane between $a$ and $c$, i.e., we want to move in direction $(e_b - e_a) + (e_b - e_c)$. As this direction does not effect the hyperplane separating the clusters $a$ and $c$, we can move in this direction until hitting a point $u(x_b)$ in Cluster $b$ (see Fig. \ref{fig:join_phase}). If $u(x_b)$ lies on the hyperplane between $a$ and $b$ we can immediately put $u(x_b)$ into cluster $a$ decreasing the energy $E$. But if the point lies on the hyperplane between $c$ and $b$ we have to change the cluster of $u(x_{ac})$ to $a$ and the cluster of $u(x_{b})$ to $c$. With this three way swap, $b$ loses the point $u(x_{b})$, $a$ gains the point $u(x_{ac})$ and $c$ stays neutral by gaining $u(x_{b})$ and loosing $u(x_{b})$ (see Fig. \ref{fig:triple_change}). Thus we also decrease $E$ in the second case. 

This motivates the translation of $m$ such that we get a path of points $u(x_{a_1 a_2}),u(x_{a_2 a_3}),$ $ \dots, u(x_{a_{k-1} a_k} )$ lying on the hyperplanes between clusters $a_1, \dots, a_k$. The first point $u(x_{a_1 a_2})$ lies on the hyperplane between a positive cluster $a_1$ and neutral cluster $a_2$. The last point $u(x_{a_{k-1} a_k})$ is between a neutral cluster $a_{k-1}$ and a negative cluster $a_k$. The intermediate points are between neutral clusters. Then we can swap along the path to get an improvement of $E$. 

In Algorithm \ref{alg:median} it is described how to find such a path which is essentially done as above described in the three cluster case. The algorithm iteratively constructs a tree $\mathcal{T}$ with nodes given by phases and edges representing points on hyperplanes between the adjacent phases (see Figure \ref{fig:visualization_tree}). Joining the clusters $i \in \mathcal{T}$ is done by picking the direction
\begin{equation}\label{eq:direction}
d_\mathcal{T} := \frac{\# \mathcal{T}}{P-1}\cdot\mathbbold{1}- \left(1+\frac{1}{P-1}\right)\sum_{i \in \mathcal{T}} e_i.
\end{equation}
Translating the order statistic in this direction doesn't change the hyperplanes $H_{ij}(m)$ for $i,j \in \mathcal{T}$. It does moves the hyperplanes $H_{ik}(m)$ for $i\in \mathcal{T}, k \notin \mathcal{T}$ in the direction increasing the volume of Cluster $i$ (and shrinking Cluster $k$).

One could equally use the direction 
\begin{equation*}
\tilde{d_\mathcal{T}} := \sum_{i \in \mathcal{T}} e_i
\end{equation*}
but we use the scaling in \eqref{eq:direction} such that $\sum_{p \in P} (d_\mathcal{T})_p = 1$. This is motivated by the fact that the points of the MBO scheme with equal surface tensions lie on this plane $\sum_{p \in P} u(x)_p = 1$.
The correctness and a first running time estimate of Algorithm \ref{alg:median} are proven next.
\begin{theorem}[Correctness of the algorithm]\label{the:correctness}
Algorithm \ref{alg:median} finds an $V$-order statistic $m$ as well as a by $m$ induced clustering $\chi^m$ that fulfils the volume constraints and solves optimization problem~ \eqref{alg:volumeMBO}. 
\end{theorem}
\begin{proof}
We first claim that lines \ref{alg:line6} to \ref{alg:line12} find a directed tree $\mathcal{T}$ with the following properties:
\begin{itemize}
\item[i)] There is exactly one node in $\mathcal{I}_+ \cup \mathcal{T}$ and this node is a leaf. 
\item[ii)] There is at least one node in $\mathcal{I}_- \cup \mathcal{T}$ and this node is the root.
\item[iii)] If there is a directed edge $(a,b)$ in $\mathcal{T}$ then there is point $u_{ab}$ in Cluster $a$ lying on the hyperplane $H_{ab}(m)$ between Cluster $a$ and $b$.
\end{itemize}
At the start of the loop $\mathcal{T}$ as defined in line \ref{alg:line4} is our tree.  It has empty edge set and there is one point in $\mathcal{I}_-$ in it. Thus $ii)$ and $iii)$ are valid. In every iteration we find a point $u_{ij}$ such that $i$ is in the tree $\mathcal{T}$ and $j$ is not. We have to check that such a point exists. By iterating over line \ref{alg:line11} we move in direction
\begin{equation}
d_{\mathcal{T}} = \sum_{i \in \mathcal{T}} \frac{1}{P-1}\cdot\mathbbold{1}- \Big(1+\frac{1}{P-1}\Big)e_{i} = \frac{\#\mathcal{T}}{P-1}\cdot\mathbbold{1} - \Big(1+\frac{1}{P-1}\Big)\sum_{i \in \mathcal{T}} e_{i}
\end{equation}
and see that for a pair $i \in \mathcal{T}$ and $j \notin \mathcal{T}$ we move the hyperplane $H_{ij}(m)$ in direction of Cluster $j$ (i.e. make Cluster $j$ smaller and Cluster $i$ bigger) because
\begin{equation}
d_{\mathcal{T}}\cdot (e_j -e_i) = 1+\frac{1}{P-1} > 0.
\end{equation}
As there is at least one non-empty cluster $j \notin \mathcal{T}, j \in \mathcal{I}_+$ there exists a point in direction $d_{\mathcal{T}}$.

Important is that the hyperplanes $H_{i_1 i_2}(m)$ for $i_1, i_2 \in \mathcal{T}$ are not changed since $d_{\mathcal{T}}$ is perpendicular to their normal $e_{i_1} - e_{i_2}$. Thus the clustering $\chi^m$ stays always by $m$ induced because we never shift a hyperplane over a point that is not in the corresponding cluster. $\mathcal{T}$ is updated in line \ref{alg:line10}  such that $j$ now also belongs to $\mathcal{T}$ so $u_{ij}$ will stay in the hyperplane $H_{ij}(m)$ (as the hyperplane will not shift afterwards anymore). Line \ref{alg:line8} describes the directed edge from $j$ to $i$. Thus, $iii)$ stays valid after the update. By line \ref{alg:line6} and \ref{alg:line9} we keep doing this until we find a cluster $i^* \in \mathcal{I}_+$ which will then be a leaf. Thus $i)$ holds which proves the claim. 

By our claim we conclude that in lines \ref{alg:line13} to \ref{alg:line16} we follow the path starting at $i^*$ back to some $i_{end} \in \mathcal{I}_-$. The path is unique as every point is a predecessor of exactly one other point. We can change the clusters of the points along the path and stay $m$ induced as the points lie on their respective hyperplanes and can thus belong to either of the two clusters by Definition \ref{def:median}. The energy $E$ is decreased by $2$ as the number of points in the cluster in the middle of the path stays the same, the number of points in $i^* \in \mathcal{I}_+$ is decreased and the number of points in $j^*\in \mathcal{I}_-$ is increased. So $E$ will be $0$ after at most $N/2$ iterations which means that $m$ is a $V$-order statistic and $\chi$ is a $m$ induced clustering that fulfills the volume constraints.
\end{proof}
\begin{theorem}[First running time estimate]\label{the:runningtime}
Algorithm \ref{alg:median} can be implemented such that every iteration of the outer while loop has running time $\mathcal{O}((P+\log(N))P^2)$. In particular for an initial $V$-order statistic guess $m$ the total running time is 
\begin{align*}
\frac{E(m)}{2} \cdot \mathcal{O}((P+\log(N))P^2) + \mathcal{O}(NP) \subset \mathcal{O}(N (\log(N) + P) P^2).
\end{align*}
\end{theorem}
\begin{proof}
Finding the $m$ induced clustering in line \ref{alg:line1} can be done via a simple calculation by \eqref{eq:price_vector} and \eqref{eq:m_induced_cluster} in $\mathcal{O}(NP)$. Except for line \ref{alg:line7} everything is straightforward to implement. Line \ref{alg:line7} can be efficiently implemented using a priority queue that supports inserting, deleting and getting the smallest priority element in $\mathcal{O}(\log(N))$ time. Note that one has to keep track of indices of every element in the priority queue such that deleting is possible in $\mathcal{O}(\log(N))$. We use for each hyperplane $H_{ij}(m), i\neq j \in \{1, \dots, P\}$ the priority queue to get the closest element $u_{ij}$ to the hyperplane going through $m$ in the direction $d$. Then we only need to take the closest of the $\mathcal{O}(P^2)$ many $u_{ij}$ with $i \in \mathcal{T}, j\notin \mathcal{T}$ to know how much we have to move $m$ in direction $d$. Thus line \ref{alg:line7} can be implemented in $\mathcal{O}(P^2)$. There are at at most $P$ iterations of the while loop from line \ref{alg:line6} to \ref{alg:line12} yielding the running time of $\mathcal{O}(P^3)$  for the while loop from line \ref{alg:line6} to \ref{alg:line12}.

The while loop from line \ref{alg:line13} to \ref{alg:line16} also has at most $P$ iteration. In every iteration one has to remove and add the changed points from the $(P-1)$ priority queues corresponding to the $(P-1)$ hyperplanes of the removed or added phase. The running time of lines \ref{alg:line13} -- \ref{alg:line16} is given by $\mathcal{O}(\log(N) P^2)$ as removing and adding from a priority queue takes $\mathcal{O}(\log(N))$.
The initialization of a priority queue takes linear time in the number of elements in the queue. In our case all points of phase $i$ are in the priority queue corresponding to hyperplane $H_{ij}$. Thus initializing all the priority queues to $H_{ij}$ for fixed $j$ takes $\mathcal{O}(N)$ time. So in total for all $j$ we need $\mathcal{O}(N P)$ time for initialization.
\end{proof}

\subsection{Extension to Inequality Constraints}\label{sec:alg_inequality}
We now aim to adapt Algorithm \ref{alg:median} to lower and upper volume constraints. The driving force behind the adaptation is the following simple lemma. It describes how a step of Algorithm \ref{alg:median} changes the objective function \eqref{eq:objective_value} by using the observation we made in \eqref{eq:change_in_m}.
\begin{lemma}\label{lem:mononotonicity}
Let a path of points $u(x_{1 ,2}),u(x_{2, 3}), \dots, u(x_{k-1, k} )$ be given. Assume the points $u(x_{j-1, j} )$ are located on the hyperplane $H_{j-1,j}(m)$ and belonging to Cluster $j$. Then changing the clusters along this path from $j$ to $j-1$ for all $j = \{2,\dots, k\}$ increases the objective function \eqref{eq:objective_value} by
\begin{equation}
m_{1} - m_{k}.
\end{equation}
\end{lemma}
\begin{proof}
Changing the cluster of point $u(x_{j-1, j} )$ from $j$ to $j-1$ increases the objective function by $u(x_{j-1, j} )_{j-1} -u(x_{j-1, j} )_{j}$. As $u(x_{j-1, j} ) \in H_{j, j-1}(m)$ there is $v \in \R^P$ such that $u(x_{j-1, j}) = m + v$ and $v \cdot (e_{j-1} - e_{j}) =0$. Thus it holds
\begin{equation*}
u(x_{j-1, j} )_{j-1} -u(x_{j-1, j} )_{j} = (m+v)\cdot(e_{j-1} - e_{j}) = m_{j-1} - m_{j}.
\end{equation*}
Summing up the changes of the objective function for $j \in \{2, \dots, k\}$ yields the result
\begin{equation*}
\sum_{j = 2}^k u(x_{j-1,j} )_{j-1} -u(x_{j-1, j} )_{j} = \sum_{j = 2}^k m_{j-1} - m_{j} = m_{1} - m_{k}. \qedhere
\end{equation*}
\end{proof}
By Lemma \ref{lem:mononotonicity} the objective function increases if and only if $m_{1} > m_{k}$. To ensure that this holds true in our adapted Algorithm \ref{alg:lower_upper} we define the to \eqref{eq:index_sets} and \eqref{eq:index_sets2} corresponding index sets
\begin{align}\label{eq:def_jplus}
\mathcal{J}_+ &:= \Big\{i \in \{1, \dots, P\} \Big| \sum_{x\in X_N} \chi_i(x) > L_i\Big\}, \\
\mathcal{J}_- &:= \Big\{i  \in \{1, \dots, P\} \Big| \sum_{x\in X_N} \chi_i(x) < U_i \Big\}.
\end{align}
As long as there is an index $w \in \mathcal{J}_+$ and an index $b \in \mathcal{J}_-$ with $m_w < m_b$ we can use the same procedure as in Algorithm \ref{alg:median}. Start with Phase $b$ and join other phases by translating $m$ until Phase $w$ is reached. Then phases of the points on the hyperplanes can be swapped along the produced path. Lemma \ref{lem:mononotonicity} guarantees a gain in the objective function.

This algorithm can be used until there are no such points $w$ and $b$ anymore. At that point the assumptions of Theorem \ref{the:opt_criterium} are satisfied. Thus $m$ is the $(L,U)$-order statistic.

The only challenge is that the order of the $m_i$ can change during the algorithm. To take care of that one needs to keep the set $\argmin_{i \in \mathcal{J}_-} m_i$ updated to ensure that all improving paths are found. For more details see the following proof that Algorithm~\ref{alg:lower_upper} finds a $(L,U)$-order statistic.
\begin{theorem}\label{the:correctness_inequality}
Let $L_i, U_i$ be given with $\sum_{i = 1}^P L_i \leq N \leq \sum_{i=1}^P U_i$. Then at the end of Algorithm~\ref{alg:lower_upper} $m$ is an $(L,U)$-order statistic and $\chi^m$ is an $m$ induced clustering solving \eqref{alg:inequalityMBO}.
\end{theorem}
\begin{proof}
To prove the theorem we have to ensure that the algorithm finds all changes that improve the objective function \eqref{eq:objective_value} and thus reaches the stopping criteria of line \ref{alg2:line2}. For that line \ref{alg2:line3} and  \ref{alg2:line6}b keep the set $\argmax_{p \in \mathcal{J}_-} m_p$ updated. With the same argumentation as in Theorem \ref{the:correctness} is follows that lines \ref{alg2:line5} to  \ref{alg2:line10} produce a path of points $u(x_{i_1 i_2}),u(x_{i_2 i_3}), \dots, u(x_{i_{k-1} i_k} )$ with $u(x_{i_{j-1} i_j} )$ on the hyperplane $H_{j-1,j}(m)$ and belonging to Cluster $j$. Additionally, $i_1$ is in $\argmax_{p \in \mathcal{J}_-} m_p$ and by line \ref{alg2:line5}, $i_k$ is in $\mathcal{J}_+$ with $m_{i_k} > m_{i_1}$. By Lemma \ref{lem:mononotonicity} changing the phases in lines  \ref{alg2:line12} -- \ref{alg2:line15} improves the objective function.

If $\mathcal{T}= \{1,\dots,P\}$ and the condition of line  \ref{alg2:line5} is still not satisfied then $\min_{i \in \mathcal{J}_+} m_i > \max_{i \in \mathcal{J}_-} m_i$. Thus the condition of line \ref{alg2:line2} will also not be satisfied in this case. Note that 
\begin{align*}
\sum_{i \in \mathcal{J}_-} \left( U_i -\sum_{x \in X_N} \chi_i^m(x) \right) \text{ and }  \sum_{i \in \mathcal{J}_+} \left( \sum_{x \in X_N} \chi_i^m(x)  - L_i \right)
\end{align*}
are decreasing by one in every step (line \ref{alg2:line3} to  \ref{alg2:line15}). Thus either the condition of line  \ref{alg2:line2} will be violated or at some point $\mathcal{I}_- = \emptyset$ or $\mathcal{I}_+ = \emptyset$ thus also violating line \ref{alg2:line2}. In all cases the violation of condition of line \ref{alg2:line2} will be reached after finitely many steps of the algorithm.

It then holds $\max_{i \in \mathcal{J}_-} m_i \leq \min_{i \in \mathcal{J}_+} m_i$. Now choose 
\begin{equation}
i_b \in \argmax_{p \in \mathcal{J}_-} m_p \text{ and } i_w \in \argmin_{p \in \mathcal{J}_+} m_p
\end{equation} with either 
\begin{equation}
m_{i_1} \geq \dots \geq m_{i_w} > \dots > m_{i_b} \geq \dots \geq m_{i_P}
\end{equation}
or 
\begin{equation}
m_{i_1} \geq \dots >m_{i_w}  = \dots = m_{i_b}> \dots \geq m_{i_P}
\end{equation}
as in the definition of the $(L,U)$-order statistic. Then it holds for $k > w$ that $m_{i_k} < m_{i_w} = \min_{p \in \mathcal{J}_+} m_p$. Thus $k$ is not in $\mathcal{J}_+$ for all $k > w$, i.e $\sum_{x \in X_N} \chi^m_k(x) = L_k$. Similarly for all $k < b$ it holds $\sum_{x \in X_N} \chi^m_k(x) = U_k$. As $m_{i_b} \leq m_{i_w}$ all requirements in Definition \ref{def:median} of the $(L,U)$-order statistic are satisfied. By Theorem \ref{the:opt_criterium} it follows that $\chi^m$ is optimal for \eqref{alg:inequalityMBO}.
\end{proof}

The running time analysis of Algorithm \ref{alg:lower_upper} is analogous to the running time analysis of Algorithm \ref{alg:median} in Theorem \ref{the:runningtime}.
\begin{theorem}\label{the:running_time_inequality}
Algorithm \ref{alg:lower_upper} can be implemented such that every iteration of the outer while loop has running time $\mathcal{O}((\log(N) +P) P^2)$. In particular for an initial $(L,U)$-order statistic guess $m$ the total running time is $\frac{E(m|m^*)}{2} \cdot \mathcal{O}((\log(N) +P) P^2) = \mathcal{O}(N(\log(N) +P) P^2)$. Here $m^*$ is the $(L,U)$-order statistic found by Algorithm \ref{alg:lower_upper} and the energy is given by
\begin{equation}
E(m|m^*) := \sum_{i = 1}^P \left|\sum_{x \in X_N} \chi_i^m(x) - \chi_i^{m^*}(x)\right|.
\end{equation}
\end{theorem}

\section{Improved Running Time Analysis}\label{sec:improved}
We first state the exact assumptions under which our results on the efficiency of Algorithm \ref{alg:median} hold.
\begin{assumption}\label{ass:manifold} \ 
\begin{itemize}
     \item The points $X_N := \{x_1, \dots, x_N\} \subset \R^D$ are i.i.d. samples of a \textit{probability measure} $\mu = \rho \Vol_M$, with $M$ a closed, smooth, connected $d$-dimensional Riemannian manifold and $\rho$ a smooth and positive density. In particular it holds $0 < \frac{1}{c} \leq \rho \leq c < \infty$.
     \item The weights between points $x, y \in X_N$ are given by
     \begin{equation*}
     w_\eps(x,y) = \frac{1}{\eps^d}\eta \left(  \frac{|x-y|}{\eps} \right).
     \end{equation*}
     where the kernel $\eta$ satisfies 
         \begin{align}
    &\eta(0) > 0 \text{ and } \eta \text{ is continuous at zero, } \\
    &\eta(t) \geq 0 \text{ for every } t > 0, \eta \text{ is non-increasing, }\\
&\eta \text{ has exponential decay }.    
\end{align}
\item The sequence of positive numbers $\{\eps_N\}_N$ satisfies
\begin{align*}
\lim_{N \rightarrow \infty} \eps_N = 0 \text{ and } \lim_{N \rightarrow \infty} \frac{\eps^{d}_N N}{\log (N)} =\infty.
\end{align*}
\item The sequence of initial clusterings $\chi^0_N: X_N \rightarrow \{0,1\}^P$ converge weakly in $TL^2$ to some $\chi^0:M \rightarrow \{0,1\}^P$, i.e.
\begin{align*}
\chi^0_N \xrightharpoonup[]{TL^2} \chi^0.
\end{align*}
See Definition \ref{def:weak_TL2} for more information on $TL^2$ convergence. Additionally $\chi^0$ is assumed to have finite perimeter, that is 
\begin{align*}
\sum_{i = 1}^P \int_M \rho\; d|\nabla \chi_i^0| < \infty.
\end{align*}
\item The sequence of prescribed volume $\{V_N\}_N \subset \N^P$ with $\sum_{i = 1}^P V_{N,i} = N$ in the volume-constrainted optimization problem \eqref{alg:volumeMBO} converges to some $V \in [0,1]^P$ with $\sum_{i = 1}^P V_i = 1$, i.e.,
\begin{equation}
\frac{V_N}{N} \rightarrow V.
\end{equation}
 \end{itemize}
\end{assumption}
We give some remarks to the assumptions in order
\begin{rem}
\begin{itemize}
\item The assumptions on the distribution of the datapoints in the so called data manifold is often refered to as manifold assumption. It is a reasonable assumption as although the data points are in a very high dimensional space $\R^D$ the data points that are contain in the same class do not have many degrees of freedom to vary. Image for example a (american) handwritten digit $1$ which might be determined only by its length and rotation. 
\item The assumptions on the weights are also standard \cite{MR3458162,laux2021large}. The hidden dimension $d$ is for the analysis import but is not needed for computing the weights. For example when the graph Laplacia $\tilde{\Delta}_N  = I - D^{-1}W$, with diagonal degree matrix $D_{xx} = \sum_{y \in X_N} w(x,y)$ is used, then the $\frac{1}{\eps^d}$ factor cancels. For the Laplacian $\Delta_N  = D - W$ one observes that only the heat semigroup $e^{-h \Delta_N}$ is needed for the computation of the scheme and $\eps^d$ is only a rescaling of time in the heat semigroup as both $D$ and $W$ have the prefactor $\frac{1}{\eps^d}$.
\item The assumptions on $\eps_N$ are optimal. The limit $\eps_N \rightarrow 0$ is necessary to ensure localization in the limit. The condition $\lim_{N \rightarrow \infty} \frac{\eps^{d}_N N}{\log (N)} =\infty$ is the optimal sampling rate; it roughly says that $\log N$ edges are connected to every node.  A bigger $\eps$ would lead to more edges and thus a bigger computational cost. A smaller $\eps$ would yield with high probability a disconnected graph. For more information on the sampling rate see for example \cite{MR3458162}.  
\item The weak convergence in $TL^2$ of the initial conditions is for example given for some $\chi^0 \in C^0(M, \{0,1\}^P)$ and $\chi^0_N := \chi^0\big|_{X_N}$.
\item The condition that $\chi^0$ has finite perimeter is for example satisfied when the $k$-means algorithm or Voronoi cells as in \cite{jacobs2018auction} are used to generate $\chi^0_N$. The assumption ensures that the initial thresholding energies are uniformly bounded for all $h > 0$ and $N \in \N$. To be more precise it holds 
\begin{equation}
\sup_{h > 0} \limsup_{N \rightarrow \infty} E_{h, N}(X_N^0) < \infty.
\end{equation}
To see that this bound holds one first uses for fixed $h$ the continuous convergence $E_{N,h}(X_N^0) \rightarrow E_h(X^0)$ of the tresholding energy \cite{laux2021large} as in Proposition \ref{the:conv_os}. Then by Lemma \ref{lem:monotonicity_independent} and the $\Gamma$-limsup inequality in \cite{laux2021large} one has additional 
\begin{align*}
E_h(\chi^0)\leq C \lim_{\tilde{h} \rightarrow 0}E_{\tilde{h}}(\chi^0) = C \sum_{i = 1}^P \int_M \rho \; d |\nabla \chi_i^0| < \infty.
\end{align*} 
Here one sees the close connection of our approach to mean curvature flow, which is the $L^2$-gradient flow to the weighted Perimeter $\int_M \rho d |\nabla \chi_i^0|$. For the convergence of the volume-constrained MBO scheme to volume-preserving mean curvature flow see for example \cite{laux2017convergence}.
\end{itemize}
\end{rem}

Under this assumptions we will show that we can expect an improved running time of Algorithm \ref{alg:median} in most iterations of the volume-preserving MBO scheme. The precise statement is given in the following Theorem.
\begin{theorem}[Improved efficiency analysis of Algorithm \ref{alg:median}]\label{the:improvedRunning}
Let Assumption \ref{ass:manifold} be in place. Assume we make $L$ iterations of the volume constrained MBO scheme \eqref{alg:volumeMBO}. Denote by $m^\ell_{N}$ the $V$-order statistic in the $\ell$-th step of the MBO scheme with diffusion time $h$. If in every iteration the center or the predecessor as initial guess of Algorithm \ref{alg:median} is used, that is $m^0 \in \{ \frac{1}{P} \mathbbold{1} ,m_N^{\ell -1}\}$, then almost surely there exists $h_0 > 0$ such that for every $h < h_0$ there exists $N_0(h) > 0$ such that for every $N > N_0(h)$ and arbitrary $\delta > 0$ the accumulated running time of Algorithm \ref{alg:median} for $L$ iterations in the volume constrained MBO scheme \eqref{alg:volumeMBO} is of the order
\begin{equation}
\mathcal{O} \left(\left(L \sqrt{h} \frac{P}{1 - 2 \delta P} + \frac{1}{\delta^2}\right) P^2 N (\log N + P) \right).
\end{equation}
In particular for $\delta = \frac{1}{4P}$ we get a running time estimate of
\begin{equation}
\mathcal{O} \left(\left(L \sqrt{h} P +  P^2 \right) P^2 N (\log N + P) \right),
\end{equation}
where we see that the $P^2$-term is independent of $L$. 
\end{theorem}

So indeed in (up to constants) only $P^2$ iterations the running time is not sped up. Note that in practice $P^2$ is a lot smaller than $L$.

The proof of Theorem \ref{the:improvedRunning} can be split into three major components. The first component is Proposition \ref{cor:iterationbound} which states that we get the improved running time for one iteration provided the order statistic guess $m^0$ is close to the actual order statistic. 
\begin{proposition}[Improved running time if close to center]\label{cor:iterationbound}
Assume the $V$-order statistic $m_N^\ell$ is close to $m^0$, i.e,. $|m_N^\ell - \frac{1}{P} \mathbbold{1}|_\infty < \delta$. If we use $m^0 = \frac{1}{P}\mathbbold{1}$ as initial guess of Algorithm \ref{alg:median} or use $m^0 = m_N^{\ell-1}$ with $|m_N^{\ell-1} - \frac{1}{P}\mathbbold{1}|_\infty < \delta$, then the running time of Algorithm \ref{alg:median} to find $m_N^\ell$ is in
\begin{align*}
\mathcal{O}\left(\frac{P}{1-2P \delta} \sqrt{h} N \log (N) + NP\right).
\end{align*}
\end{proposition}
The proof of Proposition \ref{cor:iterationbound} is explained in Subsection \ref{sec:iterationbound}. The idea is to bound the number of points that are close to the hyperplanes \eqref{eq:hyperplanes} that are induced by the center $\frac{1}{P} \mathbbold{1}$. Combined with the assumption of Proposition \ref{cor:iterationbound} that $m_N^\ell$ is $\delta$-close to the center, this will lead to a bound on the number of points that change their clustering during the algorithm.

To apply Proposition \ref{cor:iterationbound} one needs a bound on the distance between the order statistic and the center. Thus, the second component of the proof of Theorem \ref{the:improvedRunning} is an $L^2$-estimate on the distance between the order statistic and the center for the continuous counterpart of the volume constrained MBO scheme.
\begin{proposition}[$L^2$-estimate for Lagrange multiplier]\label{the:l2estimate}
Let Assumption \ref{ass:manifold} be in place. Given the iterates $\chi^\ell$ for $\ell = 1,\dots, L$ of the volume constrained MBO scheme according to Definition \ref{def:cont_median} below with $V$-order statistic $m_\mu^\ell$, for $h \ll \frac{\min_{i = 1,\dots, P} V_i}{P (E_0 + 1)} $, we have 
\begin{align*}
h \sum_{\ell = 1}^L \left|\frac{m_\mu^\ell - \frac{1}{P}\mathbbold{1}}{\sqrt{h}}\right|^2 = \frac{h}{4} \sum_{\ell=1}^L |\Lambda^\ell_h|^2 \lesssim  \frac{ P^6}{ \min_{i = 1,\dots, P} V_i^4} (1+T) (1+ E_0^6).
\end{align*}
\end{proposition}
The continuous counterpart of the volume constrained MBO scheme \eqref{alg:volumeMBO} is carefully introduced in Subsection \ref{sec:convergence}. The proof of Proposition \ref{the:l2estimate} is an adaptation of Swartz and the second author \cite{laux2017convergence} to the multi-cluster case on a Riemannian manifold and can be found in Subsection \ref{sec:l2}. One of the main tools for the proof is the well known minimizing movement interpretation of Lemma \ref{lem:lagrangeMulti} as well as the a priori energy dissipation estimate of Lemma \ref{lem:energyDis}. The Euler--Lagrange equation \eqref{eq:ELequation} that is induced by the minimizing movement interpretation is used to get the $L^2$-estimate but is also only available in the continuous setting. This is one of the several reasons why the presented techniques can not by adapted to immediately get the desired bound in the discrete setting. Therefore, one has to take an alternative route by proving that the discrete volume constrained MBO scheme converges to the continuous volume constrained MBO scheme which will be our third component to prove Theorem \ref{the:improvedRunning}.
\begin{proposition}[Convergence of the volume constrained MBO scheme]\label{the:conv_os}
Assume $\chi_N^0$ is the initial clustering of the discrete MBO scheme \eqref{alg:volumeMBO} with $u^{\ell - 1} = e^{-h \Delta_N} \chi^{\ell -1}$. Denote by $m_N$ the $V_N$-order statistic for given volumes $(V_{N,i})_{i = 1}^P$. Similarly, let $\chi^0$ be the initial clustering of the continuous MBO scheme \eqref{alg:volumeMBOcont} and $m_\mu$ the continuous $V$-order statistic for given volumes $(V_i)_{i = 1}^P$. Then almost surely the order statistics converge, i.e.,
\begin{equation}
m_N \xrightarrow[]{N \rightarrow \infty} m_\mu.
\end{equation}
Furthermore, write $\chi_N^1$ for a clustering after one iteration of the discrete volume constrained MBO scheme. Then it holds almost surely that for every sequence $(N') \subset \N$ there is a subsequence $(N'') \subset (N')$ and a clustering after one iteration of the continuous volume constrained MBO scheme $\chi^1$ with
\begin{equation}
\chi_{N''}^1 \xrightharpoonup[]{TL^2(M, \mu)} \chi^1.
\end{equation}
\end{proposition}
\begin{rem}
The statement for the convergence of the discrete solutions $\chi_N^1$ to the continuous solutions $\chi^1$ was formulated with subsequences as both $\chi_N^1$ and $\chi^1$ as solutions of \eqref{alg:volumeMBO} and \eqref{alg:volumeMBOcont} may not be unique. It could exist a set of positive measure $\{x : u_i(x) - m_i = u_j(x) -m_j\}$ for some $i,j \in \{1, \dots, P\}$ with points that can either be assigned to the $i$-th or $j$-th cluster. Due to this non-uniqueness one can not guarantee the convergence of the whole sequence in Proposition \ref{the:conv_os} but the $\Gamma$-convergence proof shows that all limit points are the $m$-induced clusters. 
\end{rem}

The prove of the proposition given in Subsection \ref{sec:convergence} will rely on a variational characterization of the order statistic (see Lemma \ref{lem:char_median}) and the weak-to-strong convergence Theorem \ref{the:weak_strong} of the heat equation in \cite{laux2021large} that has recently been improved to the optimal scaling regime in \cite{ullrich2024medianfiltermethodmean}.

Putting the three major components together, the proof of Theorem \ref{the:improvedRunning} is indeed just one application of Chebyshev's inequality:
\begin{proof}[Proof of Theorem \ref{the:improvedRunning}]
\emph{Step 1: We claim that for all $h \ll \frac{\min_{i = 1,\dots, P} V_i}{P (E_0 + 1)} $ there exists $N_0 \in \N$ such that for all $N > N_0$ it holds
\begin{equation}
\limsup_{N \rightarrow \infty} \sum_{\ell = 1}^L \left|m^\ell_N - \frac{1}{P}\mathbbold{1}\right|^2 \leq C
\end{equation}
with $C := \frac{ C_0 P^5}{ \min_{i = 1,\dots, P} V_i^4} (1+T) (1+ E_0^6)$ the constant of Proposition \ref{the:l2estimate}.} 

Assume not, then there exists a subsequence $(N_1) \subset \N$ such that 
\begin{equation}\label{eq:nogation_bound}
\lim_{N_1 \rightarrow \infty}  \sum_{\ell = 1}^L \left|m^\ell_{N_1} - \frac{1}{P}\mathbbold{1}\right|^2 > C.
\end{equation}
By Assumption \ref{ass:manifold} we can apply Proposition \ref{the:conv_os} to get convergence of the order statistic $m_{N_1}^1 \rightarrow m_\mu^1$ and a $\chi^1$, a solution of one iteration of the continuous MBO scheme given by \eqref{alg:volumeMBOcont}, together with a further subsequence $(N_2) \subset \N$ such that $\chi^1_{N_2} \xrightharpoonup[]{TL^2} \chi^1$. By inductively applying Proposition \ref{the:conv_os} we get a subsequence $(N_L) \subset \N$ such that for all $\ell \in \{1, \dots, L\}$ it holds $m_{N_L}^\ell \rightarrow m_\mu^\ell$ and  $\chi^\ell_{N_L} \xrightharpoonup[]{TL^2} \chi^\ell$. Here $\chi^\ell$ are iterates of the volume constrained MBO scheme \eqref{alg:volumeMBOcont} and $m_\mu^\ell$ are the respective $V$-order statistics. Thus we can apply Proposition \ref{the:l2estimate} to bound the norm of the $m^\ell_\mu$ which yields 
\begin{align*}
\lim_{N_L \rightarrow \infty}  \sum_{\ell = 1}^L \left|m^\ell_{N_L} - \frac{1}{P}\mathbbold{1}\right|^2 =  \sum_{\ell = 1}^L \left|m^\ell_\mu- \frac{1}{P}\mathbbold{1}\right|^2 \leq C,
\end{align*}
a contradiction to \eqref{eq:nogation_bound}.

\emph{Step 2: Conclusion via the Chebyshev inequality.}
We define the set of ``bad'' iterations by 
\begin{equation*}
S := \left\{\ell \in \{1, \dots, L\}: \left|m^\ell_N - \frac{1}{P}\mathbbold{1}\right| > \delta \right\}.
\end{equation*}
Now by the \emph{Step 1} for all $h \ll \frac{\min_{i = 1,\dots, P} V_i}{P (E_0 + 1)} $ there exists $N_0 \in \N$ such that for all $N > N_0$ it holds 
\begin{align*}
h \sum_{\ell=1}^L \left|\frac{ m_N^\ell - \frac{1}{P}\mathbbold{1} }{\sqrt{h}}\right|^2 \leq C.
\end{align*}
Thus, we can bound the number of bad iterations by calculating
\begin{equation}
\# S = \sum_{\ell \in S} 1 \leq \sum_{\ell \in S} \frac{|m^\ell_N - \frac{1}{P}\mathbbold{1}|^2}{\delta^2} =  \frac{C}{\delta^2}.
\end{equation}
Therefore, we have at most $\frac{C}{\delta^2}$ iterations for which we bound the running time only by
\begin{equation*}
\mathcal{O} \big(P^2 N (\log N +P)\big)
\end{equation*}
according to Theorem \ref{the:runningtime}.

Otherwise we have a ``good'' iteration with $|m^\ell_N - \frac{1}{P}\mathbbold{1}| \leq \delta$. Using the improved running time of a single iteration of Proposition \ref{cor:iterationbound} we infer that the running time is then given by
\begin{equation*}
\mathcal{O} \left(\sqrt{h} \frac{P}{1 - 2 \delta P} P^3 N \log N \right)
\end{equation*}
if we use $m^0 = \frac{1}{P}\mathbbold{1}$. If $m^0 = m_N^{\ell - 1}$ is used one easily adapts the proof such that also $|m^{\ell-1}_N - \frac{1}{P}\mathbbold{1}| \leq~\delta$ holds. Adding up all the iterations yields the result. 
\end{proof}

\begin{rem}
Our results can also be used to improve the running time analysis of the auction dynamics algorithm presented in \cite{jacobs2018auction}. Their auction algorithm finds an approximated solution to the linear program \eqref{alg:volumeMBO}. Approximated solution means that they find an approximated  order statistic $\tilde{m}$ in the sense that for a small parameter $\eps > 0$ it holds
\begin{equation}\label{eq:eps_cs}
u_i(x) - \tilde{m}_i + \eps \geq \max_{1 \leq j \leq P} u_j(x) - \tilde{m}_j
\end{equation}
for all $i = 1, \dots, P$ and $x \in X_N$ with $\chi_i(x) = 1$. As these inequalities do not change under a constant translation $\tilde{m} + c \mathbbold{1}$ with $c \in \R$ we can assume that $\sum_{i = 1}^P \tilde{m}_i = 0$. From the inequalities \eqref{eq:eps_cs} one can conclude that there exists an exact order statistic $m$ and $C > 0$ such that ${\sum_{i = 1}^P m_i = 0}$ and 
\begin{align*}
|\tilde{m}_i - m_i| \leq C \eps \quad \text{for all } i = 1, \dots, P.
\end{align*}
Let $\bar{m}$ denote the approximate order statistic returned by the auction algorithm. Note $\bar{m}$ does not necessarily sum to $0$. To get an improved running time for the auction algorithm it is necessary to bound $\max_{i = 1, \dots, P} \bar{m}_i$ as this yields an upper bound on the number of bets during the algorithm. To be more precise one needs at most $P N \frac{\max_{i = 1, \dots, P} \bar{m}_i}{\eps}$ many bets as at most $N$ bets on one phase are necessary to rise the price on the phase by at least $\eps$. According to \cite{jacobs2018auction} the running time of one bet is given by $\mathcal{O}(\log(N) + P)$ yielding a running time of 
\begin{equation}
\mathcal{O}\left((\log N + P) P N \frac{\max_{i = 1, \dots, P} \bar{m}_i}{\eps}\right).
\end{equation}
To obtain an upper bound on $\max_{i = 1, \dots, P} \bar{m}_i$ we observe that by design of the auction dynamics, the last unassigned point will be assigned to a phase with price $0$. As all the other prices are non-negative as well we get $\min_{i = 1, \dots,P} \bar{m} = 0$. Thus the mean $\frac{1}{P} \sum_{i = 1}^P \bar{m}_i \leq \frac{P-1}{P} \max_{i = 1, \dots, P }\bar{m}_i$. From there one concludes 
\begin{align*}
\max_{i = 1, \dots, P} \bar{m}_i &\leq \max_{i = 1, \dots, P} \left(\bar{m}_i - \frac{1}{P} \sum_{j = 1}^P \bar{m}_j\right) + \frac{1}{P} \sum_{j = 1}^P \bar{m}_j\\
&\leq \max_{i = 1, \dots, P} \left(\bar{m}_i - \frac{1}{P} \sum_{j = 1}^P \bar{m}_j\right) + \frac{P-1}{P} \max_{i = 1, \dots, P} \bar{m}_i
\end{align*}
and thus
\begin{align*}
\max_{i = 1, \dots, P} \bar{m}_i \leq P \max_{i = 1, \dots, P} \left(\bar{m}_i - \frac{1}{P} \sum_{j = 1}^P \bar{m}_j\right).
\end{align*}
By our previous observations we know that there exists an order statistic $m$ that sums to zero that is $\eps$-close to $\bar{m} - \mathbbold{1} \frac{1}{P} \sum_{i = 1}^P \bar{m}_i$. This yields a good estimate
\begin{align*}
\max_{i = 1, \dots, P} \bar{m}_i \leq P \max_{i = 1, \dots, P} m_i + \eps PC \leq P|m|_2 + \eps CP.
\end{align*}
By Proposition \ref{the:l2estimate} together with Proposition \ref{the:conv_os} one gets, as in Theorem \ref{the:improvedRunning}, that for $\ell = 1,\dots, L$ iterations, $h$ small enough and $N$ big enough
\begin{align*}
\sum_{\ell = 1}^L |m^\ell|_2 \leq \sqrt{L} \left(\sum_{\ell = 1}^L |m^\ell|_2^2\right)^{\frac{1}{2}} \leq \sqrt{L} C(P,V)
\end{align*}
where $C(P,V)$ is a constant that depends on $P$ and $V$. Putting things together one achieves a running time estimate for $L$ iterations of the auction dynamics algorithm of
\begin{align*}
\mathcal{O}\left((\log N + P) P N \left(LP + \frac{L\sqrt{h}}{\sqrt{T}}\frac{C(P,V)}{\eps}\right)\right).
\end{align*}
for $h$ small enough and $N$ big enough. Here we think of $T := Lh \sim 1$, but $\eps$ small; for moderate $\eps \gtrsim 1$, this almost matches our estimate in Theorem \ref{the:improvedRunning}.
\end{rem}

\subsection{Running Time Improvement for Known Bound on Order Statistic}\label{sec:iterationbound}
\begin{figure}
\includegraphics[scale=0.4]{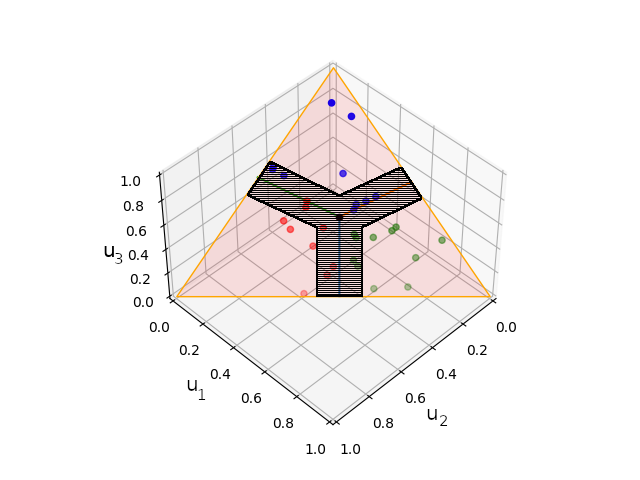}
\caption{Visualization of the region $Y_\delta$ that contains all points that can change during the algorithm if the order statistic is close to the center.}
\label{fig:y_eps}
\end{figure}
We want to show Proposition \ref{cor:iterationbound}, that is, if the order statistic $m_N^\ell$ is $\delta$-close to the center $\frac{1}{P}\mathbbold{1}$ of the simplex, then the running time of Algorithm \ref{alg:median} is improved. We observe now that points that change their cluster during the algorithm are contained in a small region $Y_\delta$ if the order statistic is close to the center $\frac{1}{P} \mathbbold{1}$. One should think of the following region $Y_\delta$ as an $\delta$-tube around the hyperplanes through $\frac{1}{P} \mathbbold{1}$ (see Fig. \ref{fig:y_eps} for a schematic picture).
\begin{lemma}[Region of points of interest]\label{lem:region}
Assume either we use $m^0 = \frac{1}{P}\mathbbold{1}$ as initial guess of Algorithm~\ref{alg:median} and the $V$-order statistic $m^\ell_N$ lies in the $\delta$-ball around $m^0$, i.e., $|m_N^\ell - \frac{1}{P}\mathbbold{1}|_\infty < \delta$, or we use the previous order statistic as initial guess $m^0 = m_N^{\ell-1}$ and it holds $\max\{|m_N^\ell - \frac{1}{P}\mathbbold{1}|_\infty, |m_N^\ell - \frac{1}{P}\mathbbold{1}|_\infty\} < \delta$. If a point $u$ is clustered differently by $m^\ell_N$ than by $m^0$ then 
\begin{equation}
u \in Y_\delta := \left\{u \in \mathbb{R}^P: 0 \leq u_k \leq 1, \exists i \neq j \text{ s.t. } u_i, u_j \geq \frac{1}{P}- 2\delta \right\}.
\end{equation}
\end{lemma}
\begin{proof}
We first consider the case $m^0 = \frac{1}{P}\mathbbold{1}$. Assume the point $u$ is assigned into Cluster $i$ by $m^0$ and into Cluster $j$ by $m^\ell_N$. Then we have by Definition \ref{def:median} that
\begin{equation*}
\begin{tabular}[c]{r l l}
$u_i - m^0_i$&\!\!\!\!$\geq u_k - m^0_k$ & \text{ for all } $k \neq i$,\\
$u_j - m^\ell_{N,j}$&\!\!\!\!$\geq u_k - m^\ell_{N,k}$& \text{ for all } $k \neq j$.
\end{tabular}
\end{equation*}
Thus, as $m^0_k = m^0_{k'}$ and $|m^\ell_{N,k} - m^0_k| < \delta$ for all $k,k'$ we conclude
\begin{alignat}{2}
u_i&\geq u_k && \text{ for all } k \neq i, \label{eq:bound1_computation}\\
u_j&\geq u_k - m_{N,k}^\ell + m_k^0 - m_j^0 + m_{N,j}^\ell \geq u_k - 2\delta && \text{ for all } k \neq j \label{eq:bound_computation}.
\end{alignat}
Now the result follows by observing $1 = \sum_{k = 1}^P u_k \leq P\ u_i$ and $1 = \sum_{k = 1}^P u_k \leq P (u_j + 2\delta)$ . The proof for $m^0 = m_N^{\ell-1}$ is analogous by replacing \eqref{eq:bound1_computation} by the counterpart of \eqref{eq:bound_computation} for $u_i$.
\end{proof}
The next step is to bound the number of points in $Y_\delta$:
\begin{lemma}[Bound on number of points of interest]\label{lem:boundonpoints}
If Assumptions \ref{ass:manifold} hold and $\chi:M \rightarrow \{0,1\}^P$ with $\sum_{i = 1}^P \chi_i = 1$ is fixed then the number of points in $Y_\delta$ is up to constants bounded by $N \sqrt{h}$. More precisely it holds
\begin{equation} 
\nu^h_N(Y_\delta) \leq \frac{N P}{1-2 \delta P} \sqrt{h}E_{N,h}(\chi),
\end{equation}
where $\nu^h_N := (u_N^h)_{\#} \left(\sum_{x \in X_N} \delta_x \right) = \sum_{x \in X_N} \delta_{u_N^h(x)}$ is $N$ times the empirical measure of the points $u_N^h(x)$ and $u_N^h = e^{-h \Delta_N} \chi$ denotes the diffused clusters.
\end{lemma}
\begin{proof}
Since for $u \in Y_\delta$ and $\sum_{i = 1}^P \chi_i = 1$
\begin{equation*}
\sum_{i = 1}^P \sum_{j \neq i} \chi_i u_j = \sum_{i = 1}^P \left(\chi_i \sum_{j \neq i} u_j\right) \geq \min_{i = 1, \dots, P} \sum_{j \neq i} u_j \geq \min_{i = 1, \dots, P} \max_{j \neq i} u_j \geq \frac{1}{P} - 2 \delta = \frac{1-2 \delta P}{P}, 
\end{equation*}
and $\nu_N^h = (u_N^h)_{\# N\mu_N}$ for $\mu_N := \frac{1}{N}\sum_{x \in X_N} \delta_x$, we have 
\begin{equation*}
\nu_N^h(Y_\delta) = N\mu_N(\{x\colon u_N^h(x) \in Y_{\delta}\})\leq \frac{P}{1-2 \delta P} \sum_{x \in X_N} \sum_{i = 1}^P \sum_{j\neq i} \chi_i(x) u_{N,j}^h(x) \leq \frac{N P}{1-2 \delta P} \sqrt{h} E_{N,h}(\chi).\qedhere
\end{equation*}
\end{proof}

Combining the two previous lemmas we can show the improved running time of Algorithm \ref{alg:median} if the order statistic is close to the center, which was stated in Proposition \ref{cor:iterationbound}.
\begin{proof}[Proof of Proposition \ref{cor:iterationbound}]
To bound the running time, according to Theorem \ref{the:runningtime}, we need to get a bound on the energy \eqref{eq:alg_energy}:
\begin{equation*}
E(m^0) = \sum_{i=1}^P \left|\sum_{x\in X_N} \chi_i^{m^0}(x) - V_i \right|\leq \sum_{i=1}^P \sum_{x\in X_N} \left|\chi_i^{m^0}(x) - \chi_i^{m^\ell_h}(x)\right|.
\end{equation*}
The right-hand side is bounded by twice the number of points that are clustered differently by $m^0$ than by $m^\ell_h$. So by Lemma \ref{lem:region} we have
\begin{equation*}
E(m^0) \leq 2 \nu_N^h(Y_\delta).
\end{equation*}
The claimed running time now follows from Lemma \ref{lem:boundonpoints} and Theorem \ref{the:runningtime}.
\end{proof}

\subsection{Convergence to Continuous order Statistic}\label{sec:convergence}
For proving convergence of the order statistic we begin by defining the volume constrained MBO scheme and the $V$-order statistic in the continuous setting of a Riemannian manifold. For that we first introduce the weighted Laplace operator on $(M, g, \mu)$
\begin{definition}[Weighted Laplace operator and heat kernel on Riemannian Manifold]\label{def:laplace_on_M}
Let $\langle v, w \rangle_x := g_x(v,w) $ be the inner product on $v,w \in T_x M$ given by the metric $g$. We use the usual gradient for a smooth function $f \in C^\infty(M)$ defined uniquely by $\langle \nabla f(x), Y\rangle_x = d_x f(Y)\ \forall \ Y \in T_x M$. Also write $\Gamma(TM)$ for the set of all smooth vector fields on $M$. Then the weighted divergence operator $\divrho : \Gamma(TM) \rightarrow C^\infty(M)$ is defined by the duality formula
\begin{equation}
\int_M \langle \nabla f(x), Y(x)\rangle_x \, d\mu (x)= - \int_M f(x) \divrho Y(x) \, d\mu (x)\quad \text{ for all } Y \in T_x M.
\end{equation}
Consequently, we define the weighted Laplacian $\Delta_\mu : C^\infty (M) \rightarrow C^\infty (M)$ as $\Delta_\mu = - \divrho \circ \nabla$. All eigenvalues of the operator are non negative by the used sign convention.
On can then show, see \cite{GrigoryanA2009Hkaa}, that there exists a smooth map $p: (0, \infty)\times M \times M \rightarrow \R$ such that 
the solution $e^{- h \Delta_\mu} \chi$ to the heat equation
\begin{align*}
\begin{cases}
\frac{d}{dh} u = - \Delta_\mu u & \text{ on } M \times (0, T],\\
u(x,0) = \chi(x) & \text{ in } M
\end{cases} 
\end{align*}
can be written as 
\begin{equation}
e^{-h \Delta_\mu} \chi = \int_M p(h,x,y) \chi(y) \, d\mu(y).
\end{equation}
We call $e^{-h \Delta_\mu}$ the heat semigroup and $p$ the heat kernel of $\Delta_\mu$.
\end{definition}
The definition of the continuous volume constrained MBO scheme and order statistic is then analogous to the discrete case:
\begin{definition}[Continuous volume constrained MBO scheme and order statistic]\label{def:cont_median}
Let volume fractions $\{V_i\}_{i = 1}^P$ with $\sum_{i = 1}^P V_i = 1$, a time-step size $h>0$, and initial conditions $\chi^0:M \rightarrow \{0,1\}^P$ with $\sum_{i = 1}^P \chi_i^0 = 1$ be given. The volume constrained MBO scheme is defined by the following iterative procedure
\begin{align}\label{alg:volumeMBOcont}
\chi_i^\ell \in \argmax_{\chi:M \rightarrow \{0,1\}^P} &\sum_{i=1}^P \int_M \chi_i(x) u^{\ell-1}_i(x)\,d\mu(x)\\
\text{s.t. } &\sum_{i=1}^P \chi_i(x) = 1 \quad \hspace{3pt} \text{ for all } x \in M, \nonumber\\
&\int_M \chi_i(x)\,dx = V_i \quad \text{ for all } i = 1,\dots,P. \nonumber
\end{align}
Here $u_i^{\ell -1 } := e^{-h \Delta_\mu} \chi^{\ell -1}_i$ is the solution to the heat equation with initial condition given by the previous phase $\chi^{\ell - 1}_i$ of the MBO scheme.

A function $\chi^{m_\mu}: M\rightarrow \{0,1\}^P$ is called an  $m_{\mu}$-induced clustering if it holds
\begin{align}\label{eq:cont_m_induced}
\sum_{i = 1}^P \chi^{m_\mu}_i(x) = 1 \quad \quad \text{ for all } x \in M,\\
\chi^{m_\mu}_i(x) = 1 \Rightarrow i \in \argmax_{j = 1, \dots, P} u_j(x) - m_j \nonumber.
\end{align}

A vector $m_{\mu}$ is called a continuous $V$-order statistic if there exists an $m_\mu$-induced clustering $\chi^{m_\mu}$ with 
\begin{equation}
\int_M \chi^{m_\mu}_i \,d\mu = V_i \quad \quad \text{ for all } i = 1, \dots, P. 
\end{equation}
\end{definition}

\begin{figure}
\begin{tikzpicture}
\draw[black, thick] (-3,0) -- (3,0);
\draw[black, thick] (3,0) -- (0,5.1961);
\draw[black, thick] (-3,0) -- (0,5.1961);
\filldraw[black] (0,1.7320) circle (0.01pt) node[]{$\frac{1}{P}\mathbbold{1}$};
\filldraw[black] (-3.2, -0.2) circle (0.01pt) node[]{$e_1$};
\filldraw[black] (3.2, -0.2) circle (0.01pt) node[]{$e_2$};
\filldraw[black] (0, 5.4) circle (0.01pt) node[]{$e_3$};
\filldraw[black] (0, -0.3) circle (0.01pt) node[]{$ |\cdot|_\triangle  = const.$};

\draw[green!25!blue, line width = 2pt] (0,0) --(1.5,2.59) -- (-1.5,2.59) -- (0,0);
\draw[green!0!blue, line width = 2pt] (0,0.86) --(0.75,2.15) -- (-0.75,2.15) -- (0,0.86);
\draw[green!40!blue, line width = 2pt] (0+ 1*0.508,0 ) --(1.5 + 1 *0.508/3 * 1.5,2.591 - 1*0.508/3 * 2.59);
\draw[green!55!blue, line width = 2pt] (0+ 2*0.508,0 ) --(1.5 + 2 *0.508/3 * 1.5,2.591 - 2*0.508/3 * 2.59);
\draw[green!70!blue, line width = 2pt] (0+ 3*0.508,0 ) --(1.5 + 3 *0.508/3 * 1.5,2.591 - 3*0.508/3 * 2.59);
\draw[green!85!blue, line width = 2pt] (0+ 4*0.508,0 ) --(1.5 + 4 *0.508/3 * 1.5,2.591 - 4*0.508/3 * 2.59);
\draw[green!100!blue, line width = 2pt] (0+ 5*0.508,0 ) --(1.5 + 5 *0.508/3 * 1.5,2.591 - 5*0.508/3 * 2.59);

\draw[green!40!blue,  line width = 2pt] (0-1*0.508,0 ) --(-1.5 - 1 *0.508/3 * 1.5,2.591 - 1*0.508/3 * 2.59);
\draw[green!55!blue,  line width = 2pt] (0- 2*0.508,0 ) --(-1.5 - 2 *0.508/3 * 1.5,2.591 - 2*0.508/3 * 2.59);
\draw[green!70!blue, line width = 2pt] (0- 3*0.508,0 ) --(-1.5 - 3 *0.508/3 * 1.5,2.591 - 3*0.508/3 * 2.59);
\draw[green!85!blue,  line width = 2pt] (0- 4*0.508,0 ) --(-1.5 - 4 *0.508/3 * 1.5,2.591 - 4*0.508/3 * 2.59);
\draw[green!100!blue,  line width = 2pt] (0- 5*0.508,0 ) --(-1.5 - 5 *0.508/3 * 1.5,2.591 - 5*0.508/3 * 2.59);

\draw[green!40!blue,  line width = 2pt] (1.5 - 1 *0.508/3 * 1.5,2.591 + 1*0.508/3 * 2.59) --(-1.5 + 1 *0.508/3 * 1.5,2.591 + 1*0.508/3 * 2.59);
\draw[green!55!blue,  line width = 2pt] (1.5 - 2 *0.508/3 * 1.5,2.591 + 2*0.508/3 * 2.59) --(-1.5 + 2 *0.508/3 * 1.5,2.591 + 2*0.508/3 * 2.59);
\draw[green!70!blue,  line width = 2pt] (1.5 - 3 *0.508/3 * 1.5,2.591 + 3*0.508/3 * 2.59) --(-1.5 + 3 *0.508/3 * 1.5,2.591 + 3*0.508/3 * 2.59);
\draw[green!85!blue, line width = 2pt] (1.5 - 4 *0.508/3 * 1.5,2.591 + 4*0.508/3 * 2.59) --(-1.5 + 4 *0.508/3 * 1.5,2.591 + 4*0.508/3 * 2.59);
\draw[green!100!blue, line width = 2pt] (1.5 - 5 *0.508/3 * 1.5,2.591 + 5*0.508/3 * 2.59) --(-1.5 + 5 *0.508/3 * 1.5,2.591 + 5*0.508/3 * 2.59);

\end{tikzpicture}
\begin{tikzpicture}
\draw[black, thick] (-3,0) -- (3,0);
\draw[black, thick] (3,0) -- (0,5.1961);
\draw[black, thick] (-3,0) -- (0,5.1961);
\filldraw[black] (-3.2, -0.2) circle (0.01pt) node[]{$e_1$};
\filldraw[black] (3.2, -0.2) circle (0.01pt) node[]{$e_2$};
\filldraw[black] (0, 5.4) circle (0.01pt) node[]{$e_3$};
\filldraw[black] (0, -0.3) circle (0.01pt) node[]{$\mathbf{H} = \partial |\cdot|_\triangle $};

\fill[green, fill opacity=0.2] (0,0) -- (0,1.7320) -- (-1.5, 2.59) -- (-3,0) -- (0,0);
\fill[red, fill opacity=0.2] (0,0) -- (0,1.7320) -- (1.5, 2.59) -- (3,0) -- (0,0);
\fill[blue, fill opacity=0.2] (1.5, 2.59) -- (0,1.7320) -- (-1.5, 2.59) -- (0,5.1961) -- (1.5, 2.59);
\filldraw[black] (0,1.7320) circle (0.01pt) node[]{$\frac{1}{P}\mathbbold{1}$};
\filldraw[black] (-1.3,0.9) circle (0.01pt) node[]{$\mathbf{H}  = e_1$};
\filldraw[black] (1.3,0.9) circle (0.01pt) node[]{$\mathbf{H}  = e_2$};
\filldraw[black] (0,2.9) circle (0.01pt) node[]{$\mathbf{H}  = e_3$};

\end{tikzpicture}
\caption{Visualization of the contour lines and the subdifferential of the asymmetric norm $|\cdot|_\triangle$}
\label{fig:tri_asy_norm}
\end{figure}
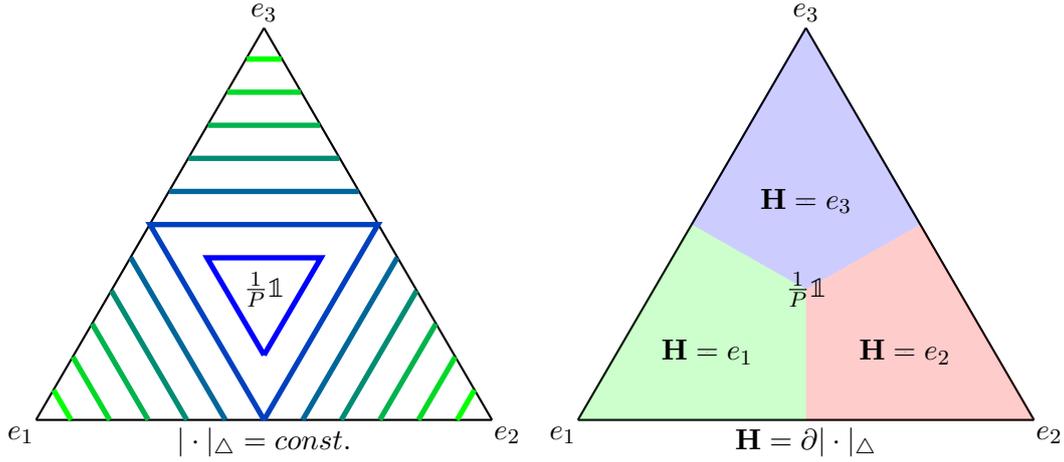

\begin{rem}
Denote $\R^P_{\Sigma} := \left\{v \in \R^P \big| \sum_{i = 1}^P v_i = 0\right\} $, then the function
\begin{equation*}
|\cdot|_\triangle:  \R^P_{\Sigma}  \rightarrow \R, v \mapsto \max_{i = 1,\dots,P} v_i,
\end{equation*}
visualized in Figure \ref{fig:tri_asy_norm}, is an asymmetric norm. This means it holds
\begin{itemize}
\item[a)] $|\lambda v|_\triangle = \lambda |v|_\triangle$ for all $v \in \R^P_{\Sigma} , \lambda \geq 0$.
\item[b)] $|v + w|_\triangle \leq |v|_\triangle + |w|_\triangle$ for all $v,w \in \R^P_{\Sigma}$.
\item[c)]  $|v|_\triangle \geq 0$ for all in $\R^P_{\Sigma}$ with $|v|_\triangle = 0$ if and only if $v = 0$.
\end{itemize}
We leave the simple proof to the interested reader.
\end{rem}
\begin{lemma}[Variational characterization of the order statistic]\label{lem:char_median}
Given volumes $\{V_i\}_{i = 1}^P$ the following two statements are equivalent:
\begin{itemize}
\item[\namedlabel{item:cont_order}{i)}]  $m_\mu$ is a continuous $V$-order statistic. 
\item[\namedlabel{item:min_order}{ii)}] $m_\mu$ is a minimizer of 
\begin{equation}
F(m):= \int_M |u(x) - m|_\triangle \,d\mu(x) + V\cdot m, \quad m \in \R^P.
\end{equation}
\end{itemize} 
\end{lemma}
The essential idea of the proof is to derive the Euler--Lagrange equation for $F$. One needs to be careful as $\max_{i = 1, \dots, P} v_i $ is not differentiable on the sets $H_{ij} := \{v \in \R^P: v_i = v_j \geq v_k \forall k\neq i,j\}$. Note that those sets may have positive mass under the pushforward measure $u_{\#} \mu$ and can thus not be neglected.

But as $F$ is a convex functional, techniques from convex analysis can be used to characterize the subdifferential of $F(m)$. 
\begin{proof}
We first define the family of linear functionals 
\begin{align*}
F_\chi(m) :=  \sum_{i = 1}^P \int_M (u_i(x) - m_i) \chi_i(x) \,d\mu(x), \quad \chi:M \rightarrow \{0,1\}^P, \sum_{i = 1}^P \chi_i = 1 \text{ a.e.}
\end{align*}
Maximizing $F_\chi(m)$ over the set of clusterings $\chi$ for fixed $m$ reduces to a pointwise maximization of $\sum_{i = 1}^P (u_i(x) - m_i) \chi_i(x)$. Thus the maximizers are given by the $m$-induced clusterings (see the definition of $m$-induced cluster \eqref{eq:cont_m_induced}), i.e.,
\begin{align*}
\argmax_{\chi:M \rightarrow \{0,1\}^P,\, \sum_{i = 1}^P \chi_i(x) = 1} F_\chi(m) = \left\{ \chi: \chi \text{ is } m \text{-induced}\right\}.
\end{align*}
Also, the maximum value is equal to $F(m) - V\cdot m$:
\begin{align*}
\max_{\chi:M \rightarrow \{0,1\}^P,\, \sum_{i = 1}^P \chi_i = 1} F_\chi(m)= \sum_{i = 1}^P \int_M (u_i(x) - m_i) \chi^m_i(x) \,d\mu(x) = F(m) - V\cdot m.
\end{align*}
So $F(m) = \max_{\chi:M \rightarrow \{0,1\}^P,\, \sum_{i = 1}^P \chi_i = 1} (F_\chi(m) + V\cdot m)$ is a maximum over affine functions and thus the subdifferential can be calculated by
\begin{align}
\partial F(m) &= \textbf{co} \left\{ \nabla_m \left(F_\chi(m) + V \cdot m\right) : F_\chi(m) = F(m)\right\} \nonumber\\
&= \textbf{co}\left\{V - \int_M \chi \,d\mu : \chi \text{ is } m \text{-induced}\right\}. \label{eq:subdifferentialF}
\end{align}
Here \textbf{co} denotes the convex hull.

If \ref{item:cont_order} holds there is an $m_\mu$-induced clustering $\chi^{m_\mu}$, i.e. $V = \int_M \chi^{m_\mu} \,d\mu$. Thus $0 \in \partial F(m_\mu)$ from which it follows that $m_\mu$ is a minimizer of $F(m)$, i.e. $ii)$ holds.

Now if \ref{item:min_order} holds there must be $0 \in \partial F(m_\mu)$. By \eqref{eq:subdifferentialF}, there exists $m$-induced $\{\chi^{(k)}\}_{k = 1, \dots, K}$ and $\{\lambda^{(k)}\}_{k = 1, \dots, K}$ with $\sum_{k = 1}^K \lambda^{(k)} = 1$ such that 
\begin{align*}
0 = \sum_{k = 1}^K \lambda^{(k)} \left(V  - \int_M \chi^{(k)}(x) \,d\mu(x) \right) = V - \int_M  \sum_{k = 1}^K \lambda^{(k)} \chi^{(k)}(x) \,d\mu(x).
\end{align*}
Thus $\sum_{k = 1}^K \lambda^{(k)} \chi^{(k)}$ has the correct volumes but does not necessarily has values in $\{0,1\}^P$. So it remains to prove that we can change $\lambda^{(k)} \chi^{(k)}$ to values in $\{0,1\}^P$ without changing its volumes.

For all 
\begin{equation*}
y \in A:= \{x \in M \, |\,\exists i \in \{1, \dots, P\}: \forall j\neq i,  u_i(x) - {m_\mu}_i(x) > u_j(x) - {m_\mu}_j(x) \  \}
\end{equation*}
it already holds $\sum_{k = 1}^K \lambda^{(k)} \chi^{(k)}(y) =  \chi^{(1)}(y) \in \{0,1\}^P$ as all $\chi^{(k)}$ are $m_\mu$ induced. So we can define the new clustering $\chi^{m_\mu}$ by $\chi^{m_\mu}(y) := \sum_{k = 1}^K \lambda^{(k)} \chi^{(k)}(y)$ for all $y \in A$.

The remaining set $M \setminus A$ is now partitioned into
\begin{align*}
&H_{i_1,\dots, i_s}(m_\mu):=\left\{u_{i_1}(\cdot)  - m_{\mu, i_1}(\cdot) = \dots = u_{i_s}(\cdot)  - m_{\mu,i_s}(\cdot) > u_{j}(\cdot)  - m_{\mu,j}(\cdot) \ \forall j\notin \{ i_1,\dots, i_s\} \right\}
\end{align*}
for $s= 2, \dots, P,\ i_1, \dots, i_s \in \{1, \dots, P\} $ pairwise different.
Now define a partition $\chi^{m_\mu} \in \{0,1\}^P$ on every $H_{i_1,\dots, i_s}(m_\mu)$ such that 
\begin{align*}
\int_{H_{i_1,\dots, i_s}(m_\mu)} \chi^{m_\mu}_j(x) \,d\mu(x) = \int_{H_{i_1,\dots, i_s}(m_\mu)} \sum_{k = 1}^K \lambda^{(k)} \chi^{(k)}_j(x) \,d\mu(x) \ \forall j \in \{1, \dots, P\}
\end{align*}
 holds. This is possible as $\mu$ is atomless by assumption and the mass on $H_{i_1,\dots, i_s}(m_\mu)$ is preserved
\begin{align*}
\sum_{j = 1}^P \int_{H_{i_1,\dots, i_s}(m_\mu)} \chi^{m_\mu}_j(x) \,d\mu(x) &= \mu(H_{i_1,\dots, i_s}(m_\mu))\\
&= \sum_{j = 1}^P \int_{H_{i_1,\dots, i_s}(m_\mu)} \sum_{k = 1}^K \lambda^{(k)} \chi^{(k)}_j(x) \,d\mu(x),
\end{align*}
so that we can set $\chi_j^{m_\mu}(x) = 1 \Leftrightarrow x \in A_j$ on $H_{i_1,\dots, i_s}(m_\mu)$, where $H_{i_1,\dots, i_s}(m_\mu) = A_{i_1} \dot{\cup} \dots \dot{\cup} A_{i_s}$ is some partition s.t. $\mu(A_j) = \int_{H_{i_1,\dots, i_s}(m_\mu)} \sum_{k = 1}^K \lambda^{(k)} \chi_j^{(k)}$.
This definition directly ensures that the mass of $\chi^{m_\mu}$ is equal to the mass of $\sum_{k = 1}^K \lambda^{(k)} \chi^{(k)}$. Thus $\chi^{m_\mu}$ is a $m_\mu$-induced clustering as $\chi^{m_\mu} = \chi^{(1)}$ on $A$, $\chi^{m_\mu}$ has the correct mass and $\chi^{m_\mu} \in \{0,1\}^P$. In a nutshell, $m_\mu$ is a $V$-order statistic.
\end{proof}

\begin{lemma}[Variational characterization of the discrete order statistic]\label{lem:discrete_char}
Given volumes $V = \{V_i\}_{i = 1}^P$ the following two statements are equivalent:
\begin{itemize}
\item[i)]  $m_N \in \R^P$ is a discrete $V$-order statistic.
\item[ii)] $m_N \in \R^P$ is a minimizer of 
\begin{equation}
F_N(m):= \sum_{x \in X_N} |u(x) - m|_\triangle + V\cdot m .
\end{equation}
\end{itemize} 
\end{lemma}
\begin{proof}
The proof is analogous to the proof of the continuous counterpart, Lemma \ref{lem:char_median}. The only part that takes some extra effort is to argue that the mass of convex-combinations $\lambda^{(k)} \chi^{(k)}$ can be assumed to be an integer on $H_{i_1,\dots, i_s}(m_N)$. This is necessary as the empirical measure is not atomless any more. We give in the following a quick sketch of this part.

So assume we have a convex combination $\{\lambda^{(k)} \chi^{(k)}\}_k$ with $\sum_{k=1}^K \sum_{x \in X_N} \lambda^{(k)} \chi^{(k)} =V$. Partition the edge cases into $H_{i_1,\dots, i_s}(m_N)$ as done in the proof of Lemma \ref{lem:char_median}.  Denote by
\begin{align*}
\mu_{N,j}(H_{i_1,\dots, i_s}(m_N)) := \sum_{k=1}^K \sum_{x \in X_N \cap H_{i_1,\dots, i_s}(m_N)} \lambda^{(k)} \chi_j^{(k)}(x)
\end{align*} the measure of the $j$-th phase on the edge case $H_{i_1,\dots, i_s}(m_N)$. We observe that 
\begin{align*}
\sum_{j = 1}^P \mu_{N,j}(H_{i_1,\dots, i_s}(m_N)) = \#\big(X_N \cap H_{i_1,\dots, i_s}(m_N)\big)
\end{align*}
is an integer. Thus, if $\mu_{N,j_0}(H_{i_1,\dots, i_s}(m_N))$ is not an integer then there is a $j_1 \neq j_0$ s.t. $\mu_{N,j_1}(H_{i_1,\dots, i_s}(m_N))$ is not an integer either. Denote by 
\begin{align*}
A_j := A_{j}(m_N) := \left\{x \in X_N: u_{j}(x) - m_{N,j}(x) \geq u_i(x) - m_{N,i} \forall i \neq j \right\}
\end{align*}
the domain of points that can potentially be assigned to phase $j \in \{1, \dots, P\}$. It also holds 
\begin{align*}
\sum_{k=1}^K \sum_{x \in X_N \cap A_{j_1}} \lambda^{(k)} \chi_j^{(k)}(x) = V_j
\end{align*}
which is an integer. Therefore, there has to be an edge case of $A_{j_1}$ named $H_{i^1_1,\dots, i^1_{s^1}}(m_N)$ with $\mu_{N,j_1}(H_{i^1_1,\dots, i^1_{s^1}}(m_N)) \notin \N$. Repeating this argument inductively one gets a  cycle of indices $j_a, j_{a+1}, \dots, j_b, j_a$ with edge cases $H_{i^\ell_1,\dots, i^\ell_{s^\ell}}(m_N)$ between $A_\ell$ and $A_{\ell+1}$ that are  not integers. One then performs a circular change of fixed mass such that one of the $\mu_{N,j_1}(H_{i^\ell_1,\dots, i^\ell_{s^\ell}}(m_N))$ is an integer while the volume constraint $\sum_{x\in X_N} \tilde{\chi} = V$ for the new phases $\tilde{\chi}$ is contained. Thus there is one more edge case with integer mass. Note that it may happen that in the circular change, points have to separate their mass to different phases. This is not an issue, because after inductively applying the whole argument all edge case sets have integer mass such that the mass can be validly reassigned as in the proof of Lemma \ref{lem:char_median}. 
\end{proof}

\begin{lemma}[Existence and uniqueness of the continuous order statistic]\label{lem:uniqueness}
Assume volumes $\{V_i\}_{i = 1}^P$ with $\sum_{i = 1}^P V_i = N$ and $V_i > 0$ are given, then there exists a unique continuous $V$-order statistic $m_\mu$.
\end{lemma}
\begin{proof}
\textit{Existence:} By Lemma \ref{lem:char_median} it is enough to show that 
\begin{align*}
F(m) = \int_M |u(x) - m|_\triangle \,d\mu(x) + V\cdot m
\end{align*}
attains its minimum. As $F(m) = F(m + c \mathbbold{1})$ it is enough to restrict the domain to all $m$ with $\sum_{i = 1}^P m_i = 1$. We claim that one can even constrain the space of competitors to  the compact set
\begin{align*}
K := \Bigg\{ m \in \R^P: &\sum_{i = 1}^P m_i = 1, \ \exists \{i_1,\dots,i_P\}=\{1,\dots,P\} \text{ s.t. } m_{i_1} \leq \dots \leq m_{i_P}\\
 & \text{ and } |m_{i_j} - m_{i_{j+1}}| \leq 1 \text{ for all } j = 1, \dots, P-1 \Bigg\}.
\end{align*}
If the claim holds, the existence of a minimizer of $F(m)$ follows by the compactness of $K$ and continuity of $F(m)$. To see the claim, take a point $m \in \R^P \setminus K$. We now construct a point $\tilde{m} \in K$ with $F(m) > F(\tilde{m})$. So $\inf_{x \in K}F(x) = \inf_{x \in \R^P} F(x)$.

As $m \in \R^P \setminus K$ there exists now an index $k$ such that w.l.o.g.\ after a change of coordinates
\begin{align*}
m_1 \leq \dots \leq m_k <  m_{k+1} -1 \leq \dots \leq m_P - 1
\end{align*}
holds. For $\eps> 0$ to be defined later set $\tilde{m_i} = m_i + \eps$ for $i \leq k$ and $\tilde{m_i} = m_i - \eps$ for $i > k$. By definition it follows that
\begin{align*}
V \cdot \tilde{m} = V \cdot m + \eps \sum_{i \leq k} V_i - \eps \sum_{i > k} V_i < V \cdot m + \eps.
\end{align*}
Now we observe that by assumption on $m$ and the fact that $u_i(x) \in [0,1]$, it is true that for all $x \in M$
\begin{align*}
u_i(x) - m_i \geq - m_i> 1 - m_j \geq u_j(x) - m_j \quad \quad \text{ for any }i \leq k \text{ and } j > k.
\end{align*}
Choose $\eps > 0$ such that $\tilde{m}_k = \tilde{m}_{k+1} -1$ and it follows again
\begin{align*}
u_i(x) - \tilde{m}_i \geq - \tilde{m}_i \geq 1 - \tilde{m}_j \geq u_j(x) - \tilde{m}_j \quad \quad \text{ for any }i \leq k \text{ and } j > k.
\end{align*}
From this it follows that
\begin{align*}
\max_{i = 1,\dots,P} (u_i(x) - m_i) &= \max_{i = 1,\dots,k} (u_i(x) - m_i )\\
&\geq \max_{i = 1,\dots,k} (u_i(x) - \tilde{m}_i) + \eps\\
&= \max_{i = 1,\dots,P} (u_i(x) - \tilde{m}_i) + \eps
\end{align*}
for all $x \in M$. All together we get
\begin{align*}
F(m) &= \int_M \max_{i  =1,\dots,P} (u_i(x) - m_i) \,d\mu(x) + V \cdot m \\
&> \int_M \max_{i  =1,\dots,P} (u_i(x) - \tilde{m}_i) + \eps \,d\mu(x) + V \cdot \tilde{m} - \eps \\
&= F(\tilde{m}).
\end{align*}
We can also assume that $\sum_{i = 1}^P \tilde{m}_i = 1$ as $F(\tilde{m}) = F(\tilde{m} + c \mathbbold{1})$ for any $c \in \R$. As $m_{i_1} - m_{i_2} = \tilde{m}_{i_1} - \tilde{m}_{i_2}$ for all $i_1, i_2 \leq k$ or $i_1, i_2 > k$ we can repeat the argument for any index $k$ with $m_k < m_{k+1} - 1$. This yields an element $\tilde{m} \in K$ with $F(m) > F(\tilde{m})$.

\textit{Uniqueness:} Assume that there are two $V$-order statistics $m$ and $\tilde{m}$.
First we note that $\mu = \rho \Vol_M$ has (path-)connected support as $\rho \geq C > 0$ and $M$ is connected by assumption. Hence also the support of $u_{\#} \mu$ is (path-)connected as $u = e^{-h \Delta_\mu} \chi$ is continuous. To see this, note that $\supp(u_{\#}\mu) \supset u(\supp (\mu))$ as preimages of open sets of $u$ are open. Also  $\supp(u_{\#}\mu) \subset u(\supp( \mu))$ as $\supp \mu = M$ is compact and $u$ maps compact sets to compact sets which implies $u(\supp \mu)$ is closed by the Heine--Borel theorem. The property of being path-connected follows by taking a path in $\supp (\mu)$ and mapping it by $u$ into $u(\supp( \mu))$. 

Denote by $\chi^m$ and $\chi^{\tilde{m}}$, respectively the $m$ and $\tilde{m}$ induced clusterings that fulfill the volume constraints. Assume without loss of generality 
\begin{align*}
m_1 - \tilde{m}_1 \leq \dots \leq m_P - \tilde{m}_P.
\end{align*}
It then follows that 
\begin{align*}
1 \in \argmax_{i = 1, \dots, P} u_i(x) - m_i
\end{align*}
if 
\begin{align*}
1 \in \argmax_{i = 1, \dots, P} u_i(x) - \tilde{m}_i.
\end{align*}
So we conclude that $\{\chi^{\tilde{m}}_1 = 1\} \subset \{\chi^{m}_1 = 1\} $. Denote the region in which the images $u( \{\chi^{\tilde{m}}_1 = 1\} )$ and $u(\{\chi^{m}_1 = 1\})$ lie by
\begin{align*}
U_1^{\tilde{m}} &:= \left\{u \in \R^P: u_1 - \tilde{m}_1 > u_i - \tilde{m}_i \ \forall i\neq 1 \right\},\\
U_1^{m} &:= \left\{u \in \R^P: u_1 - m_1 \geq u_i - \tilde{m}_i \ \forall i\neq 1 \right\}.
\end{align*}
As by assumption $V_i > 0$ for all $i = 1,\dots,P$  it has to hold
\begin{align*}
 \supp (u_{\#} \mu) \neq \supp (u_{\#} \mu) \cap  U_1^{\tilde{m}} \neq \emptyset.
\end{align*}
By the (path-)connectivity of the support of $u_{\#} \mu$ there exists an index $k \in \{2, \dots, P\}$ and a point $u^{(k)} \in H_{1k}(\tilde{m})$ such that for every $\eps > 0$ it holds
\begin{align*}
\supp (u_{\#} \mu) \cap \left(B_{\eps}(u^{(k)}) \setminus U_1^{\tilde{m}}  \right)\neq \emptyset.
\end{align*} 
This is visualized in Fig.~\ref{fig:proof_vis}. W.l.o.g.\ assume $k = 2$. If $m_1 - \tilde{m}_1 < m_2 - \tilde{m}_2$ then there exists a $\eps > 0 $ such that $m_1 - \tilde{m}_1 < m_2 - \tilde{m}_2 - \eps$. Recall that the hyperplane $H_{ij}(m)$ between Phase $i$ and Phase $j$ was given by 
\begin{align*}
H_{ij}(m) = \big\{u \in \R^P: u_i - m_i = u_j -m_j \geq u_k - m_k \ \forall k \neq i, j\big\}.
\end{align*}
which yields 
\begin{align*}
B_\eps(u^{(2)}) \subset U_1^m. 
\end{align*}
Therefore there exists $u^{(1,2)} \in \supp(u_{\#}\mu) \cap (U_1^{m} \setminus U_1^{\tilde{m}}) $ and $s > 0$ with $B_s(u^{(1,2)}) \subset U_1^{m}  \setminus U_1^{\tilde{m}} $.  Thus it holds 
\begin{align*}
u_{\#} \mu \left(U_1^{m} \setminus U_1^{\tilde{m}}\right) \geq u_{\#} \mu \left(B_s(u^{(1,2)}) \right) > 0.
\end{align*}
This leads to the following contradiction to the volume constraints 
\begin{align*}
 0 &= \int_M \chi_1^m \,d\mu -\int_M \chi^{\tilde{m}}_i \,d\mu\\
 &= \mu\left(\left\{x \in M:  u_1(x) - m_1 > u_i(x) - m_i \ \forall i\neq 1\right\}\setminus \left\{ x \in M:  u_1(x) - \tilde{m}_1 \geq u_i(x) - \tilde{m}_i \ \forall i\neq 1 \right\} \right)\\
 &= \mu(u^{-1}( U_1^{m} \setminus U_1^{\tilde{m}} ))\\
 &= u_{\#} \mu \left(U_1^{m} \setminus U_1^{\tilde{m}}\right)\\
 &> 0.
 \end{align*}
So we conclude that $m_1 - \tilde{m}_1 = m_2 - \tilde{m}_2$ has to be true.

Repeating the whole argument for the other components (you get $\{\chi^{\tilde{m}}_1 = 1\} \cap \{\chi^{\tilde{m}}_2 = 1\} \subset \{\chi^{m}_1 = 1\} \cap \{\chi^{m}_2 = 1\} $... ) leads to $m_1 - \tilde{m}_1 = \dots = m_P - \tilde{m}_P$. As $\sum_{i=1}^P m_i = \sum_{i=1}^P \tilde{m}_i = 1$ it follows that $m = \tilde{m}$. 
\end{proof}
\begin{rem}
As $F$ is convex a more standard approach for proving the uniqueness in Lemma \ref{lem:uniqueness} would be to calculate the Hessian of $F$ and see whether it is positive definite. The distributional second derivative for a test function $\phi \in C^1(M, [0, \infty))$ can be calculated to be
\begin{align*}
\langle D^2_{ij} F, \phi \rangle = 
 \left\{\begin{array}{lr}
        -\frac{1}{\sqrt{2}} \int_{M} \int_{\partial U_i \cap \partial U_j}  \phi(m + u(x)) \;d\mathcal{H}^{P-1}(m) \;d\mu(x),  & \text{for } j \neq i\\
        \frac{1}{\sqrt{2}} \int_{M} \int_{\partial U_i}  \phi(m + u(x)) \;d\mathcal{H}^{P-1}(m) \;d\mu(x),  & \text{for } j = i
        \end{array}\right.
\end{align*}
with $U_i := \{m_i > m_j\ \forall j \neq i\}$. The definiteness of the matrix $A_{ij} := \langle D^2_{ij} F, \phi \rangle$ depends on the mass of the pushforward $u_{\#}\mu$ on the boundaries $\partial U_i \cap \partial U_j$. In particular the matrix is not positive definite in general. But the curvature in the direction $(e_i - e_j)$
\begin{align*}
(e_i - e_j)^\intercal A (e_i - e_j) = \sum_{k \in \{i,j\}} \frac{1}{\sqrt{2}} \int_{M} \int_{\partial U_k}  \phi(m + u(x)) \;d\mathcal{H}^{P-1}(m) \;d\mu(x)
\end{align*}
is positive if and only if one of the boundaries $\partial U_i \cap \partial U_j$ has positive pushforward measure. With the same densities arguments as in the proof of Lemma \ref{lem:uniqueness} one can thus conclude that two minimizers $m$ and $\tilde{m}$ of $F$ can not differ in the direction $(e_i - e_j)$, i.e.\ $m_i - m_j = \tilde{m}_i - \tilde{m}_j$. But also with this approach one has to argue inductively with the connectedness of the support to get positive mass on the boundaries such that no simplification of the proof is expected.
\end{rem}
\begin{figure}
\begin{tikzpicture} 
\draw[draw=none,fill=red] plot [smooth] 
    coordinates {(0,0.4) (-1.8,0.5) (-1.7,0.9) (-0.9,0.8) (0, 0.9)}
     -- cycle;
     
 \draw[draw=none,fill=red] plot [smooth] 
    coordinates {(0, 0.9) (0.6,1.3) (0.45,2.0)}
     -- 
     plot [smooth] coordinates {(1.1,2.4) (0.9,0.9) (0, 0.4)} --cycle;
     
\draw [draw=none, fill=red] plot [smooth] 
    coordinates {(1.1,2.4) (1.0,2.6) (0.05,3.9) (0.45,2.0)}
     -- cycle;     
\draw[black, thick] (-3,0) -- (3,0);
\draw[black, thick] (3,0) -- (0,5.1961);
\draw[black, thick] (-3,0) -- (0,5.1961);
\filldraw[black] (0,1.7320) circle (3pt) node[anchor=south]{$m$};
\draw[black, thick] (0,0) -- (0,1.7320);
\draw[black, thick] (-1.5,2.5980) -- (0,1.7320);
\draw[black, thick] (1.5,2.5980) -- (0,1.7320);

\filldraw[black] (-0.6,0.9) circle (3pt) node[anchor=south]{$\tilde{m}$};
\draw[black, thick] (-0.6,0) -- (-0.6,0.9) ;
\draw[black, thick] (-2.0,1.7) -- (-0.6,0.9) ;
\draw[black, thick] (1.73,2.2) -- (-0.6,0.9) ;

 \draw[black,thick](-0.3,0.55) circle (0.2) node[anchor=west]{$\ \ B_s(u^{(1,2)})$};
 \filldraw[black] (-0.3,0.55) circle (1pt);
 \filldraw[black] (-0.6,0.55) circle (1pt) node[anchor=east]{$u^{(2)}$};

\end{tikzpicture}
\caption{Visualization of the uniqueness proof. The support of $u_{\#}\mu$ is drawn in red. }
\label{fig:proof_vis}
\end{figure}
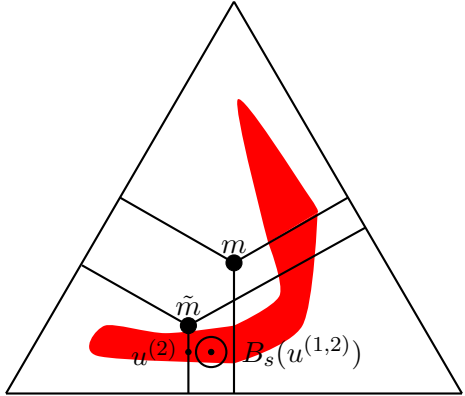

To prove convergence of the order statistic to its continuous counterpart, the key theorem is a theorem of Lelmi and the second author \cite{laux2021large} which has recently been improved by Ullrich and the second author \cite{ullrich2024medianfiltermethodmean}. It states that the solution of the heat equation converges strongly in $TL^2(M)$ if the initial values converge weakly in $TL^2(M)$. The $TL^p$ metric was introduced by Garcia-Tr\'{i}llos and Slep\v{c}ev \cite{MR3458162} to generalize $L^p$ convergence between different measure spaces. In our case, the first measure space is the discrete space $(X_N, \frac{1}{N} \sum_{x \in X_N} \delta_x )$ and the second is $(M, g, \mu)$. We use the following characterization of $TL^2$ convergence that can be proved as Garcia-Tr\'{i}llos and Slep\v{c}ev did in \cite{MR3458162}.
\begin{theorem}\label{prop:TL2convergence}
Let $T_N$ be  an arbitrary $2$-stagnating sequence of transport maps $T_N$. This means it holds
\begin{align*}
\mu_N := \frac{1}{N}\sum_{x \in X_N} \delta_x = \left(T_N\right)_{\#} \mu \quad \text{ and }\\
\int_M \dist_M^2(T_N(x), x) \,d\mu(x) \xrightarrow[]{N \rightarrow \infty} 0.
\end{align*}
Let $\mu$ be absolutely continuous with respect to the volume measure. Then the following two statements are equivalent:
\begin{itemize}
\item [i)] $(u_N, \mu_N) \rightarrow (u, \mu)$ in $TL^2$.
\item [ii)] $u_N \circ T_N \rightarrow u$ in $L^2(M, \mu)$.
\end{itemize}
\end{theorem}
This leads to the definition of weak $TL^2$ convergence introduced by Lelmi and the second author in \cite{laux2021large}.
\begin{definition}[Weak $TL^2$ convergence]\label{def:weak_TL2}
Let $\mu$ be absolutely continuous with respect to the volume measure. We say $(u_N, \mu_N)$ converges weakly to $(u, \mu)$ in $TL^p$ if
$u_N \circ T_N$ converges to $u$ weakly in $L^2(\mu)$. Here $T_N$ denotes a $2$-stagnating transport map as in Proposition \ref{prop:TL2convergence}.
\end{definition}
We recall that the graph Laplacian is given by
\begin{align}\label{def:graph_laplace}
\Delta_N u(x) =  -\frac{1}{N \eps^2_N} \sum_{y \in X_N} \frac{1}{\eps_N^d} \eta\left(\frac{\dist_M(x,y)}{\eps_N}\right) (u(y) - u(x)).
\end{align}
and the scaling of $\eps_N$ is given in Assumption \ref{ass:manifold}.
The convergence result of the heat semigroups of Laux and Lelmi \cite[Theorem 2]{laux2021large} with the improvement of \cite[Theorem 4.3]{ullrich2024medianfiltermethodmean} then reads  as follows:
\begin{theorem}[Weak to strong convergence \cite{laux2021large, ullrich2024medianfiltermethodmean}]\label{the:weak_strong}
Let Assumption \ref{ass:manifold} be in place. If
\begin{align*}
\chi_N \xrightharpoonup[]{TL^2(M,\mu)} \chi \quad \text{ weakly in } TL^2 \text{ as } N \rightarrow \infty,
\end{align*}
then almost surely for every $t>0$ the solutions to heat equation converge strongly, i.e.,
\begin{align*}
e^{-t \Delta_N} \chi_N \xrightarrow[]{TL^2(M,\mu)} e^{-t \Delta_\mu} \chi \quad \text{ strongly in } TL^2 \text{ as } N \rightarrow \infty.
\end{align*} 
\end{theorem}
The convergence of the volume constrained MBO scheme will be proven with the help of Theorem \ref{the:weak_strong} and the minimizing movement interpretation of the MBO scheme. The minimizing movement interpretation was first given by Esedo\={g}lu and Otto \cite{esedog2015threshold} in the continuous setting and afterwards adapted to the discrete setting by Bertozzi et al. \cite{MR3115457}. This interpretation shows that the volume MBO scheme is a time discretization of a gradient flow of the so-called tresholding energy. The thresholding energy in the continuous setting is given by 
\begin{align}\label{eq:thresholding_energy}
E_{h}(\chi) := \frac{1}{\sqrt{h}}\sum_{i\neq j} \int_M \chi_i(x) e^{-h \Delta_\mu} \chi_j(x) \,d\mu(x)
\end{align}
where $e^{-h\Delta_\mu}$ is the semigroup generated by the Laplace operator $\Delta_\mu$ on the manifold $(M, g, \mu)$. The metric in the minimizing movement scheme is given by
\begin{align*}
\frac{1}{2h} d_{h}^2(\chi,\tilde{\chi}) := - E_h(\chi - \tilde{\chi}) = \frac{1}{\sqrt{h}}\sum_{i = 1}^P \int_M (\chi_i - \tilde{\chi}_i)(x) \left(e^{-h\Delta_\mu}(\chi_i -\tilde{\chi}_i)\right)(x) \,d\mu(x).
\end{align*}
For more information on why this term behaves like a distance see Section 5.3 in \cite{esedog2015threshold}.

The quantities are defined analogously in the discrete setting. The discrete thresholding energy is given by
\begin{align}\label{eq:discrete_thresholding_energy}
E_{N,h}(\chi) := \frac{1}{\sqrt{h}}\sum_{i\neq j} \frac{1}{N} \sum_{x \in X_N} \chi_i(x) e^{-h \Delta_N} \chi_j(x)
\end{align} 
while the discrete distance term is
\begin{align}\label{eq:discrete_thresholding_distance}
\frac{1}{2h} d_{N,h}^2(\chi,\tilde{\chi}) := - E_{N,h}(\chi - \tilde{\chi})  = \frac{1}{\sqrt{h}}\sum_{i = 1}^P  \frac{1}{N}\sum_{x \in X_N} (\chi_i - \tilde{\chi}_i)(x) \left(e^{-h\Delta_N}(\chi_i -\tilde{\chi}_i)\right)(x).
\end{align}
In the minimizing movement interpretation we will also see the role of the order statistic as rescaled Lagrange multiplier. 
\begin{lemma}[Continuous minimizing movement interpretation]\label{lem:lagrangeMulti}
Let $m^\ell_\mu \in \R^p$ be the $V$-order statistic for the continuous MBO scheme. Then for the Lagrange multiplier 
\begin{equation}\label{def:lagrange}
\Lambda_\mu^\ell := \frac{2}{\sqrt{h}}\left(m_\mu^\ell - \frac{1}{P}\mathbbold{1}\right)
\end{equation}
it holds $\mathbbold{1} \cdot \Lambda_\mu^\ell = 0$ and we have the equivalence
\begin{equation}
\chi^\ell \in \argmin_{\int_M \chi \,d\mu= V } E_h(\chi) + \frac{1}{2h} d_{h}^2(\chi,\chi^{\ell-1})\\
\end{equation}
if and only if 
\begin{equation}
\chi^\ell \in \argmin_{\chi} E_h(\chi) + \frac{1}{2h} d_{h}^2(\chi,\chi^{\ell-1}) + \Lambda_\mu^\ell \cdot \int_M \chi \,d\mu. \label{eq:constrained_min_problem}
\end{equation}
\end{lemma}
\begin{proof}
By Definition \ref{def:cont_median} of the $V$-order statistic,
\begin{equation}\label{eq:lagrangemultis}
\chi^\ell \in \argmax_{\chi} \int_M \sum_{i=1}^P \chi_i (u_i^{\ell -1} - m_{\mu,i}) \,d\mu
\end{equation}
fulfills $\int_M \chi^\ell \,d\mu =  V = \int_M \chi^{\ell-1} \,d\mu$ where as always $u_i^{\ell-1} = e^{-h \Delta_\mu} \chi_i^{\ell-1}$.
By Lemma \ref{lem:uniqueness} such an $m_\mu^\ell$ exists and is optimal in the sense that
\begin{gather}
\chi^\ell \in \argmax_{\chi} \int_M \sum_{i=1}^P \chi_i (u_i^{\ell-1} - m_{\mu,i}^\ell) \nonumber \,d\mu \\
\Longleftrightarrow \nonumber\\ 
\chi^\ell \in \argmax_{\chi} \int_M \sum_{i=1}^P \chi_i u_i^{\ell-1} \,d\mu \quad \text{ s.t. } \int_M \chi^\ell \,d\mu=  V = \int_M \chi^{\ell-1} \nonumber \,d\mu
\end{gather}
The proof for the optimality of $\chi^m$ is analogous to the discrete proof of Theorem \ref{the:opt_criterium}. 
Also note that $m_\mu^\ell$ is optimal in the above sense if and only if $m_\mu^\ell + c\mathbbold{1}$ is optimal with $c \in \mathbb{R}$. We  pick $c = \frac{1}{P}$ s.t. $\mathbbold{1} \cdot \Lambda_\mu^\ell  = 0$ holds.
Then the result follows from the computation
\begin{align*}
&E_h(\chi) - E_h(\chi- \chi^{\ell-1}) + \Lambda_\mu^\ell \cdot  \int_M \chi \,d\mu \\
&= \frac{2}{\sqrt{h}} \sum_{i\neq j}\int_M \chi_i e^{-h\Delta_\mu} \chi_j^{\ell-1} \,d\mu + C + \Lambda_\mu^\ell \cdot \int_M \chi \,d\mu\\
&= \frac{2}{\sqrt{h}} \sum_{i=1}^P \int_M (1-\chi_i) e^{-h\Delta_\mu} \chi_i^{\ell-1} \,d\mu(x) + C + \frac{2}{\sqrt{h}} \smedian \cdot \int_M \chi \,d\mu\\
&= -\frac{2}{\sqrt{h}} \int_M \sum_{i=1}^P \chi_i \left(u_i^{\ell-1} - m_{\mu,i}^\ell\right) \,d\mu + \tilde{C}.
\end{align*}
Here, $C$ and $\tilde{C}$ denote constants independent of $\chi$. 
\end{proof}
\begin{lemma}[Discrete minimizing movement interpretation]\label{lem:lagrangeMulti_discrete}
Let $m_N^\ell \in \R^p$ be the $\{V_{N,i}\}_i$-order statistic for the continuous MBO scheme. Then for the Lagrange multiplier 
\begin{equation}
\Lambda_N^\ell := \frac{2}{\sqrt{h}}\left(m_N^\ell - \frac{1}{P}\mathbbold{1}\right)
\end{equation}
it holds $\mathbbold{1} \cdot \Lambda_N^\ell = 0$ and we have the equivalence
\begin{equation}
\chi^\ell \in \argmin_{\chi} E_{N,h}(\chi) + \frac{1}{2h} d_{N,h}^2(\chi,\chi^{\ell-1}) \text{ s.t. } \sum_{x \in X_N} \chi(x) = \sum_{x \in X_N} \chi^{\ell-1}(x)  \\
\end{equation}
if and only if
\begin{equation}
\chi^\ell \in \argmin_{\chi} E_{N,h}(\chi) + \frac{1}{2h} d_{N,h}^2(\chi,\chi^{\ell-1}) + \Lambda_N^\ell \cdot \frac{1}{N} \sum_{x \in X_N} \chi(x).
\end{equation}
\end{lemma}
\begin{proof}
Analogous to the continuous case.
\end{proof}
Having done the preparation we are now in the position to prove the main result of the section, namely the convergence of the order statistic and the volume constrained MBO scheme. 
\begin{proof}[Proof of Proposition \ref{the:conv_os}]
We start by proving the convergence of the order statistic. Remember, the order statistic is the unique minimizer of the continuous energy
\begin{align*}
F(m)= \int_M \max_{i = 1, \dots, P}(u_i(x) - m_i) \,d\mu(x) + V \cdot m
\end{align*}
by the characterization in Lemma \ref{lem:char_median}. Therefore, it is enough to show that $F(m_N) \xrightarrow{N \rightarrow \infty} F(m)$.

For showing the convergence of the order statistic we first notice that the lower bound
\begin{align*}
0 \leq F(m_N) - F(m)
\end{align*}
always holds as $m$ is the minimizer of $F$. To derive a matching upper bound we rewrite 
\begin{align*}
F_N(m_N) &= \sum_{x \in X_N} \max_{i = 1, \dots, P}(u_{N,i}(x) - m_{N,i}) + V_N \cdot m_N \\
&= N \int_M \max_{i = 1, \dots, P}(\tilde{u}_{N,i}(x) - m_{N,i})\,d\mu(x) + V_N \cdot m_N.
\end{align*}
Here $\tilde{u}_N := u_N \circ T_N$ where $T_N$ is an arbitrary $2$-stagnating transport map as described in Theorem \ref{prop:TL2convergence}. Theorem \ref{the:weak_strong} yields the convergence of $\tilde{u}_N$ pointwise a.e.\ for a subsequence to $u$. The subsequence is again denoted by the index $N \in \N$ for easier readability. Using $F_N(m_N) \leq F_N(m)$ (by Lemma \ref{lem:discrete_char}) we compute
\begin{align*}
F(m_N) - F(m) &= F(m_N) - \frac{F_N(m_N)}{N} + \frac{F_N(m_N)}{N} - F(m)\\
&\leq F(m_N) - \frac{F_N(m_N)}{N} + \frac{F_N(m)}{N} - F(m)\\
&\leq (|m_N| + |m|)\Big|V - \frac{V_N}{N}\Big| + \int_M  \max_{i = 1, \dots, P}(u_{N,i} \circ T_N  - u_i)\,d\mu \\
 &\hspace{118pt}+ \int_M  \max_{i = 1, \dots, P}(u_i - u_{N,i} \circ T_N )\,d\mu.
\end{align*}
As the right-hand side goes to zero for $N \rightarrow \infty$ (by the dominated convergence theorem), the convergence of the order statistic follows. 

Again by Theorem \ref{the:weak_strong} the thresholding energy converges continuously meaning that 
\begin{align*}
E_{N,h}(\chi_N) \xrightarrow[]{N \rightarrow \infty} E_h(\chi) \quad \text{ as } N \rightarrow \infty.
\end{align*}
Together with the convergence of the order statistics it follows that
\begin{align*}
E_{N,h}(\bullet) + \frac{1}{2h}d^2_{N,h}(\bullet, \chi_N) + \langle \Lambda_N, \bullet\rangle_{L^2(X_N)} \xrightharpoonup[]{\Gamma} E_{h}(\bullet) + \frac{1}{2h}d^2_{h}(\bullet, \chi) + \langle \Lambda_\mu, \bullet \rangle_{L^2(M, \mu)}
\end{align*}
in the weak $TL^2$-topology. The $\Gamma$-compactness follows from the fact that $\|\chi_N \circ T_N\|_{L^2(\mu_N)} \leq 1$ for arbitrary clusterings $\chi_N: X_N \rightarrow \{0,1\}^P$ and transport maps $T_N$. Thus, the volume constrained MBO scheme converges up to subsequences. By induction the result holds also true for all later iterations for further subsequences. 
\end{proof}

\subsection{$L^2$-Estimate on the $V$-Order Statistic}\label{sec:l2}
The aim of this chapter is to introduce the necessary tools for the $L^2$-estimate, Proposition \ref{the:l2estimate}, and then prove it.
We sketch our approach for the $L^2$-estimate for two phases in the sharp interface limit. These types of estimates are already well known in the community of volume-preserving mean curvature flow; see for example \cite{MR3744817, MR3455792, laux2017convergence}. However, we find it instructive to think of the simplest case first, which unfortunately cannot be found in the literature. Hence we decide to sketch the idea of the proof in this simplest case: We take a volume-preserving mean curvature flow, say on the flat torus $M = \mathbb{T}^d$.

Assume that we have an open set $\Omega(t)$ that evolves with volume-preserving mean curvature flow. The motion is then described by
\begin{equation}\label{eq:preserving_MCF}
v = -H + \Lambda \quad \text{ on } \Sigma(t) = \partial \Omega(t).
\end{equation}
where $v$ is the normal velocity, $H$ is the mean curvature and $\Lambda = \frac{1}{\Area(\Sigma(t))} \int_{\Sigma(t)} H \,dS \in \R$ is the Lagrange multiplier ensuring the mass conservation. Testing \eqref{eq:preserving_MCF} with $\xi \cdot n$ for an arbitrary vector field $\xi$ and the normal $n$ as well as using $\int_{\Sigma(t)} -H \xi \cdot n \,ds = \int_{\Sigma(t)} \div(\xi) \,dS$ and Gauss' Theorem yields
\begin{equation}\label{eq:distributional_MCF}
\int_{\Sigma(t)} v \xi \cdot n \,dS = \int_{\Sigma(t) } \div \xi \,dS + \Lambda(t) \int_{\Omega(t)} \div(\xi) \,ds.
\end{equation}
The idea is to choose for every $t \in [0,T]$  a test vector field $\xi$ such that 
\begin{equation}\label{eq:sharp_xi}
\div(\xi) = \chi_{\Omega(t)} - \int \chi_{\Omega(t)} \,dx
\end{equation}
and hope that $\| \xi \|_{\infty}, \|\nabla \xi\|_{\infty}$ are controlled. One gets by rearranging and squaring \eqref{eq:distributional_MCF}, integrating in time and using $\int \chi_\Omega \,dx = V \in (0,1)$ for some fixed $V$ that
\begin{align*}
 V^2(1-V)^2 \int_0^T \Lambda^2(t) \,dt\leq \|\xi\|_{\infty}^2 \sup_t \Area(\Sigma(t)) \int_0^T \int_{\Sigma(t)} v^2 \,dS \,dt + \|\nabla\xi\|_\infty^2 T^2 \left( \sup_t \Area(\Sigma(t))\right)^2
\end{align*}
As $\frac{d}{\,dt} \Area(\Sigma(t)) = - \int_{\Sigma(t)}v^2 \,dS$ it follows $\int_0^T \int_{\Sigma(t)} v^2 \,dS \,dt \leq \Area(\Sigma(0)) =: E_0$ and $\Area(\Sigma(t)) \leq E_0$, which implies
\begin{equation}
\int_0^T \Lambda^2(t) \,dt \leq \frac{1}{V^2(1-V)^2} \|\xi\|_{C^{1,1}}^2 T(1+T) E_0^2.
\end{equation}
Note that \eqref{eq:sharp_xi} is underdetermined, so we make the Ansatz $\xi = \nabla \phi$. Then $\Delta \phi = \chi_\Omega -V$ has a unique solution and elliptic regularity gives the sharp estimates (modulo mollifying the right-hand side on some scale $\eps > 0$).

The proof for the estimate of the Lagrange multiplier of the volume-constrained MBO scheme is similar to the presented proof sketch but has some further complications: The differential equation \eqref{eq:preserving_MCF} is not valid for the MBO scheme. Thus we have to replace \eqref{eq:distributional_MCF} by the Euler-Lagrange equation of the minimizing movement interpretation of the MBO scheme (see \eqref{eq:ELequation}). The estimates of the first variation of the energy and distance term in the Euler-Lagrange equation will be one of the main challenges of the proof later. The other hurdle will be the complications introduced by the multi-phase setting. This is solved by picking $\xi$ as in \eqref{eq:sharp_xi} for every phase by its own and then finding a suitable linear combination to isolate a norm of $\Lambda$ from the corresponding term in \eqref{eq:distributional_MCF}.

We begin by giving some more results on the minimizing movement interpretation we already witnessed in the previous section.
The minimizing movement interpretation is a strong tool that allows variational techniques to be used.

Iteratively testing the minimizing movement interpretation with the predecessor we immediately get the following a-priori energy dissipation estimate.
\begin{lemma}[Energy dissipation estimate]\label{lem:energyDis}
Let $\chi^\ell$ result of $\ell$ iterations of the MBO scheme then it holds
\begin{equation}
E_h(\chi^L) + \frac{h}{2}\sum_{\ell = 1}^L \left(\frac{d_h(\chi^\ell, \chi^{\ell-1})}{h} \right)^2 \leq E_0 := E_h(\chi^0).
\end{equation}
\end{lemma}
In the next lemma we adapt the proof of Yip and one of the authors \cite{MR3925538} of the monotonicity of a rescaled thresholding energy. The scaling is usually used for tresholding schemes in codimension 2. Although our scheme is ``the classical one'' in codimension 1 the monotonicity is still useful to compare energies at different scales. Note that the approximate monotonicity of Esedo\={g}lu and Otto \cite{esedog2015threshold} is not expected to hold in our case.

\begin{lemma}\label{lem:monotonicity_rescaled}
Let $\chi$ be fixed. Then for every $h>0$ it holds
\begin{align}
&\frac{d}{dh}\left(\frac{1}{\sqrt{h}}E_h(\chi) \right)\leq 0, \label{eq:mon_negativ}\\
&\frac{d}{dh} \left( \sqrt{h} E_h(\chi) \right) \geq 0.\label{eq:mon_positiv}
\end{align}
In particular for any constant $C > 0$ it holds
\begin{equation}
E_{Ch}(\chi) \leq \max\left\{ \sqrt{C}, \frac{1}{\sqrt{C}}\right\} E_h(\chi).
\end{equation}
\end{lemma}
\begin{proof}
We start by proving inequality \eqref{eq:mon_negativ}. Denote by $\{\psi_k\}_{k \in \mathbb{N}}$ an orthonormal basis of $L^2(M, \mu)$ of eigenfunctions of $\Delta_\mu$, i.e. 
\begin{align*}
\Delta_\mu \psi_k = \lambda_k \psi_k,
\end{align*}
where $0 \leq \lambda_1 \leq \lambda_2 \leq \dots$. It then holds $e^{-h\Delta_\mu}\psi_k = e^{-\lambda_k h}\psi_k$ as
\begin{align*}
\frac{d}{dh} e^{-\lambda_k h} \psi_k = -\lambda_k e^{-\lambda_k h}\psi_k = - \Delta_\mu e^{-\lambda_k h} \psi_k
\end{align*}
is the unique solution to the heat equation on $M$.

Now the orthonormal basis is used to transform $\frac{1}{\sqrt{h}}E_h(\chi)$ with similar ideas as used in Fourier analysis. Due to the orthonormality of the $\psi_k$, every function $f \in L^2(M, \mu)$ can be written as 
\begin{align*}
f(x) = \sum_{k = 1}^\infty \langle \psi_k, f\rangle_{L^2(M, \mu)} \psi_k(x) \quad \text{ and } \quad  \langle f,g \rangle_{L^2(M,\mu)} = \sum_{k = 1}^\infty \langle \psi_k, f\rangle_{L^2(M, \mu)} \langle \psi_k, g\rangle_{L^2(M, \mu)}.
\end{align*}
Now this decomposition yields, using also $\int_M e^{-h \Delta_\mu} \chi_i \; d\mu = \int_M \chi_i \;d \mu = \int_M \chi_i^2 \;d\mu$,
\begin{align*}
\frac{1}{\sqrt{h}} E_h(\chi) &= \frac{1}{h} \sum_{i \neq j} \int_M \chi_j(x) (e^{-h \Delta_\mu} \chi_i)(x) \,d\mu(x)\\
&=\frac{1}{h}  \sum_{i = 1}^P \left(\int_M \chi_i^2(x) \,d\mu(x) - \int_M \chi_i(x) (e^{-h \Delta_\mu} \chi_i)(x) \,d\mu(x) \right)\\
&= \sum_{i = 1}^P \sum_{k = 1}^\infty  \langle \psi_k, \chi_i \rangle_{L^2(M, \mu)} ^2 \lambda_k\left(\frac{1-e^{-\lambda_k h}}{\lambda_k h}\right).
\end{align*}
Thus it remains to prove for inequality \eqref{eq:mon_negativ} that $\frac{1-e^{-\lambda_k h}}{\lambda_k h}$ is non-increasing in $h$ for every $k$. Denoting $s = h \lambda_k$ this follows from the fact that $e^{-s} (1+s) \leq 1$ for all $ s \in \R$ and the direct computation
\begin{equation*}
\frac{d}{ds} \left(\frac{1-e^{-s}}{s}\right) = \frac{e^{-s}}{s} - \frac{1- e^{-s}}{s^2} = \frac{e^{-s} (1+s) - 1}{s^2} \leq 0.  
\end{equation*}
To prove inequality \eqref{eq:mon_positiv} we use the semi-group property to write
\begin{align*}
\sqrt{h}E_h(\chi) &= \sum_{i = 1}^P \int_M (1- \chi_i) e^{-h \Delta_\mu}\chi_i \, d\mu\\
&= \sum_{i = 1}^P \left( \int_M e^{-h \Delta_\mu} \chi_i \, d\mu - \int_M \chi_i e^{-\frac{h}{2} \Delta_\mu} e^{-\frac{h}{2} \Delta_\mu}\chi_i \, d\mu \right)\\
&= \sum_{i = 1}^P \int_M \chi_i \, d\mu - \sum_{i = 1}^P \int_M \left(e^{-\frac{h}{2} \Delta_\mu} \chi_i \right)^2 d\mu.
\end{align*}
Now, by our sign convention of $\Delta_\mu = -\divrho \circ \nabla$ inequality \eqref{eq:mon_positiv} follows immediately:
\begin{align*}
\frac{d}{dh} \left(\sqrt{h} E_h(\chi)\right) &= - \sum_{i = 1}^P \int_M e^{-\frac{h}{2} \Delta_\mu} \chi_i (-\Delta_\mu) e^{-\frac{h}{2} \Delta_\mu} \chi_i \, d\mu \\
&= \int_{i = 1}^P \int_M \left| \nabla e^{-\frac{h}{2} \Delta_\mu}\chi_i \right|^2 \, d\mu\\
&\geq 0. \qedhere
\end{align*}
\end{proof}

The previous lemma can only be used for time-scales that are comparable. To get some kind of monotonicity across scales in the tresholding energy we use the following comparison of the thresholding energy with the $L^1$-norm of the gradient. In the Euclidean case, it is easy to see that by the fundamental theorem of calculus, Fubini, and the translation invariance of $\int_{\R^d} dx$ 
\begin{align*}
\int_{\R^d} \left|u - e^{-h\Delta}u\right| \,dx &\leq \int_{R^d} \int_{R^d} \frac{1}{(4\pi)^{\frac{d}{2}}}e^{-\frac{|z|^2}{4}} |u(x)- u(x- \sqrt{h}z)| \, dz \, dx\\
&=  \int_{R^d} \int_{R^d} \frac{1}{(4\pi)^{\frac{d}{2}}}e^{-\frac{|z|^2}{4}} \left|\int_0^1 \frac{d}{dt} u(x + t\sqrt{h}z) \, dt\right| \, dz \, dx\\
&\leq \int_0^1 \int_{\R^d} \frac{1}{(4\pi)^{\frac{d}{2}}} e^{\frac{|z|^2}{4}} \sqrt{h} |z| \int_{\R^d} |\nabla u(x+t\sqrt{h}z) | \,dx \,dz \, dt\\
&\leq \sqrt{h} \left(\int_{\R^d} \frac{1}{(4\pi)^{\frac{d}{2}}} |z| e^{- \frac{|z|^2}{4}} \, dz \right) \int_{\R^d} |\nabla u(x)| \,dx.
\end{align*}
So in the Euclidean case, we have
\begin{align*}
\int_{\R^d} \left| u - e^{-h \Delta} u\right| \leq C \sqrt{h} \int_{\R^d} |\nabla u|
\end{align*}
with $C$ the first moment of the standard Gaussian.

The proof in our case of a (weighted) manifold is a careful adaptation of the proof of the Segment Inequality by Cheeger and Colding \cite{MR1320384}. The only additional ingredient are Gaussian bounds on the heat kernel that were also used in \cite{laux2021large}.
\begin{lemma}\label{lem:bound_by_grad}
For any $h > 0$ and $u \in C^1(M)$ it holds
\begin{align}\label{eq:bound_by_grad}
\int_M \left| u - e^{-h \Delta_\mu} u\right| \,d\mu \leq C(M, \mu) \sqrt{h} \left(\int_M \left| \nabla u \right| \,d\mu + 1\right)
\end{align}
where $C(M, \mu)$ is a constant that depends only on $M$ and $\mu$ but neither on $u$ or $h$.
\end{lemma}
\begin{rem}
Recall, by Assumption \ref{ass:manifold} is our domain given by a smooth, closed (and hence compact), connected Riemannian manifold $(M, \mu, g)$. The probability measure $\mu$ is absolutely continuous with respect to the volume form and has a smooth positive density. The smooth metric on the manifold is denoted by $g$. The Laplace operator $\Delta_\mu$, the heat semigroup $e^{-h \Delta_\mu}$ and the heat kernel $p(h,x,y)$ are as introduced in Definition \ref{def:laplace_on_M}.
\begin{itemize}
\item[i)] \textbf{Scaling of volume and area.} Denote by $B_r(x)$ the ball around $x$ with radius $r$ with respect to the Riemannian distance. There exists $C > 0$ such that for any $x \in M$ and any $r > 0$ it holds 
\begin{align}\label{eq:doubling_prop}
 \frac{1}{C} r^{d} \leq \mu(B_r(x)) \leq C r^d.
\end{align}
There exists $\bar{C} > 0$ and $r_0 >0$ such that, denoting by $\mathcal{A}(\theta, r)$ the area element for polar normal coordinates around a point $x \in M$, for any $\theta \in \mathbb{S}^{d-1}$ and $r > 0$ it holds
\begin{equation}\label{eq:area_scaling}
 \frac{1}{\bar{C}} r^{d-1} \leq \mathcal{A}(\theta, r) \leq \bar{C} r^{d-1}.
\end{equation}
\item[ii)] \textbf{Asymptotic expansion of the heat kernel.} There are smooth coefficient functions $v_i \in C^\infty(M \times M)$ such that for any $j \in \N$ and $K > j + \frac{d}{2}$ there exists a constant $C_K$ such that for $x,y \in M$ with $\dist_M(x,y) \leq \frac{\inj}{2}$ it holds 
\begin{align}\label{eq:asymptotic_exp}
\left|\nabla^j \left(p(h,x,y) - \frac{e^{\frac{-\dist_M^2(x,y)}{4h}}}{(4\pi h)^{d/2}} \sum_{i = 0}^K v_i(x,y) h^i \right) \right| \leq C_K h^{K + 1 - \frac{d}{2}}.
\end{align}
\item[iii)]\textbf{Gaussian bounds for heat kernel.} There are constants $C_1,C_2,C_3,C_4 > 0$ such that for every $h> 0$ and $x,y \in M$ it holds
\begin{align}\label{eq:gaussian_one}
\frac{C_{1}}{\mu(B_{\sqrt{h}}(x))}e^{-\frac{\dist_M^2(x,y)}{C_{2} h}} \leq p(h,x,y) \leq \frac{C_{3}}{\mu(B_{\sqrt{h}}(x))}e^{-\frac{\dist_M^2(x,y)}{C_{4} h}}.
\end{align}
\item[iv)] \textbf{Upper Gaussian bound for gradient of heat kernel.} There are constants $\tilde{C}_1, \tilde{C}_2 > 0$ such that for every $h> 0$ and $x,y \in M$ it holds
\begin{align}\label{eq:gaussian_two}
|\nabla_x p(h,x,y)| \leq \frac{\tilde{C}_1}{\sqrt{h} \mu(B_{\sqrt{h}}(x))}e^{-\frac{\dist_M^2(x,y)}{\tilde{C}_2 h}}.
\end{align}
\end{itemize}
The scaling of area and volume are induced by the Bishop--Gromorov inequality \cite{MR169148} (see also \cite[Lemma 36]{MR2243772}) and the explicit representation of the volume and area element on the hyperbolic space. The necessary bound on the Ricci curvature follows from our regularity assumption and the compactness of the manifold.
The asymptotic expansion can be proven by the parametrix method as done in \cite{MR1462892}.
The Gaussian bounds \eqref{eq:gaussian_one} and \eqref{eq:gaussian_two} are a consequence of the Li--Yau inequality \cite{MR1186481}.
\end{rem}
\begin{proof}[Proof of Lemma \ref{lem:bound_by_grad}]
We begin by using the heat kernel $p(h,x,y)$ to rewrite the left-hand-side of \eqref{eq:bound_by_grad} as
\begin{align*}
\int_M \left| u - e^{-h \Delta_\mu} u\right| \,d\mu = \int_M \left|\int_M p(h,x,y) (u(x) - u(y))  \,d\mu(x)\right| \,d\mu(y).
\end{align*}
Next we use the Gaussian bound \eqref{eq:gaussian_one} to estimate the heat kernel by a kernel that looks similar to the Euclidean heat kernel. This yields
\begin{align}\label{eq:rhs_of_grad}
\int_M \left| u - e^{-h \Delta_\mu} u \right| \,d\mu \lesssim  \int_M \int_M \frac{1}{\mu\big(B_{\sqrt{h}} (x)\big) } e^{-\frac{\dist_M^2(x,y)}{C_4 h}} | u(x) - u(y)| \,d\mu(y) \,d\mu(x).
\end{align}
As we want to work with polar normal coordinates it is necessary to restrict the domain of integration to the set $\left\{ x,y \in M: \dist_M(x,y) \leq \inj \right\}$ where $\inj$ is the injectivity radius of $M$. Therefore, we use the scaling property \eqref{eq:doubling_prop} and the fact that $\frac{1}{h^{d/2+1}}e^{-1/h}$ is uniformly bounded in $h$ to estimate the integral over the complementary set by
\begin{align*}
&\int_M \int_{M \setminus B_{\inj}(x)} \frac{1}{\mu\big(B_{\sqrt{h}} (x)\big) } e^{-\frac{\dist_M^2(x,y)}{C_4 h}} | u(x) - u(y)| \,d\mu(y) \,d\mu(x) \\
&\lesssim \sqrt{h} \int_M \int_{M \setminus B_{\inj}(x)} \frac{1}{h^{\frac{d}{2}+1} } e^{-\frac{\inj^2}{C_4 h}} | u(x) - u(y)| \,d\mu(y) \,d\mu(x)\\
&\lesssim \sqrt{h}.
\end{align*} 
Denote by $\gamma_{x,y}:[0, \dist(x,y)] \rightarrow M$ the unit speed geodesic from $x \in M$ to $y \in M$ to estimate the remaining term of the right-hand side of \eqref{eq:rhs_of_grad} by
\begin{align*}
&\int_M \int_{B_{\inj}(x)} \frac{1}{\mu\left(B_{\sqrt{h}} (x)\right) }  e^{-\frac{\dist_M^2(x,y)}{C_4 h}} | u(x) - u(y)| \,d\mu(y) \,d\mu(x) \\
&\leq \int_M \int_{B_{\inj}(x)} \frac{1}{\mu\left(B_{\sqrt{h}} (x)\right) }  e^{-\frac{\dist_M^2(x,y)}{C_4 h}}  \int_0^{\dist_M(x,y)} \left| \frac{d}{dt} u(\gamma_{x,y}(t)) \right|\,dt  \,d\mu(y) \,d\mu(x)\\
& \leq  \int_M \int_{B_{\inj}(x)} \frac{1}{\mu\left(B_{\sqrt{h}} (x)\right) }  e^{-\frac{\dist_M^2(x,y)}{C_4 h}}  \int_0^{\dist_M(x,y)} | \nabla u(\gamma_{x,y}(t)) |\, dt \,d\mu(y) \,d\mu(x).
\end{align*}
Now define 
\begin{align*}
F(x,y)&:= \int_0^{\dist_M(x,y)} | \nabla u(\gamma_{x,y}(t)) | \,dt, \quad \quad f(x,y) = h^{-\frac{d}{2} }  e^{-\frac{\dist_M^2(x,y)}{C_4 h}}
\end{align*}
and, with a slight abuse of notation, denote by $\gamma_{x,\theta}:[0, \inj] \rightarrow M$ the unit speed geodesic starting in $x \in M$ and going in direction $\theta \in T_x M$. Using polar normal coordinates with the area element $\mathcal{A}(\theta,s)$ at radius $s$ and direction $\theta \in \mathbb{S}^{d-1}$ as well as once more the scaling property \eqref{eq:area_scaling}, it follows
\begin{align}
&\int_M \int_{B_{\inj}(x)} \frac{1}{\mu\big(B_{\sqrt{h}} (x)\big) }  e^{-\frac{\dist_M^2(x,y)}{C_4 h}} | u(x) - u(y)| \,d\mu(x) \,d\mu(y) \nonumber\\
&\lesssim   \int_M \int_{B_{\inj}(x)} f(x,y) F(x,y)\,d\mu(x) \,d\mu(y) \nonumber\\
&=\int_M \int_0^{\inj} \int_{\mathbb{S}^{d-1}} \mathcal{A}(\theta, s)  f(x,\gamma_{x,\theta}(s)) F(x,\gamma_{x,\theta}(s)) \,d\mathcal{H}^{d-1}(\theta)\,ds \,d\mu(x). \label{eq:estimate_fF}
\end{align}
By definition $f(x,\gamma_{x,\theta}(s)) = h^{-\frac{d}{2}}  e^{-\frac{s^2}{C_4 h}}$ is independent of $x$ and $\theta$. It holds by the scaling properties of the area element \eqref{eq:area_scaling} that
\begin{align*}
\mathcal{A}(\theta,s) =  \mathcal{A}(\theta,t) \frac{\mathcal{A}(\theta,s)}{\mathcal{A}(\theta,t)} \lesssim \mathcal{A}(\theta,t) \frac{s^{d-1}}{t^{d-1}}.
\end{align*} 
Together with the obvious identity $\gamma_{x,  \gamma_{x,\theta}(s)}(t) = \gamma_{x,\theta}(t)$ and $t = \dist_M(x, \gamma_{x,\theta}(t))$ it follows that
\begin{equation}\label{eq:estimate_areaF}
 \mathcal{A}(\theta,s)  F(x,\gamma_{x,\theta}(s)) \lesssim s^{d-1} \int_0^s \mathcal{A}(\theta,t)\frac{\left|\nabla u \left(\gamma_{x,\theta}(t)\right) \right|}{\dist_M^{d-1}(x, \gamma_{x,\theta}(t) )} \,dt.
\end{equation}
Plugging inequality \eqref{eq:estimate_areaF} into estimate \eqref{eq:estimate_fF} and recognizing that we can view $(\theta,t)$ (instead of $(\theta,s)$) as polar coordinates $y = \gamma_{x,\theta}(t)$ one arrives at 
\begin{align*}
&\int_M \int_{B_{\inj}(x)} \frac{1}{\mu\big(B_{\sqrt{h}} (x)\big) }  e^{-\frac{\dist_M^2(x,y)}{C_4 h}} | u(x) - u(y)| \,d\mu(x) \,d\mu(y) \\
&\lesssim \int_0^{\inj} \frac{1}{\sqrt{h}^d} e^{-\frac{s^2}{C^4h}} \int_M \int_{B_{s}(x)} \frac{\left|\nabla u \left(y\right) \right|}{\dist_M^{d-1}(x, y )} \,d\mu(y) \,d\mu(x) \,ds\\
&\lesssim  \sqrt{h}  \int_0^{\inj} \frac{s^d}{\sqrt{h}^{d+1}} e^{-\frac{s^2}{C^4h}}\,ds  \int_M \left|\nabla u \left(y\right) \right| \,d\mu(y),
\end{align*}
where we used Fubini's theorem in $x$ and $y$ in the last step and that since $\mathcal{A}(\theta,t) \lesssim t^{d-1}$ by \eqref{eq:area_scaling}, we get for any $y \in M$ 
\begin{align*}
\int_{B_s(y)} \frac{1}{\dist_M^{d-1}(x,y)} \,d\mu(x) &= \int_{\mathbb{S}^{d-1}} \int_0^s \frac{\mathcal{A}(\theta,t)}{t^{d-1}}  \,dt \,d\mathcal{H}^{d-1}(\theta) \lesssim s.
\end{align*}
The proof follows by the final estimate
\begin{equation*}
\int_0^{\inj} \frac{s^d}{\sqrt{h}^{d+1}} e^{-\frac{s^2}{C^4h}}\,ds \leq \int_0^{\infty} s^d e^{-\frac{s^2}{C^4}} \,ds < \infty.\qedhere
\end{equation*}
\end{proof}

Using the previous lemma we get a similar monotonicity estimate for the thresholding energy as in Lemma \ref{lem:monotonicity_rescaled}. The difference is that we can compare arbitrary times across scales and not only times that are scaled by a constant of $\mathcal{O}(1)$. 
\begin{lemma}\label{lem:monotonicity_independent}
For any $\chi:M \rightarrow \{0,1\}^P$ with $\sum_{i=1}^P \chi_i = 1$ a.e.\ and any $0 < h_0 \leq h$ it holds
\begin{equation}
E_h(\chi) \leq C(M, \mu) (E_{h_0}(\chi) + 1 ).
\end{equation}
\end{lemma}
\begin{proof}
For any $\tilde{\chi} \in \{0,1\}$ and $u \in [0,1]$ it holds 
\begin{equation}\label{eq:zero_one_property}
|\tilde{\chi} - u| = (1-\tilde{\chi})u + \tilde{\chi} (1 - u).
\end{equation}
Thereby, thanks to the self-adjointness of $e^{-h \Delta_\mu}$, we get the following representation of the thresholding energy
\begin{align*}
E_h(\chi) = \frac{1}{\sqrt{h}}  \sum_{i=1}^P \int_M \chi_i e^{-h \Delta_\mu} (1- \chi_i) \,d\mu = \frac{1}{2 \sqrt{h}}  \sum_{i=1}^P \int_M \left| \chi_i - e^{-h \Delta_\mu}\chi_i \right| \,d\mu.
\end{align*} 
Using the triangle inequality, the semigroup property $e^{-(t+s) \Delta_\mu} = e^{-t \Delta_\mu} e^{-s \Delta_\mu}$, and Lemma \ref{lem:bound_by_grad} with $u = e^{-h_0 \Delta_\mu} \chi_i$ one concludes
\begin{align*}
E_h(\chi) &\lesssim \frac{1}{\sqrt{h_0}}  \sum_{i=1}^P \int_M \left| \chi_i - e^{-h_0 \Delta_\mu}\chi_i \right| \,d\mu + \frac{1}{\sqrt{h - h_0}}  \sum_{i=1}^P \int_M \left| e^{-h_0 \Delta_\mu}\chi_i - e^{-(h-h_0) \Delta_\mu} e^{-h_0 \Delta_\mu}\chi_i \right| \,d\mu\\
&\lesssim E_{h_0}(\chi) + \sum_{i=1}^P \int_M \left|\nabla e^{-h_0 \Delta_\mu} \chi_i\right| \,d\mu + 1.
\end{align*}
Using $\int_M \nabla_x p(h,x,y) \,d\mu = \nabla_x \int_M  p(h,x,y) \,d\mu = 0$ as well as the Gaussian bounds \eqref{eq:gaussian_one} and \eqref{eq:gaussian_two} one computes
\begin{align*}
\sum_{i=1}^P \int_M \left|\nabla e^{-h_0 \Delta_\mu} \chi_i\right| \,d\mu  &= \sum_{i=1}^P \int_M \left| \int_M \nabla_x p\left(h_0, x,y\right) (\chi_i(y) - \chi_i(x)) \,d\mu(y)\right| \,d\mu(x) \\
&\leq \frac{1}{h_0} \sum_{i=1}^P \int_M  \int_M  p\left(\frac{\tilde{C}_2	}{C_2} h_0, x,y\right) \left|\chi_i(y) - \chi_i(x)\right| \,d\mu(y) \,d\mu(x) .
\end{align*}
By Equation \eqref{eq:zero_one_property} and the monotonicity of Lemma \ref{lem:monotonicity_rescaled} the proof is concluded with the calculation
\begin{equation*}
\sum_{i=1}^P \int_M \left|\nabla e^{-h_0 \Delta_\mu} \chi_i\right| \,d\mu  \leq 2 E_{\frac{\tilde{C}_2	}{C_2} h_0} (\chi) \lesssim E_{h_0}(\chi). \qedhere
\end{equation*}
\end{proof}

Next the first variation of the thresholding energy is defined. Note that this is done by inner variations such that the constraints $\chi_s \in \{0,1\}^P$ and $\sum_{i = 1}^P \chi_{s,i} = 1$ a.e.\ on $M$ are always satisfied. 
\begin{definition}
For any $\chi: M \rightarrow \{0,1\}^P$ with $\sum_{i = 1}^P \chi_i = 1$ a.e.\ on $M$ and any given smooth vector field $\xi \in \Gamma (TM)$ let $X_{s,i}$ for $i = 1,\dots, P$ be generated by the flow of $\xi$, i.e., $\chi_{0,i} = \chi_i$ and $\chi_{s,i}$ solves the following distributional equation:
\begin{equation}
\partial_s \chi_{s,i} + \xi \cdot \nabla \chi_{s,i} = 0.
\end{equation}
The first variation along this flow is denoted by 
\begin{equation}
\delta E_h(\chi,\xi):=\frac{d}{ds}E_h(\chi_s)\Big|_{s=0}, \quad \delta \left( \frac{1}{2h} d_h^2(\cdot, \bar{\chi}) \right)(\chi, \xi):=\frac{d}{ds} \frac{1}{2h} d_h^2(\chi_s, \bar{\chi})\Big|_{s=0}
\end{equation}
where $\bar{\chi}:M \rightarrow \in \{0,1\}^P$ with $\sum_{i = 1}^P \bar{\chi}_i = 1$ a.e.\ is fixed. 
\end{definition}
Due to the definition of the first variation we immediately get a Euler--Lagrange equation for the optimization problem \eqref{eq:constrained_min_problem} which is formulated in the next lemma.
\begin{lemma}[Euler--Lagrange equation]
Let $\chi^\ell$ be obtained by $\chi^{\ell-1}$ through the volume constrained MBO scheme  \eqref{alg:volumeMBOcont} with $V$-order statistic $m_\mu^\ell$. Then $\chi^\ell$ solves the following Euler--Lagrange equation corresponding to \eqref{eq:constrained_min_problem}
\begin{equation}\label{eq:ELequation}
\delta E_h(\chi^\ell,\xi) + \delta \left(\frac{1}{2h} d_h^2(\cdot, \chi^{\ell-1})\right)(\chi^\ell,\xi) + \Lambda_\mu^\ell \cdot \int_M \divrho (\xi)\, \chi^\ell \,d\mu = 0.
\end{equation}
\end{lemma}

The following proof of the $L^2$-estimate, Proposition \ref{the:l2estimate}, is a generalization of Laux and Swartz~\cite{laux2017convergence} to the multi cluster setting on a Riemmanian manifold. 
\begin{proof}[Proof of Proposition \ref{the:l2estimate}]
Squaring the Euler--Lagrange equation \eqref{eq:ELequation} gives for any $\ell = 1, \dots, L$ and $\xi \in \Gamma(TM)$
\begin{equation}\label{eq:euler_squared}
\left(\Lambda_\mu^\ell \cdot \int_M \divrho ( \xi) \chi^\ell \,d\mu(x) \right)^2 \lesssim (\delta E_h(\chi^\ell,\xi))^2 + \left(\delta \left( \frac{1}{2h} d_h^2(\cdot - \chi^{\ell-1}) \right)(\chi^\ell, \xi) \right)^2.
\end{equation}
The proof strategy is now to bound the right-hand side from above for an arbitrary vector field $\xi\in \Gamma(TM)$ and then construct specific vector fields to get a lower bound for the left-hand side.

For easier notation we write in the following
\begin{align*}
\normxi{\xi} := \max\{\|\xi\|_{C^0}, \|\divrho \xi\|_{C^0}, \Lip(\xi)\} 
\end{align*}
with $\Lip(\xi)$ the smallest constant such that
\begin{align*}
|\nabla_v \xi |_x \leq \Lip (\xi) |v|_x
\end{align*}
for all $x \in M$ and $v \in T_x M$. Here, we denoted by $\nabla: TM \times \Gamma(TM) \rightarrow \Gamma(TM)$ the Levi--Civita connection.

\emph{Step 1: Estimates on $\delta E_h(\chi, \xi)$. For any $\chi:M \rightarrow \{0,1\}^P$ with $\sum_{i = 1}^P \chi_i=1$ a.e.\ and any $\xi \in \Gamma(TM)$, we have 
\begin{equation}
|\delta E_h(\chi,\xi)| \lesssim \normxi{\xi}(1 + E_h(\chi)).
\end{equation}}
We use $-\langle\xi, \nabla \chi \rangle = - \divrho(\chi, \xi) + \divrho(\xi)\chi$ to compute the first variation of the energy term by integration by parts:
\begin{align}\label{eq:hard}
&\delta  E_h(\chi,\xi)\\
&= \frac{d}{ds} \frac{1}{\sqrt{h}} \sum_{i  \neq  j} \int_M \chi_{s,i} e^{-h \Delta_\mu} \chi_{s,j} \,d\mu(x) \Big\vert_{s=0} \nonumber \\
&=\frac{1}{\sqrt{h}} \sum_{i  \neq  j} \int_M \langle - \xi, \nabla \chi_i \rangle e^{-h \Delta_\mu} \chi_j + \chi_i e^{-h \Delta_\mu} \langle -\xi , \nabla \chi_j \rangle \,d\mu(x) \nonumber \\
&=\frac{1}{\sqrt{h}} \sum_{i  \neq  j} \int_M \int_M
\begin{aligned}[t] &(-\divrho(\chi_i \xi)(x) + \divrho(\xi)(x) \chi_i(x)) p(h,x,y) \chi_j(y) \nonumber \\
&+ \chi_i(x) p(h,x,y) (-\divrho(\chi_j \xi)(y) + \divrho(\xi)(y) \chi_j(y) \,d\mu(x) \,d\mu(y) 
\end{aligned} \nonumber \\
&=\frac{1}{\sqrt{h}} \sum_{i  \neq  j} \int_M \int_M
\begin{aligned}[t]
& \chi_i(x) \big( \langle \xi(x), \nabla_x p(h,x,y)\rangle_x +   \langle \xi(y), \nabla_y p(h,x,y) \rangle_y \big)\chi_j(y) \nonumber \\
&+ \divrho (\xi)(x) \chi_i(x) p(h,x,y) \chi_j(y) + \chi_i(x) p(h,x,y) \divrho(\xi)(y) \chi_j(y) \,d\mu(x) \,d\mu(y). \end{aligned} \nonumber
\end{align}
The two terms containing $\divrho \xi$ can be estimated by
\begin{align*}
&\left| \frac{1}{\sqrt{h}} \sum_{i  \neq  j} \int_M \int_M \divrho(\xi)(x) \chi_i(x) p(h,x,y)  \chi_j(y) + \chi_i(x) p(h,x,y) \divrho(\xi)(y) \chi_j(y) \,d\mu(x) \,d\mu(y) \right| \\
&\leq \|\divrho \xi\|_{C^0} \frac{1}{\sqrt{h}} \sum_{i  \neq  j} \int_M \int_M \chi_i(x) p(h,x,y)  \chi_j(y) + \chi_i(x) p(h,x,y)   \chi_j(y) \,d\mu(x) \,d\mu(y)\\
&\lesssim \|\divrho \xi\|_{C^0} E_h(\chi).
\end{align*}
For the remaining term we want to apply the asymptotic expansion \eqref{eq:asymptotic_exp} which only holds inside the injectivity radius. Thus we make the following splitting into near- and far-field contributions:
\begin{align*}
&\frac{1}{\sqrt{h}} \sum_{i  \neq  j} \int_M \int_M \chi_i(x) \big( \langle \xi(x), \nabla_x p(h,x,y)\rangle_x +   \langle \xi(y), \nabla_y p(h,x,y) \rangle_y \big)\chi_j(y)  \,d\mu(x) \,d\mu(y) \\
&= \frac{1}{\sqrt{h}} \sum_{i  \neq  j} \int_M \int_{M\setminus B_{\inj}(x)} \chi_i(x) \big( \langle \xi(x), \nabla_x p(h,x,y)\rangle_x )+   \langle \xi(y), \nabla_y p(h,x,y) \rangle_y \big) \chi_j(y)  \,d\mu(x) \,d\mu(y) \\
&\hspace{12pt}+ \frac{1}{\sqrt{h}} \sum_{i  \neq  j} \int_M \int_{B_{\inj}(x)} \chi_i(x) \big( \langle \xi(x), \nabla_x p(h,x,y)\rangle_x +   \langle \xi(y), \nabla_y p(h,x,y)\rangle_y \big) \chi_j(y)  \,d\mu(x) \,d\mu(y) \\
&=: I + I\!I.
\end{align*}
To bound the far field contribution $I$ we use the Gaussian bounds \eqref{eq:gaussian_one} and  \eqref{eq:gaussian_two}  and the scaling property \eqref{eq:doubling_prop}
\begin{align*}
I &\leq \frac{\|\xi\|_{C^0}}{\sqrt{h}} \sum_{i \neq j} \int_M \int_{M \setminus B_{\inj}(x)} \big( |\nabla_x p(h,x,y)| + |\nabla_y p(h,x,y)| \big) \,d\mu(x) \,d\mu(y) \\
&\lesssim  \frac{\|\xi\|_{C^0}}{\sqrt{h}} \sum_{i \neq j} \int_M \int_{M \setminus B_{\inj}(x)} \frac{1}{\sqrt{h}\mu\left(B_{\sqrt{h}}(x)\right)}e^{-\frac{\inj^2}{\tilde{C}_2 h} }\,d\mu(x) \,d\mu(y)\\
&\leq \|\xi\|_{C^0} \frac{P (P-1)}{h^{\frac{d}{2}+1}} e^{-\frac{\inj^2}{\tilde{C}_2 h}}\\
&\lesssim \|\xi\|_{C^0}.
\end{align*}
In fact we see that this term goes to $0$ as $h$ goes to $0$. Now we use the  asymptotic expansion \eqref{eq:asymptotic_exp} with $K = \frac{d}{2} + 2$ for the near-field contribution $I\!I$ and calculate
\begin{align*}
I\!I &\lesssim  \frac{1}{\sqrt{h}} \sum_{i  \neq  j} \int_M \int_{B_{\inj}(x)} \chi_i(x) \big( \langle \xi(x), \nabla_x p_K(h,x,y)\rangle_x +   \langle \xi(y), \nabla_y p_K(h,x,y)\rangle_y \big)\chi_j(y)  \,d\mu(x) \,d\mu(y)\\
& \hspace{12pt} + \|\xi\|_{C^0}.
\end{align*}
Here we denoted the approximate heat kernel $p_K(h,x,y) := \frac{1}{(4\pi h)^{d/2}}e^{\frac{-\dist_M^2(x,y)}{4h}} \sum_{k = 0}^K v_k(x,y) h^k$. With the approximate heat kernel we can now compute
\begin{align*}
&\frac{1}{\sqrt{h}} \sum_{i  \neq  j} \int_M \int_{B_{\inj}(x)} \chi_i(x) \big( \langle \xi(x), \nabla_x p_K(h,x,y)\rangle_x +   \langle \xi(y), \nabla_y p_K(h,x,y)\rangle_y \big)\chi_j(y)  \,d\mu(x) \,d\mu(y)\\
&=\frac{1}{\sqrt{h}} \sum_{i  \neq  j} \int_M \int_{B_{\inj}(x)} 
\begin{aligned}[t]
&\chi_i(x) \big( \big\langle \xi(x), \nabla_x \frac{1}{4h} \dist^2_M(x,y) \big\rangle_x + \big\langle \xi(y), \nabla_y \frac{1}{4h} \dist^2_M(x,y)\big\rangle_y \big)\\
\cdot& p_K(h,x,y)\chi_j(y)  \,d\mu(x) \,d\mu(y)
\end{aligned}\\
& \hspace{12pt}+ \frac{1}{\sqrt{h}} \sum_{i  \neq  j} \int_M \int_{B_{\inj}(x)}
\begin{aligned}[t]
 &\chi_i(x) \left( \langle \xi(x), \sum_{k = 1}^K \nabla_x v_k(x,y) h^k \rangle_x +   \langle \xi(y), \sum_{k = 1}^K \nabla_y v_k(x,y) h^k\rangle_y \right)\\
\cdot & \frac{1}{(4\pi h)^{d/2}}e^{\frac{-\dist_M^2(x,y)}{4h}} \chi_j(y)  \,d\mu(x) \,d\mu(y)
\end{aligned} \\
& =: I\!I' + I\!I''.
\end{align*}
For $I\!I''$ we use that the coefficients $v_k(x,y)$ are uniformly bounded for all $x,y \in M$ as $M$ is compact and the $v_k$'s are smooth. Together with the scaling property \eqref{eq:doubling_prop}, the Gaussian bound \eqref{eq:gaussian_one}, and Lemma \ref{lem:monotonicity_rescaled} this yields
\begin{align*}
I\!I'' \lesssim \frac{\|\xi\|_{C^0}}{\sqrt{h}} \sum_{i  \neq  j} \int_M \int_{B_{\inj}(x)} \chi_i(x)   p\left(\frac{4h}{C_2},x,y\right) \chi_j(y)  \,d\mu(x) \,d\mu(y)\lesssim ||\xi||_{C^0} E_h(\chi).
\end{align*}
To deal with $I\!I'$ we claim that 
\begin{align}\label{eq:cancelation_in_distance}
\left|\langle \xi(x), \nabla_x \dist^2_M(x,y) \rangle_x + \langle \xi(y), \nabla_y \dist^2_M(x,y)\rangle_y\right| \leq 2\  \dist^2_M(x,y) \Lip(\xi).
\end{align}
To prove this claim we need the following ingredients:
\begin{itemize}
\item[1.] $\nabla_x \dist_M^2(x,y) = -2 \exp_x^{-1}(y)$ for all $x,y \in M$ with $\dist_M(x,y) \leq \inj$. Here, $\exp$ is the exponential function on the Riemannian manifold. This can be proven by using a geodesic variation of the geodesic connecting $x$ to $y$.
\item[2.] $PT_{x \rightarrow y}^\gamma \left(-2\exp_x^{-1}(y)\right) = 2 \exp_y^{-1}(x)$ where $PT_{x \rightarrow y}^\gamma$ is the parallel transport (with respect to the Levi--Civita connection) from $x$ to $y$ along the geodesic $\gamma$ connecting $x$ and $y$. To see that this holds, note that $\dot{\gamma}(0) = \exp_x^{-1}(y)$ and $\dot{\gamma}(1) = -\exp_x^{-1}(y)$.
\item[3.] $|PT_{x \rightarrow y}^\gamma (\xi(x)) - \xi(y)| \leq \dist_M(x,y) \Lip(\xi)$. This holds true by  Proposition 10.46 of \cite{MR4533407}. 
\item[4.] $\left|\exp_x^{-1}(y)\right| = \dist_M(x,y)$ holds by definition of the exponential and distance function.
\item[5.] $\langle f(x), g(x) \rangle_x = \langle PT_{x \rightarrow y}^\gamma(f(x)), PT_{x \rightarrow y}^\gamma (g(x)) \rangle_y $ as the Levi--Civita connection preserves the Riemannian metric.
\end{itemize}
Combining all the ingredients we get indeed
\begin{align*}
&\left|\langle \xi(x), \nabla_x \dist^2_M(x,y) \rangle_x + \langle \xi(y), \nabla_y \dist^2_M(x,y)\rangle_y\right|\\
&= \left|\langle \xi(x), -2 \exp_x^{-1}(y) \rangle_x + \langle \xi(y), -2 \exp_y^{-1}(x)\rangle_y\right|\\
&= \left|\langle PT_{x \rightarrow y}^\gamma(\xi(x)), PT_{x \rightarrow y}^\gamma\left(-2 \exp_x^{-1}(y)\right) \rangle_y + \langle \xi(y), -2 \exp_y^{-1}(x)\rangle_y\right|\\
&= \left|\langle \xi(y) - PT_{x \rightarrow y}^\gamma(\xi(x)), -2 \exp_y^{-1}(x) \rangle_y  \right| \\
&\leq |PT_{x \rightarrow y}^\gamma (\xi(x)) - \xi(y)|  |2 \exp_y^{-1}(x)|\\
&\leq 2\ \dist^2_M(x,y) \Lip(\xi),
\end{align*} 
which is precisely the claim \eqref{eq:cancelation_in_distance}.

Now it follows by the claim \eqref{eq:cancelation_in_distance}, the Gaussian bound \eqref{eq:gaussian_one}, the scaling property \eqref{eq:doubling_prop}, and the uniform bound of the coefficient functions $v_k$ that 
\begin{align*}
&\left( \big\langle \xi(x), \nabla_x \frac{1}{4h} \dist^2_M(x,y) \big\rangle_x + \big\langle \xi(y), \nabla_y \frac{1}{4h} \dist^2_M(x,y)\big\rangle_y \right) p_K(h,x,y)\\
&\lesssim \frac{\dist_M^2(x,y)  \Lip(\xi)}{h} \frac{1}{(4\pi h)^{d/2}}e^{\frac{-\dist_M^2(x,y)}{4h}}\\
&\lesssim \frac{ \Lip(\xi)}{(4\pi h)^{d/2}}e^{\frac{-\dist_M^2(x,y)}{2\cdot 4h}}\\
&\lesssim  \Lip(\xi)\, p\left(\frac{8h}{C_2},x,y\right).
\end{align*}
So we can finally conclude with the monotonicity of Lemma \ref{lem:monotonicity_rescaled} that 
\begin{align*}
I\!I' &= \frac{1}{\sqrt{h}} \sum_{i  \neq  j} \int_M \int_{B_{\inj}(x)} 
\begin{aligned}[t]
&\chi_i(x) \left( \big\langle \xi(x), \nabla_x \frac{1}{4h} \dist^2_M(x,y) \big\rangle_x + \big\langle \xi(y), \nabla_y \frac{1}{4h} \dist^2_M(x,y)\big\rangle_y \right)\\
\cdot & p_K(h,x,y)\chi_j(y)  \,d\mu(x) \,d\mu(y)
\end{aligned}\\
&\lesssim  \Lip(\xi) E_{\frac{8h}{C_2}}(\chi)\\
&\lesssim \Lip(\xi) E_{h}(\chi).
\end{align*}
Putting all the estimates together yields Step 1.

\emph{Step 2: Estimate on $\delta \left( \frac{1}{2h} d_h^2(\cdot, \chi^{\ell-1}) \right)(\chi^\ell, \xi)$: For $\chi^\ell:M \rightarrow \{0,1\}^P$ and any $\xi_\ell \in \Gamma(TM)$ with $\ell \in \{1, \dots, L\}$ it holds
\begin{equation}
h \sum_{\ell=1}^L \left(\delta \left( \frac{1}{2h} d_h^2(\cdot, \chi^{\ell-1}) \right)(\chi^\ell, \xi) \right)^2 \lesssim P \max_{\ell = 1, \dots, L} \normxi{\xi_\ell}^2 (1+E_0) E_0.
\end{equation}}
The argument for \emph{Step 2} is as follows. For any $\xi\in \Gamma(TM)$ and any $\ell \in \{1,\dots, L\}$, we use the semigroup property $e^{-(h+\tilde{h}) \Delta_\mu} =  e^{-h \Delta_\mu} e^{-\tilde{h} \Delta_\mu}$ and $-\langle\xi, \nabla \chi \rangle = - \divrho(\chi, \xi) + \divrho(\xi)\chi$ to compute
\begin{align*}
&\delta \left( \frac{1}{2h} d_h^2(\cdot, \chi^{\ell-1}) \right)(\chi^\ell, \xi)\\
&= \frac{d}{ds} \frac{1}{\sqrt{h}} \sum_{i  = 1}^P \int_M \int_M (\chi^\ell_{s,i} - \chi_{i}^{\ell-1})(x) p(h,x,y) (\chi_{s,i}^\ell - \chi_{i}^{\ell-1})(y) \,d\mu(x) \,d\mu(y)\Big\vert_{s=0}\\
&= \frac{2}{\sqrt{h}} \sum_{i  = 1}^P \int_M \int_M  \partial_s \chi^\ell_{s,i}(x) p(h,x,y) (\chi_{s,i}^\ell - \chi_{i}^{\ell-1})(y) \,d\mu(x) \,d\mu(y)\Big\vert_{s=0}\\
&= \frac{2}{\sqrt{h}} \sum_{i  = 1}^P \int_M \int_M   p\left(\frac{h}{2},x,\tilde{y}\right) \partial_s\chi^\ell_{s,i}(\tilde{y}) \,d\mu (\tilde{y}) \int_M p\left(\frac{h}{2},x,y\right) (\chi_{s,i}^\ell - \chi_{i}^{\ell-1})(y) \,d\mu(y) \,d\mu(x) \Big\vert_{s=0}\\
&= \frac{2}{\sqrt{h}} \sum_{i  = 1}^P \int_M 
\begin{aligned}[t]
&\int_M  \left(p \left(\frac{h}{2},x,\tilde{y}\right) (-\divrho(\chi^\ell_i \xi)(\tilde{y})+\divrho(\xi)(\tilde{y})\chi_i^\ell(\tilde{y})) \right)\,d\mu(\tilde{y}) \\
 \cdot& \int_M p\left(\frac{h}{2},x,y\right) (\chi_{i}^\ell - \chi_{i}^{\ell-1})(y) \,d\mu(y) \,d\mu(x)
 \end{aligned}\\
&= \frac{2}{\sqrt{h}} \sum_{i  = 1}^P \int_M \int_M  \left\langle \xi(\tilde{y}), \nabla_{\tilde{y}} p\left(\frac{h}{2},x,\tilde{y}\right)\right\rangle_{\tilde{y}} \chi^\ell_i(\tilde{y})  \,d\mu(\tilde{y}) \int_M p\left(\frac{h}{2},x,y\right) (\chi_{i}^\ell - \chi_{i}^{\ell-1})(y) \,d\mu(y) \,d\mu(x)\\
&\quad + \frac{2}{\sqrt{h}} \sum_{i  = 1}^P \int_M \int_M  p\left(\frac{h}{2},x,\tilde{y}\right) \chi^\ell_i(\tilde{y}) \divrho(\xi)(\tilde{y}) \,d\mu(x)(\tilde{y}) \int_M p\left(\frac{h}{2},x,y\right) (\chi_{i}^\ell - \chi_{i}^{\ell-1})(y) \,d\mu(y) \,d\mu(x).
\end{align*}
We define the dissipation measures
\begin{equation}
\nu^\ell_j:= \frac{1}{\sqrt{h}}\int_M \left(e^{-\frac{h}{2} \Delta_\mu}(\chi_j^\ell-\chi_j^{\ell-1})\right)^2 d\mu(x)
\end{equation}
and use the Cauchy--Schwarz inequality once for the integral and once for the sum to compute
\begin{align*}
&\left(\delta \left( \frac{1}{2h} d_h^2(\cdot - \chi^{\ell-1}) \right)(\chi^\ell, \xi) \right)^2 \\
&\lesssim  \left(\frac{1}{\sqrt{h}} \sum_{i  = 1}^P \int_M \int_M  \left\langle \nabla_{\tilde{y}} p\left(\frac{h}{2},x,\tilde{y}\right) , \xi(\tilde{y}) \right\rangle_{\tilde{y}} \chi^\ell_i(\tilde{y}) \,d\mu(\tilde{y}) \int_M p\left(\frac{h}{2},x,y\right) (\chi_{i}^\ell - \chi_{i}^{\ell-1})(y) \,d\mu(y) \,d\mu(x)\right)^2\\
&\quad + \|\divrho \xi \|_{C^0}^2\left(\frac{1}{\sqrt{h}} \sum_{i  = 1}^P \int_M \int_M  p\left(\frac{h}{2},x,\tilde{y}\right) \chi^\ell_i(\tilde{y}) \,d\mu(\tilde{y}) \left| \int_M p\left(\frac{h}{2},x,y\right) (\chi_{i}^\ell - \chi_{i}^{\ell-1})(y) \,d\mu(y)\right| \,d\mu(x)  \right)^2 \\
&\lesssim \frac{1}{\sqrt{h}} \sum_{i  = 1}^P \int_M \left(\int_M  \left\langle \nabla_{\tilde{y}} p\left(\frac{h}{2},x,\tilde{y}\right) , \xi(\tilde{y}) \right\rangle_{\tilde{y}} \chi^\ell_i(\tilde{y}) \,d\mu(\tilde{y})\right)^2 d\mu(x) \sum_{j=1}^P \nu^\ell_j\\
& \quad+ \|\divrho \xi \|_{C^0}^2 \frac{1}{\sqrt{h}} \sum_{i  = 1}^P \int_M \left( \int_M  p\left(\frac{h}{2},x,\tilde{y}\right) \chi^\ell_i(\tilde{y}) \,d\mu(\tilde{y}) \right)^2 d\mu(x) \sum_{j=1}^P \nu^\ell_j\\
&=: I + I\!I.
\end{align*}
The second term $I\!I$ can be estimated by computing
\begin{align*}
&\|\divrho \xi \|_{C^0}^2 \frac{1}{\sqrt{h}} \sum_{i  = 1}^P \int_M \left( \int_M p\left(\frac{h}{2},x,\tilde{y}\right) \chi^\ell_i(\tilde{y}) \,d\mu(\tilde{y}) \right)^2 \,d\mu(x) \sum_{j=1}^P \nu^\ell_j\\
&= \|\divrho \xi \|_{C^0}^2 \frac{1}{\sqrt{h}} \sum_{i  = 1}^P \int_M \chi_i^\ell(x) \int_M  p\left(h,x,\tilde{y}\right) \chi^\ell_i(\tilde{y}) \,d\mu(\tilde{y}) \,d\mu(x) \sum_{j=1}^P \nu^\ell_j\\
&\leq \|\divrho \xi \|_{C^0}^2 \frac{1}{\sqrt{h}} \sum_{i  = 1}^P \int_M \chi_i^\ell(x) \int_M  p\left(h,x,\tilde{y}\right) \,d\mu(\tilde{y}) \,d\mu(x) \sum_{j=1}^P \nu^\ell_j\\
&\leq \|\divrho \xi \|_{C^0}^2 \frac{1}{h} \sqrt{h} \sum_{j=1}^P  \nu^\ell_j.
\end{align*}
The sum over the dissipation measures is estimated by the total dissipation by using $\sum_{j=1}^P (\chi_j^\ell - \chi_j^{\ell-1}) = 0$:
\begin{align*}
\sum_{j=1}^P \nu^\ell_j &=  \sum_{j=1}^P \frac{1}{\sqrt{h}}\int_M \left(e^{-\frac{h}{2} \Delta_\mu}(\chi_j^\ell-\chi_j^{\ell-1})\right)^2 d\mu(x)\\
&=  \sum_{j=1}^P \frac{1}{\sqrt{h}}\int_M (\chi_j^\ell-\chi_j^{\ell-1}) e^{-h \Delta_\mu}(\chi_j^\ell-\chi_j^{\ell-1}) \,d\mu(x)\\
&= -\sum_{j\neq k=1}^P \frac{1}{\sqrt{h}}\int_M (\chi_k^\ell-\chi_k^{\ell-1}) e^{-h \Delta_\mu}(\chi_j^\ell-\chi_j^{\ell-1}) \,d\mu(x)\\
&= \frac{1}{2h} d_h^2(\chi^\ell, \chi^{\ell-1}).
\end{align*}
To simplify the notation for the first term $I$, we denote $g(z, x):= \left\langle \xi(x), \sqrt{h} \nabla_{x} p\left(\frac{h}{2},z,x\right) \right\rangle_{x}$. As $\nabla_{x} \int_M p\left(\frac{h}{2}, z, x\right) \,d\mu(z) = 0$ it follows that $\int_M g(z, x) \,d\mu(z) = 0$. Together with $\chi_i^\ell = 1- \sum_{k \neq i} \chi_k^\ell$ one obtains
\begin{align*}
&\int_M \chi_i^\ell(z) \chi_i^\ell(y) g(x,z) g(x,y) \,d\mu(z)\\
&=  \int_M -\chi_i^\ell(z) \chi_i^\ell(y) g(z,x) g(x,y) + \chi_i^\ell(z) \chi_i^\ell(y) (g(z,x) + g(x,z)) g(x,y) \,d\mu(z)\\
& =\sum_{k \neq i} \int_M \chi_k^\ell(z) \chi_i^\ell(y) g(z,x) g(x,y) + \chi_i^\ell(z) \chi_i^\ell(y) (g(z,x) + g(x,z)) g(x,y) \,d\mu(z).
\end{align*}
Using this equality we can estimate term $I$ by
\begin{align*}
I &\leq \frac{1}{\sqrt{h}} \sum_{i = 1}^P \int_M \left(\int_M \chi^\ell_i(y)  \left\langle \nabla_y  p\left(\frac{h}{2},x,y\right), \xi(y) \right\rangle_y \,d\mu(y)\right)^2 d\mu(x) \sum_{j = 1}^P \nu^\ell_j\\
&= \frac{1}{h}\frac{1}{\sqrt{h}} \sum_{i = 1}^P \int_M \int_M \int_M \chi_i^\ell(z) \chi_i^\ell(y) g(x,z) g(x,y) \,d\mu(y) \,d\mu(z) \,d\mu(x) \sum_{j = 1}^P \nu^\ell_j\\
&= \frac{1}{h}\frac{1}{\sqrt{h}} \sum_{i \neq k} \int_M \int_M \int_M \chi_k^\ell(z) \chi_i^\ell(y) g(z,x) g(x,y) \,d\mu(y) \,d\mu(z) \,d\mu(x) \sum_{j = 1}^P \nu^\ell_j\\
&\hspace{12pt}+ \frac{1}{h} \frac{1}{\sqrt{h}} \sum_{i = 1}^P \int_M \int_M \int_M \chi_i^\ell(z) \chi_i^\ell(y) (g(z,x) + g(x,z)) g(x,y) \,d\mu(y) \,d\mu(z) \,d\mu(x) \sum_{j = 1}^P \nu^\ell_j\\
&=: a +b.
\end{align*}
To estimate $a$ we use once more the Gaussian bounds \eqref{eq:gaussian_one} and \eqref{eq:gaussian_two} as well as the monotonicity of Lemma \ref{lem:monotonicity_rescaled} to compute
\begin{align*}
a &\leq \frac{\|\xi\|^2_{C^0}}{h}\frac{1}{\sqrt{h}} \sum_{i \neq k}^P \int_M \int_M \int_M
\begin{aligned}[t]
&\chi_k^\ell(z) \chi_i^\ell(y) \left|\sqrt{h} \nabla_x p\left(\frac{h}{2},x,z\right) \right|_x \\
\cdot & \left|\sqrt{h} \nabla_y p\left(\frac{h}{2}, x, y\right) \right|_y \,d\mu(y) \,d\mu(z) \,d\mu(x) \sum_{j = 1}^P \nu^\ell_j
\end{aligned}\\
&\lesssim  \frac{\|\xi\|^2_{C^0}}{h}\frac{1}{\sqrt{h}} \sum_{i \neq k}^P \int_M \int_M \int_M \chi_k^\ell(z) \chi_i^\ell(y) p\left(\frac{h\ \tilde{C}_2}{2\ C_2},x,z \right)p\left(\frac{h\ \tilde{C}_2}{2\ C_2},x,y \right) \,d\mu(y) \,d\mu(z) \,d\mu(x) \sum_{j = 1}^P \nu^\ell_j\\
&=  \frac{\|\xi\|^2_{C^0}}{h} E_{\frac{\tilde{C}_2}{C_2}h}(\chi^\ell)\frac{1}{2h}d_h^2(\chi^\ell, \chi^{\ell-1})\\
&\lesssim \frac{\|\xi\|^2_{C^0}}{h} E_{h}(\chi^\ell)\frac{1}{2h}d_h^2(\chi^\ell, \chi^{\ell-1}).
\end{align*}
For the other term $b$ we need the asymptotic expansion \eqref{eq:asymptotic_exp} again. Thus we split $b$ along the injectivity radius in the $z$ variable
\begin{align*}
b &= \frac{1}{h} \frac{1}{\sqrt{h}} \sum_{i = 1}^P \int_M \int_M \int_{B_{\inj}(x)} \chi_i^\ell(z) \chi_i^\ell(y) (g(z,x) + g(x,z)) g(x,y) \,d\mu(z) \,d\mu(y) \,d\mu(x) \sum_{j = 1}^P \nu^\ell_j\\
& \hspace{12pt}+ \frac{1}{h} \frac{1}{\sqrt{h}} \sum_{i = 1}^P \int_M \int_M \int_{M\setminus B_{\inj}(x)} \chi_i^\ell(z) \chi_i^\ell(y) (g(z,x) + g(x,z)) g(x,y)\,d\mu(z) \,d\mu(y) \,d\mu(x) \sum_{j = 1}^P \nu^\ell_j\\
&=: b' + b''.
\end{align*}
The Gaussian bound \eqref{eq:gaussian_two} yields
\begin{align*}
b'' &\lesssim \frac{\|\xi\|^2_{C^0}}{\sqrt{h}} \sum_{i = 1}^P \int_M \int_M \int_{M\setminus B_{\inj}(x)} 
\left(\frac{\tilde{C}_1}{\mu(B_{\sqrt{\frac{h}{2}}}(x))}e^{-\frac{\inj}{\tilde{C}_2 \frac{h}{2}}} \right)^2 d\mu(z) \,d\mu(y) \,d\mu(x) \sum_{j = 1}^P \nu^\ell_j\\
&\lesssim  \frac{\|\xi\|^2_{C^0}}{h} \frac{1}{2h}d_h^2(\chi^\ell, \chi^{\ell-1}).
\end{align*}
Now use the asymptotic expansion (as in \textit{Step 1}) and the Gaussian bound \eqref{eq:gaussian_one} to estimate 
\begin{align*}
b' &\lesssim  \frac{\|\xi\|_{C^0}}{h \sqrt{h}} \sum_{i = 1}^P \int_M  \int_{B_{\inj}(x)} \left|g(x,z) + g(z,x)\right| \,d\mu(z) \,d\mu(x) \sum_{j = 1}^P \nu^\ell_j\\
&\lesssim    \frac{\|\xi\|_{C^0}}{h} \sum_{i = 1}^P
\begin{aligned}[t]
 &\int_M \int_{B_{\inj}(x)} \left| \left\langle \xi(x), \nabla_x p_K\left(\frac{h}{2},x,z\right) \right\rangle_x + \left\langle \xi(z), \nabla_z p_K\left(\frac{h}{2} x,z\right)\right\rangle_z \right| \,d\mu(z) \,d\mu(x) \\
 \cdot & \sum_{j = 1}^P \nu^\ell_j + \|\xi\|^2
 \end{aligned}
\end{align*}
This yields together with the claim \eqref{eq:cancelation_in_distance}, the Gaussian bound \eqref{eq:gaussian_one}, the scaling property \eqref{eq:doubling_prop}, and the smoothness of the $v_k$ the final estimate
\begin{align*}
b'&\lesssim \frac{\|\xi\|_{C^0}}{h} \sum_{i = 1}^P \int_M \int_{B_{\inj}(x)} \frac{1}{h} p_K\left(\frac{h}{2},x,z\right) \left| \langle \xi(x), \nabla_x \dist_M^2(x,z) \rangle_x + \langle \xi(z), \nabla_z \dist_M^2(x,z) \rangle_z \right|\\
& \hspace{12pt}+ \left| \left\langle \xi(x), \nabla_x \sum_{k = 1}^K v_k(x,z) h^k \right\rangle_x + \left\langle \xi(z), \nabla_z \sum_{k = 1}^K v_k(x,z) h^k \right\rangle_z \right| \frac{e^{\frac{-\dist_M^2(x,z)}{4h}}}{(4\pi h)^{d/2}} \,d\mu(z) \,d\mu(x) \sum_{j = 1}^P \nu^\ell_j\\
& \hspace{12pt}+ \|\xi\|^2\\
&\lesssim  \frac{\|\xi\|_{C^0}}{h} \sum_{i = 1}^P \int_M \int_M  \frac{e^{\frac{-\dist_M^2(x,z)}{4h}}}{(4\pi h)^{d/2}} \left( \Lip(\xi)\frac{\dist^2_M(x,z)}{h} + \|\xi\|_{C^0}\right) \,d\mu(z) \,d\mu(x) \sum_{j = 1}^P \nu^\ell_j + \|\xi\|^2\\
&\lesssim  \frac{\normxi{\xi}^2}{h} \frac{P}{2h}d_h^2(\chi^\ell, \chi^{\ell-1}) + \|\xi\|^2.
\end{align*}
All in all we get
\begin{align*}
\left( \delta \left( \frac{1}{2h} d_h^2(\cdot, \chi^{\ell-1}) \right)(\chi^\ell, \xi )\right)^2 &\lesssim  \frac{P}{h}\left(\normxi{\xi}^2 ( E_h(\chi^\ell) +  1)\right)
\frac{1}{2h}d_h^2(\chi^\ell, \chi^{\ell-1}).
\end{align*}
Together with the energy dissipation estimate \eqref{lem:energyDis} this yields the result.

\emph{Step 3: Choice of $\xi$. Fix $\eps > 0$ which will be determined in \emph{Step 4}. Then, for any $\chi: M \rightarrow \{0,1\}^P$ and any $h \leq \eps^2$ there exist for $i \in \{1,\dots,P\}$ vector fields ${\xi_i \in \Gamma(TM)}$, smooth function $\chi_{i,\eps} \in C^\infty(M)$ and a constant $\textbf{C}$ such that it holds
\begin{align}
\divrho(\xi_i) &= \chi_{i,\eps} - \langle \chi_i \rangle \quad \text{ for all } i \in \{1,\dots,P\}, \label{eq:div_xi}\\
\sum_{i = 1}^P \int_M |\chi_{i,\eps} - \chi_i| \,d\mu &\leq \textbf{C} \eps (E_h(\chi) + 1), \label{eq:chi_eps}\\
\normxi{\xi_i} &\leq 1 + \frac{1}{\eps} \quad \text{ for all } i \in \{1,\dots,P\} \label{eq:upper_xi}.
\end{align}}
Here $\langle f \rangle = \int_M f \,d\mu(x)$ denotes for any $f \in L^1(M)$ the integral mean. Note that $\xi$ depends on $\chi$ but the estimate \eqref{eq:upper_xi} is independent of $\chi$.

We start by defining
\begin{align*}
\chi_{i,\eps} := e^{-\eps^2 \Delta_\mu} \chi_i.
\end{align*}
As $\chi_i \in \{0,1\}$ it holds $|\chi_{i,\eps} - \chi_i| = (1 - \chi_{i, \eps}) \chi + (1 - \chi_i) \chi_{i, \eps}$. This yields together with Lemma~\ref{lem:monotonicity_independent} the estimate \eqref{eq:chi_eps}:
\begin{align*}
\sum_{i = 1}^P \int_M |\chi_{i,\eps} - \chi_i| \,d\mu  = 2 \sum_{i = 1}^P \int_M (1-\chi) \chi_{i,\eps} \,d\mu = 2 \eps E_{\eps^2}(\chi) \lesssim \eps (E_h(\chi) + 1).
\end{align*}
Now let $\phi:M \rightarrow \R$ be the unique solution to
\begin{align*}
\Delta_\mu \phi = \chi_{i,\eps} - \langle \chi_i \rangle,
\end{align*}
with $\int_M \phi \, d\mu = 0$. This unique solution exists as $M$ is a closed manifold without boundary and $\chi_{i,\eps} - \langle \chi_i \rangle$ is a smooth right-hand side that integrates to $0$, cf. Theorem 4.7 in \cite{MR681859}. Now define
\begin{align*}
\xi := \nabla \phi.
\end{align*}
It immediately follows equation \eqref{eq:div_xi} because
\begin{align*}
\divrho(\xi) = \divrho(\nabla \phi) = \Delta_\mu \phi = \chi_{i,\eps} - \langle \chi_i \rangle.
\end{align*}
By \eqref{eq:div_xi} it also follows that 
\begin{align*}
\|\divrho \xi\|_{C^0} = \sup_{x \in M} |\chi_{i,\eps}(x) - \langle \chi_i \rangle| \leq 1.
\end{align*} 
To see the last bound of \eqref{eq:upper_xi} we compute for any $x \in M$ and $v \in T_x M$
\begin{align*}
 |\nabla_v \xi|_x^2 &=  |\nabla_v \nabla \phi |_x^2 \\
 &= \langle \nabla_v \nabla \phi , \nabla_v \nabla \phi \rangle_x\\
 &= (\Hess \phi)(v, \nabla_v \nabla \phi)\\
 &\leq |\Hess \phi|_x |v|_x |\nabla_v \nabla \phi|_x.
\end{align*}
The Hessian $\Hess(f) := \nabla \nabla f \in \Gamma(T^*M \bigotimes T^* M)$ is the $(0,2)$-tensor field given by the Levi--Civita connection $\nabla$. For the above calculation we used that $\Hess(f)(X,Y) = \langle \nabla_X \nabla \ f, Y \rangle$ for every $X,Y \in \Gamma(TM)$.
So we conclude that 
\begin{align*}
|\nabla_v \xi|_x \leq  |\Hess \phi|_x |v|_x,
\end{align*}
from which it follows that $\Lip(\xi) \leq \sup_{x \in M} |\Hess \phi|_x$. Thus it is enough to show that $|\Hess \phi|_x = |\nabla \nabla \phi|_x$ is bounded uniformly over $M$. The Laplacian $\Delta_\mu$ is a elliptic operator in local normal coordinates as the density $\rho$ is assumed to be positive. Thus, one gets an interior Schauder estimate, c.f.\ Theorem~3 in \cite{MR1459795}. By a covering argument we get for every $\alpha \in (0,1)$ the global Schauder estimate 
\begin{align*}
\|\phi\|_{C^{2,\alpha}} &\lesssim \| \Delta_\mu \phi\|_{C^{\alpha}} +  \| \phi\|_{C^{0}}.
\end{align*}
Adapting the proof of Theorem 4.13 on Green's function on compact, unweighted manifolds in \cite{MR681859} to weighted manifolds with positive density one gets 
\begin{equation}
 \| \phi\|_{C^{0}} \lesssim  \| \Delta_\mu \phi\|_{C^{0}} =  \| \chi_{i, \eps} - \langle \chi_i \rangle \|_{C^{0}}
\end{equation}
as $\phi$ is fixed by $\int_M \phi \, d\mu = 0$. Together, this yields
\begin{align*}
&|\Hess \phi|_x \\
&\lesssim \|\chi_{i, \eps} - \langle \chi_i \rangle \|_{C^{0 , \alpha}}\\
&= \|\chi_{i, \eps} - \langle \chi_i \rangle \|_{C^0} + \max\left\{\sup_{x \neq y, \dist_M(x,y) < 1} \frac{|\chi_{i, \eps}(x) - \chi_{i, \eps}(y)|}{\dist_M^\alpha(x,y)}, \sup_{x \neq y, \dist_M(x,y) \geq 1} \frac{|\chi_{i, \eps}(x) - \chi_{i, \eps}(y)|}{\dist_M^\alpha(x,y)} \right\}\\
&\leq 1 + \max \{ \|\nabla \chi_{i, \eps}\|_{C^0} ,1\}.
\end{align*}
We use the Gaussian bounds \eqref{eq:gaussian_one} and \eqref{eq:gaussian_two} to compute the norm of the gradient by
\begin{align*}
\|\nabla \chi_{i, \eps}\|_{C^0} = \sup_{x \in M} \int_M | \nabla p(\eps^2, x,y)| \chi_i(y) \, d\mu(y) \lesssim \sup_{x \in M} \frac{1}{\eps} \int_M p\left(\frac{\tilde{C}_2}{C_2}\eps^2, x,y\right) \, d\mu(y) \leq \frac{1}{\eps} 
\end{align*}
The bound on $\|\xi\|_{C^0} $ follows by the previous computations together with 
\begin{align*}
\|\xi\|_{C^0} = \|\nabla \phi\|_{C^0} \leq \|\phi\|_{C^{2,\alpha}}.
\end{align*}

\emph{Step 4: Lower bound for the left-hand side of \eqref{eq:euler_squared}.}
Pick for every $\ell \in \{1, \dots, L\}$ and every $i \in \{1, \dots, {P-1}\}$ the test vector field $\xi_{i,\ell}$ by \emph{Step 3}. Therefore, we get from \emph{Step 1} and \emph{Step 2} plugged into equation \eqref{eq:euler_squared} together with the energy dissipation estimate from Lemma \ref{lem:energyDis} that
\begin{align}\label{eq:step_4_euler}
&\sum_{i = 1}^{P-1} \sum_{\ell = 1}^L\left(\Lambda_\mu^\ell \cdot \int_M \divrho ( \xi_{i, \ell}) \chi^\ell \,d\mu(x) \right)^2 \nonumber\\
&\lesssim \sum_{i = 1}^{P-1} \sum_{\ell = 1}^L \left( (\delta E_h(\chi^\ell,\xi_{i}))^2 + \left(\delta \frac{1}{2h} d_h(\cdot - \chi^{\ell-1})(\chi^\ell,\xi_{i,\ell})\right)^2\right)\nonumber \\
&\lesssim  \frac{P}{h}\sum_{i = 1}^{P-1}  \left(1 + \frac{1}{\eps}\right)^2 \left(Lh (1+ E_0)^2 + (1+E_0)E_0 \right) \nonumber \\
&\lesssim  \frac{P^2}{h}  \left(1 + \frac{1}{\eps}\right)^2 \left(T (1+ E_0^2) + 1+E_0^2 \right)
\end{align}
where  $T :=L h$ is the terminal time and $\eta_{i, \ell} \in \R$ for $i = 1, \dots,P-1$ are coefficients we will determine now. Namely, the aim is to find coefficients such that 
\begin{align*}
|\Lambda_\mu^\ell|_2^2 \stackrel{!}{=} \sum_{i = 1}^{P-1}\eta_{i, \ell}\, \Lambda_\mu^\ell \cdot \int_M \divrho ( \xi_{i}) \chi^\ell \,d\mu = \eta_\ell^\intercal A \Lambda_\mu^\ell,
\end{align*}
where $A_{i,j}:= \int_M \divrho ( \xi_{i}) \chi_j^\ell \,d\mu \in \R^{(P-1)\times P}$. \\
By our definition of the Lagrange multiplier we have $\Lambda_\mu^\ell  \cdot \mathbbold{1} = 0$. Note that $\R^P_{\Sigma} = \{v \in \R^P: v  \cdot \mathbbold{1} = 0\} \simeq \R^{P-1}$ with the bijection 
\begin{align*}
\Pi = \begin{pmatrix}
1 &  & 0 & 0 \\
 & \ddots &  & \vdots \\
0 &  & 1 & 0 
\end{pmatrix} \in  \R^{(P-1) \times P}, \quad \quad 
\Pi^{-1} = \begin{pmatrix}
1 &  & 0  \\
 & \ddots &  \\
0 &  & 1 \\
 -1 &\dots &-1 
\end{pmatrix} \in  \R^{P \times (P-1)}.
\end{align*}
Next we show that $A \Pi^{-1}  \in \R^{(P-1) \times (P-1)}$ is invertible and get a bound on the spectral norm $|(A \Pi^{-1})^{-1}|_2$. As $\langle \chi_i^\ell \rangle = V_i$ for all $i \in \{1, \dots, P\}$ and $\ell \in \{1, \dots, L\}$, it holds for $i\neq j \in \{1, \dots, P-1\}$ that
\begin{align*}
{A}_{i,j} &= \int_M \divrho ( \xi_{i}) \chi_j^\ell \,d\mu = \int_M (\chi_{i,\eps}^\ell - \langle \chi_i^\ell \rangle) \chi_j^\ell \,d\mu\\
&= \int_M (\chi_{i,\eps}^\ell - V_i) (\chi_j^\ell -  V_j) \,d\mu\\
&= \int_M (\chi_{i}^\ell - V_i) (\chi_j^\ell -  V_j) \,d\mu + \int_M (\chi_{i,\eps}^\ell -  \chi_i^\ell) (\chi_j^\ell -  V_j) \,d\mu\\
&= - V_i V_j + \int_M (\chi_{i,\eps}^\ell -  \chi_i^\ell) (\chi_j^\ell -  V_j) \,d\mu.
\end{align*}
Therefore, it holds
\begin{align*}
( A \Pi^{-1} )_{i,j} &= A_{i,j} - A_{i,P}\\
 &= \left(- V_i V_j  + V_i V_P\right) + \left(\int_M (\chi_{i,\eps}^\ell -  \chi_i^\ell) (\chi_j^\ell -  V_j) \,d\mu - \int_M (\chi_{i,\eps}^\ell -  \chi_i^\ell) (\chi_P^\ell -  V_P) \,d\mu\right) \\
 &=: A^{(0)}_{i,j} + A_{i,j}^{(\eps)}.
\end{align*}
The analogous computation for $i = j \in \{1, \dots, P-1\}$ yields
\begin{align*}
( A \Pi^{-1})_{i,i} &= \left( V_i - V_i^2  + V_i V_P\right) + \left(\int_M (\chi_{i,\eps}^\ell -  \chi_i^\ell) (\chi_i^\ell -  V_i) \,d\mu - \int_M (\chi_{i,\eps}^\ell -  \chi_i^\ell) (\chi_P^\ell -  V_P) \,d\mu\right)\\
 &=: A^{(0)}_{i,i} + A_{i,i}^{(\eps)}.
\end{align*}
Denote $V := (V_1, \dots, V_{P-1})^\intercal \in \R^{(P-1)}$ to rewrite
\begin{align*}
A^{(0)} =\diag(V) - V ( V - V_P \mathbbold{1})^\intercal.
\end{align*} 
Because 
\begin{align*}
1 - ( V - V_P \mathbbold{1})^\intercal  \diag^{-1}(V) V = P  V_P >0 
\end{align*}
it holds by the Sherman--Morrison formula \cite{MR35118,MR40068} that $A^{(0)}$ is invertible with inverse
\begin{align*}
\left(A^{(0)}\right)^{-1} = \diag^{-1}(V) + \frac{\mathbbold{1} \mathbbold{1}^\intercal - V_P \mathbbold{1} \mathbbold{1}^\intercal  \diag^{-1}(V)}{P  V_P}
\end{align*}
Thus, the bound on the spectral norm of the inverse of $A^{(0)}$ is given by 
\begin{align*}
\left|\left(A^{(0)}\right)^{-1}\right|_2 \leq \max_{i = 1,\dots, P} \frac{3}{V_i}
\end{align*}
To get invertibility of the whole matrix $A \Pi^{-1}$ we use the Neumann series. It holds 
\begin{align*}
A \Pi^{-1}  = A^{(0)}+ A^{(\eps)} = A^{(0)}  (I +\left(A^{(0)}\right)^{-1}  A^{(\eps)}).
\end{align*} 
The Neumann series converges to the inverse of $(I +{A^{(0)}}^{-1}  A^{(\eps)})$ if $\left|\left(A^{(0)}\right)^{-1}  A^{(\eps)}\right|_2 < 1$. By equation \eqref{eq:chi_eps} and $E_h(\chi^\ell) \leq E_0$ we get for $h \leq \eps^2$  that
\begin{align*}
\big|A^{(\eps)}\big|_2 \leq \big|A^{(\eps)}\big|_F \leq 2 \left( \sum_{i,j = 1}^{P-1} \left|\int_M |\chi_{i,\eps}^\ell -\chi_i^\ell| \,d\mu \right|^2 \right)^{\frac{1}{2}} \leq 2P \textbf{C}\eps( E_0 + 1).
\end{align*}
Now pick $\eps := \frac{\min_{i = 1,\dots, P} V_i}{12 P \textbf{C} (E_0 + 1)}$  such that
\begin{align*}
\Big|\left(A^{(0)}\right)^{-1} A^{(\eps)}\Big|_2 \leq \Big|\left(A^{(0)}\right)^{-1}\Big|_2 \Big|A^{(\eps)}\Big|_2\leq  \eps \frac{6 P \textbf{C} (E_0 + 1)}{ \min_{i = 1,\dots, P} V_i} \leq \frac{1}{2}.
\end{align*}
The Neumann series yields now
\begin{align*}
|(A\Pi^{-1})^{-1}|_2 &= \Big|\left(A^{(0)}\right)^{-1} \Big|_2 \Big| (I +\left(A^{(0)}\right)^{-1} A^{(\eps)})^{-1}\Big|_2 \\
&\leq \max_{i = 1,\dots, P} \frac{1}{V_i} \frac{1}{1-\Big|\left(A^{(0)}\right)^{-1} A^{(\eps)}\Big|_2 } \\
&\leq \max_{i = 1,\dots, P} \frac{2}{V_i}.
\end{align*}
Finally we define $\eta_{\ell}^\intercal := {\Lambda_\mu^\ell}^\intercal \Pi^{-1} (A \Pi^{-1})^{-1}$ to get
\begin{align*}
 \eta_\ell^\intercal A \Lambda_\mu^\ell = {\Lambda_\mu^\ell}^\intercal \Pi^{-1} (A \Pi^{-1})^{-1} A \Pi^{-1} \Pi  \Lambda_\mu^\ell = {\Lambda_\mu^\ell}^\intercal \Lambda_\mu^\ell = |\Lambda_\mu^\ell|_2^2
\end{align*} 
and
\begin{align*}
| \eta_\ell |_2 \leq |\Lambda_\mu^\ell|_2 |(A\Pi^{-1})^{-1}|_2 \leq |\Lambda_\mu^\ell|_2  \max_{i = 1,\dots, P} \frac{2P}{V_i}.
\end{align*}
Plugging everything into equation \eqref{eq:step_4_euler} and using the Cauchy--Schwarz inequality twice gives the claimed $L^2$-estimate, i.e.
\begin{align*}
\sum_{\ell = 1}^L |\Lambda_{\mu}^\ell|^2_2 &= \sum_{\ell = 1}^L \sum_{i = 1}^{P-1}\eta_{i, \ell} \Lambda_\mu^\ell \cdot \int_M \divrho ( \xi_{i}) \chi^\ell \,d\mu(x) \\
&\leq \left(\sum_{\ell = 1}^L \sum_{i = 1}^{P-1} |\eta_{i, \ell}|^2 \right)^{\frac{1}{2}} \left(\sum_{\ell = 1}^L \sum_{i = 1}^{P-1} \left(\Lambda_\mu^\ell \cdot \int_M \divrho ( \xi_{i}) \chi^\ell \,d\mu(x)\right)^2 \right)^{\frac{1}{2}}\\
&\leq  \left(\max_{i = 1, \dots, P} \frac{4P^2}{V_i^2 } \sum_{\ell = 1}^L |\Lambda_\mu^\ell|^2_2 \right)^{\frac{1}{2}}  \left( \frac{P^2}{h} \left(1 + \frac{1}{\eps}\right)^2 \left(T (1+ E_0^2) + (1+E_0^2)\right) \right)^{\frac{1}{2}}\\
&\leq  \frac{1}{h^2}\left(\sum_{\ell = 1}^L |\Lambda_\mu^\ell|^2_2 \right)^{\frac{1}{2}} \left(\max_{i = 1, \dots, P} \frac{4P^4}{V_i^2 }  \left(1 +  \frac{12 P \textbf{C} (E_0 + 1)}{ V_i}\right)^2 \left(T (1+ E_0^2) + (1+E_0^2)\right) \right)^{\frac{1}{2}}
\end{align*}
\end{proof}

\section{Numerics and Empirical Tests}
\subsection{Computational Complexity}
To compute the volume constrained MBO scheme it is necessary to calculate the diffused clusters $u = e^{-h \Delta_N}\chi$. But the calculation of the exponential matrix $e^{-h \Delta_N}$ is infeasible in practice. To circumvent this cost, one can use simpler kernels. For example Jacobs, Merkurjev and Esedo\={g}lu used the squared weight matrix as kernel in their auction dynamics algorithm \cite{jacobs2018auction}. 
We will discuss several possibilities of kernels that make to different degrees a trade-off between computational complexity and approximation accuracy of $e^{-h \Delta_N}$. The key property we do not want to violate is the unconditional stability stemming from the variational interpretation \eqref{eq:constrained_min_problem}.

We start be introducing the volume constrained MBO scheme for arbitrary kernels. For some initial clustering $\chi^0: X_N \rightarrow \{0,1\}$ and kernel $A: [0, \infty) \rightarrow \R^{N \times N} \cong (X_N \rightarrow X_N)$ that depends on the time step $h$, a clustering $\chi^{\ell -1 }$ is updated to $\chi^{\ell}$ by
\begin{align}\label{alg:volumeKernel}
\chi^\ell \in \argmax_{\chi:X_N \rightarrow \{0,1\}^P} &\sum_{i=1}^P \sum_{x \in X_N} \chi_i(x) \big(A(h) \chi^{\ell-1}_i\big)(x)\\
\text{s.t. } &\sum_{i=1}^P \chi_i(x) = 1 \quad \hspace{3pt} \text{ for all } x \in X_N, \nonumber\\
\hspace{20.5pt}&\!\!\sum_{x \in X_N}\!\! \chi_i(x) = V_i \quad \text{ for all } i = 1,\dots,P. \nonumber
\end{align}
The natural minimizing movement interpretation associated to the scheme (compare with Lemma~\ref{lem:lagrangeMulti_discrete}) is
\begin{equation}
\chi^\ell \in \argmin_{\chi:X_N \rightarrow \{0,1\}^P, \sum_{x \in X_N} \chi(x) = V} E_{A(h)}(\chi) + \frac{1}{2h} d_{A(h)}^2(\chi, \chi^{\ell -1})
\end{equation}
where the energy and distance term are given for some scaling function $s_A:(0, \infty) \rightarrow (0, \infty)$ by
\begin{align}
E_{A(h)}(\chi) &= \frac{1}{s_A(h)} \sum_{i \neq j} \frac{1}{N} \sum_{x\in X_N} \chi_i(x) (A(h) \chi_j)(x), \label{eq:energyAh}\\
\frac{1}{2h} d_{A(h)}^2(\chi, \tilde{\chi}) &= \frac{1}{s_A(h)}  \sum_{i = 1}^P \frac{1}{N} \sum_{x \in X_N} (\chi_i - \tilde{\chi}_i)(x) \big(A(h) (\chi_i - \tilde{\chi}_i)\big)(x). \label{eq:distance_Ah}
\end{align}
The computational complexity of the scheme consist of three parts, namely:
\begin{enumerate}
\item[\namedlabel{item:matrix_computation}{i)}] Calculate once for all iterations the kernel $A = A(h)$, usually depending on the weights $w:X_N \times X_N \rightarrow [0, \infty)$. For example, for the in Section \ref{sec:improved} discussed heat kernel matrix $e^{-h \Delta_N}$, one needs to compute the expensive exponential matrix of the scaled graph Laplacian $- h \Delta_N$. 
\item[\namedlabel{item:matrix_product}{ii)}] Compute $A \chi_i$ for all $i \in \{1, \dots, P\}$. These matrix vector multiplications take for the dense matrix $e^{- h \Delta_N}$ naively $\mathcal{O}(PN^2)$.
\item[\namedlabel{item:solve_opt}{iii)}] Solve the optimization problem \eqref{alg:volumeKernel} with Algorithm \ref{alg:median}. By the analysis of Section \ref{sec:improved} we expect for the kernel $e^{-h \Delta_N}$ a running time of $\mathcal{O}(\sqrt{h} N (\log N + P) P^2)$ for most iterations.
\end{enumerate}
It would be optimal, in regard of balancing the complexity of item \ref{item:matrix_product} and item \ref{item:solve_opt}, to find a matrix $A$ such that the matrix vector multiplication of item \ref{item:matrix_product} can be computed in $\mathcal{O}(\sqrt{h}N \log N)$. To that aim we can use a simple trick: By the trivial identity
\begin{equation}\label{eq:sped_up_trick}
A \chi^\ell = A (\chi^\ell - \chi^{\ell -1 }) + A \chi^{\ell - 1}
\end{equation}
one can compute the matrix vector product $A \chi^\ell$ by computing $A (\chi^\ell - \chi^{\ell -1 })$ and adding the stored product $A \chi^{\ell - 1}$ of the last iteration. The clou is now, when $A = e^{-h \Delta_N}$ is used  then $\chi^\ell - \chi^{\ell -1 }$ is expected to be $\sqrt{h}N$-sparse, meaning that the vector has only $\sqrt{h}N$ many non-zero entries. If additional the columns of the matrix $A(h)$ are $\log N$-sparse then the computation of $A (\chi^\ell - \chi^{\ell -1 })$ takes the sought after $\mathcal{O}(\sqrt{h} \log N)$ running time. We make the $\sqrt{h}N$-sparsity of $A (\chi^\ell - \chi^{\ell -1 })$ rigorous in the following under mild assumptions on the kernel $A(h)$, but will later not completely reach our aim to find a well suited kernel with $\log N$-sparse columns. 
\begin{assumption}\label{ass:kernel}\ 
\begin{itemize}
\item[1.]Symmetry: $A(h)^\intercal = A(h)$
\item[2.]Conservation of mass: $A(h) \mathbbold{1} = \mathbbold{1}$
\item[3.]Energy scaling: There exists a function $s_A: (0, \infty) \rightarrow (0, \infty)$ such that for all $\chi: X_N \rightarrow \{0,1\}^P$ it holds
\begin{align*}
\limsup_{h \rightarrow 0} \limsup_{N \rightarrow \infty} \frac{1}{s_A(h)}  \sum_{i \neq j} \frac{1}{N} \sum_{x\in X_N} \chi_i(x) \big(A(h) \chi_j\big)(x) < \infty. 
\end{align*}
\end{itemize}
\end{assumption}
\begin{rem}
For $A(h) = e^{-h \Delta_N}$ all the properties of Assumption \ref{ass:kernel} are satisfied with $s_A(h) = \sqrt{h}$.
\end{rem}

\begin{lemma}[$L^1$ estimate of $\chi^\ell - \chi^{\ell - 1}$]\label{the:l1estimate}
Denote by $\chi^\ell$ the iterates of the MBO scheme \eqref{alg:volumeKernel} with arbitrary diffusion matrix $A(h)$ that satisfies the Assumption \ref{ass:kernel}. Then it holds 
\begin{align*}
|\chi^\ell - \chi^{\ell - 1}|_{\ell^1(X_N)} \leq 4 s_A(h) E_{A(h)}(\chi^0).
\end{align*}
\end{lemma}
\begin{proof}
We follow the lines of \cite[Lemma 2.5]{MR3556529} . Adapting the minimizing movement interpretation of Lemma \ref{lem:lagrangeMulti_discrete} yields
\begin{align}\label{eq:min_mov_A}
E_{A(h)}(\chi^\ell) + \frac{1}{2h} d_{A(h)}^2(\chi^\ell, \chi^{\ell - 1}) \leq E_{A(h)}(\chi^{\ell - 1}).
\end{align}
We want to bound $|\chi^\ell - \chi^{\ell - 1}|_{\ell^1(X_N)}$ in terms of  $d_{A(h)}(\chi^\ell, \chi^{\ell - 1})$. To this aim, we observe that for any $\chi, \tilde{\chi} \in \{0,1\}$ and any $u, \tilde{u} \in [0,1]$ it holds
\begin{align*}
|\chi - \tilde{\chi}| &= (\chi - \tilde{\chi}) (\chi - \tilde{\chi}) \\
& = (\chi - \tilde{\chi})(u - \tilde{u}) + (\chi - \tilde{\chi})(\chi -u - (\tilde{\chi} - \tilde{u})) \\
&\leq (\chi - \tilde{\chi})(u - \tilde{u}) + |\chi - u | + |\tilde{\chi} - \tilde{u}|.
\end{align*}
Hence for $\chi = \chi_i^\ell, \tilde{\chi} = \chi_i^{\ell -1}, u - \tilde{u}= A(h)(\chi_i^\ell - \chi_i^{\ell - 1})$ one concludes 
\begin{align*}
|\chi_i^\ell - \chi_i^{\ell - 1}|_{\ell^1(X_N)} &= \frac{1}{N} \sum_{x \in X_N} |\chi_i^\ell(x) - \chi_i^{\ell - 1}|\\
&\leq  \frac{1}{N} \sum_{x \in X_N} (\chi_i^\ell(x) - \chi_i^{\ell - 1}(x))\big(A(h)(\chi_i^\ell - \chi_i^{\ell - 1})\big)(x) \\
& \hspace{12pt}+ \frac{1}{N} \sum_{x \in X_N}|\chi_i\ell(x) - \big(A(h)\chi_i^\ell\big)(x) | + |\chi_i^{\ell - 1}(x) - \big(A(h) \chi_i^{\ell - 1}\big)(x)|.
\end{align*}
Summing over $i = 1, \dots, P$, we observe that the first term is just $\frac{1}{2h} s_A(h) d_{A(h)}^2(\chi^\ell, \chi^{\ell - 1})$ and for the second one we use $\chi_i^\ell \in \{0,1\}$ and $A(h)\chi_i^\ell \in [0,1]$ as well as $A(h) \mathbbold{1} = \mathbbold{1}$ to calculate
\begin{align*}
&\frac{1}{N} \sum_{x \in X_N} |\chi_i^{\ell - 1}(x) - (A(h) \chi_i^{\ell - 1})(x)|\\
 &= \frac{1}{N} \sum_{x \in X_N} \chi_i^{\ell - 1}(x)(1 - (A(h) \chi_i^{\ell - 1})(x)) + (1-\chi_i^{\ell - 1}(x)) (A(h) \chi_i^{\ell - 1})(x)\\
&= \frac{1}{N} \sum_{x \in X_N} \chi_i^{\ell - 1}(x)(A(h)(1 - \chi_i^{\ell - 1}))(x) + (1-\chi_i^{\ell - 1}(x)) (A(h) \chi_i^{\ell - 1})(x).
\end{align*}
Summing over $i = 1, \dots, P$, the right-hand side is equal to $2 s_A(h)  E_{A(h)}(\chi^\ell)$.
Iterating inequality \eqref{eq:min_mov_A} yields $E_{A(h)}(\chi^\ell) + \frac{1}{2h}d^2_{A(h)} (\chi^\ell, \chi^{\ell -1}) \leq E_{A(h)}(\chi^0)$ and thus we conclude
\begin{align*}
|\chi^\ell - \chi^{\ell - 1}|_{\ell^1(X_N)} &\leq \frac{1}{2h} s_A(h) d_{A(h)}^2(\chi^\ell, \chi^{\ell - 1}) + 2 s_A(h) (E_{A(h)}(\chi^\ell) + E_{A(h)}(\chi^{\ell-1}))\\
&\leq 4 s_A(h) E_{A(h)}(\chi^0). \qedhere
\end{align*}
\end{proof}
\begin{rem}\label{rem:direct_bound}
If $A(h)$ is positive definite, one can also directly estimate the $\ell^1$-norm against the distance term:
\begin{align*}
|\chi^\ell - \chi^{\ell - 1}|_{\ell^1(X_N)} &= \frac{1}{N} \sum_{x \in X_N} |\chi^\ell(x) - \chi^{\ell-1}(x)|^2 \\
&\leq \frac{1}{\lambda_{min} N} \sum_{x \in X_N} (\chi^\ell(x) - \chi^{\ell-1}(x)) A(h)(\chi^\ell(x) - \chi^{\ell-1}(x))\\
& = \frac{s_A(h)}{\lambda_{min}} \frac{1}{2h} d_A(h)^2(\chi^\ell, \chi^{\ell - 1}).
\end{align*}
For $A(h) = e^{- h \Delta_N}$ one has with the correct scaling (see \eqref{def:graph_laplace}) of the Laplacian that 
\begin{align*}
\lim_{N \rightarrow \infty} \lambda_{min}(A(h)) = 0
\end{align*}
such that for $N$ big the estimate degenerates. Our statement in Lemma \ref{the:l1estimate} is more robust. 
\end{rem}

\subsection{Different Diffusion Kernels}\label{sec:empiric}
The first kernel we want to have a look at is the best possible approximation of $e^{-h \Delta_N}$ with rank $K = \log N$. We will see that from a theoretical viewpoint this matrix has a lot of strong properties and leads to a run time of $\mathcal{O}(N \log N)$ in the computation of the diffusion step \eqref{item:matrix_product}.

Denote the eigenvalues and eigenvectors of the graph Laplacian $\Delta_N$ by $\lambda_1 \leq \dots \leq \lambda_N$ and $v_1 ,\dots, v_N$, respectively. Then the heat operator applied to some clustering $\chi$ is given by
\begin{align*}
e^{-h \Delta_N} \chi_i(x) = \sum_{k = 1}^N e^{-h \lambda_i} \langle v_k, \chi_i \rangle_{\ell^2(X_N)} v_k(x) \quad \text{ for all } x \in X_N.
\end{align*}
As a cheaper approximation we take 
\begin{align*}
e^{-h \Delta_N}_K \chi_i(x) = \sum_{k = 1}^K e^{-h \lambda_k} \langle v_k, \chi_i \rangle_{\ell^2(X_N)} v_k(x) \quad \text{ for all } x \in X_N
\end{align*}
with $K \approx \log(N)$. The matrix $e^{-h \Delta_N}_K$ is the best rank-$K$ approximation of $e^{-h \Delta_N}$ which means that it is the matrix minimizing the distance to $e^{-h \Delta_N}$ in the Frobenius norm under all matrices with rank $K$. We have three motivations for choosing $K \approx \log(N)$:
\begin{itemize}
\item[1.] Lelmi and the second author proved in \cite{laux2023large} the convergence of the two phase MBO scheme to mean curvature flow viscosity solutions if $K \geq \log^q(N)$ for an explicit exponent $q >0$. 
\item[2.] The computation of $e^{-h \Delta_N}_K \chi$ requires $\mathcal{O}(N K) = \mathcal{O}(N\log(N))$ time, which matches the asymptotic running time of our Algorithms \ref{alg:median} and \ref{alg:lower_upper}. Thus the diffusion step \eqref{item:matrix_product} and thresholding step \eqref{item:solve_opt} are balanced in worst case running time. However, for small time steps $h$ we expect the thresholding step to be in $\mathcal{O}(\sqrt{h} N \log N)$ while the diffusion still takes $\mathcal{O}(N \log N)$.
\item[3.] The Johnson-Lindenstrauss lemma \cite{johnson1986extensions, larsen2017optimality} states that one can almost isometrically embed $N$ points into $\log(N)$ dimensions. More precisely, given $N$ points in $\R^N$, there exist $N$ points of dimension $\mathcal{O}\left(\frac{\log(N)}{\delta^2} \right)$ such that the distances between the points only changes by a factor of $(1 \pm \delta)$. The lemma also tells that the scaling of $\frac{\log(N)}{\delta^2}$ for the dimension (up to constants) is necessary to ensure this ``almost isometry'' property \cite{larsen2017optimality}. In our case, the graph Laplacian is an $(N \times N)$-matrix that gets compressed to the $\log(N)$-dimensional subspace spanned by the first $K$ eigenvectors.
\end{itemize}
Note that the eigenvalues and eigenvectors only need to be computed once for a dataset. In our tests, the computational time for that was comparable to the assembly of the weight matrix when working with $k$-nearest-neighbor-graphs.

The only downside of the rank-$K$ approximation is, that the sparsity of $\chi^\ell - \chi^{\ell -1 }$ of Theorem \ref{the:l1estimate} doesn't improve the running time of the matrix vector product as $e^{-h \Delta_N}_K$ is only of low rank but not sparse. 

Thus, we also pursue a second approach to get sparser matrices. Using the Taylor series of the matrix exponential one writes
\begin{equation}
e^{-h \Delta_N} = \sum_{j = 0}^\infty \frac{(-h \Delta_N)^j}{j!}.
\end{equation}
For simplicity consider the random walk graph Laplacian $\Delta_N = I - D^{-1} W$ with $D$ the diagonal matrix with entries $d(x):= \sum_{y \in X_N} w(x,y)$ and $W$ the weight matrix with entries $w(x,y)$. The graph Laplacian $ \Delta_N$ is sparse when using $k$-nearest neighbor graphs. To be more precise it only has $N k$ non-zero entries with $k$ non-zero entries in every row. By our manifold assumption, Assumption \ref{ass:manifold}, one also expects that every column has $\mathcal{O}(k)$ non zero elements in every column. The number of neighbors $k$ should be chosen at least by $k = \mathcal{O}(\log N)$ as otherwise the resulting graph would be disconnected with high probability. Also in the case of random geometric graphs as in our Assumption \ref{ass:manifold}, the typical number of neighbors of a point is $\log N$. Thus, the approximation 
\begin{align}\label{kern:best-rank}
e^{-h \Delta_N}_{Tay,J} := \sum_{j = 0}^J \frac{(-h \Delta_N)^j}{j!}.
\end{align}
leads to a expected running time of $\mathcal{O}(\sqrt{h} N \log^J(N))$ in the matrix vector multiplication \eqref{item:matrix_product}. But, problematic with this choice of kernel for small $J$ is no matter which choice of $h \in (0, \infty)$ is taken, one either has problems with pinning or the kernel is not positive semi-definite. For example, regard the case $J = 1$, that is $A(h) = (1-h)I + h D^{-1}W$. On the one hand, for $h \leq 0.5$ the scheme is pinned, meaning that $\chi^\ell = \chi^{\ell - 1}$, so the scheme frozen at the initial condition. On the other hand, for $h > 0.5$ the matrix is not positive semi-definite such that the distance term \eqref{eq:distance_Ah} is not guaranteed to be non-negative anymore. Thus the energy \eqref{eq:energyAh}, which is in this case the for clustering very suited Graph Total Variation, doesn't need to be dissipated. For bigger $J$ the pinning problem gets only slightly better.

Therefore we propose another approach where we use that $e^{-h \Delta_N} = e^{-h}e^{hD^{-1}W}$ which leads to the approximation 
\begin{align}\label{kern:positive_series}
e^{-h \Delta_N}_{Pos, J} := \frac{1}{\sum_{j=0}^J \frac{h^j}{j!}} \sum_{j = 0}^J \frac{h^j \big( D^{-1}W \big)^j}{j!}.
\end{align}
Note the prefactor is chosen such that $e^{-h \Delta_N}_{Pos, J}\chi (x) \in [0,1]$ for $\chi:X_N \rightarrow \{0,1\}$. For even $J$ one can check that $e^{-h \Delta_N}_{Pos, J}$ is positive semi-definite and therefore we have unconditional stability.

We will also have a look at just the matrix $(D^{-1}W)^\intercal D^{-1}W$ which can be seen as the limit $h \rightarrow \infty$ of $e^{-h \Delta_N}_{Pos, 2}$ and was used in \cite{jacobs2018auction}. The matrix-vector multiplication takes $\mathcal{O}(\eps_N N \log^2(N))$ and can be sped up by \eqref{eq:sped_up_trick} as $\chi^{\ell} - \chi^{\ell-1}$ is sparse, which follows from Lemma \ref{the:l1estimate} and $s_N(h) = \eps_N$, and \cite{MR3458162}. In the spirit of Remark \ref{rem:direct_bound} on expects more sparsity in later iterations. Additionally also $((D^{-1}W)^\intercal D^{-1}W)^2 = \lim_{h \rightarrow \infty} e^{-h \Delta_N}_{Pos, 4}$ is tested as this matrix connects nodes further apart in the graph and thus spreads more information. The running time for the matrix vector multiplication is similar to before given by $\mathcal{O}(\eps_N N \log^4(N))$.

One can also expand the Taylor series around an other point than zero. Therefore one wants to use $e^{-h \Delta_N} = e^{-(h-r)I}e^{h D^{-1}W - rI}$. To get a positive semi-definite matrix, we use $ e^{-(h-r)I}e^{h (D^{-1}W)^\intercal D^{-1}W - r I}$ where $r = \lambda_1$, the smallest eigenvalue of $(D^{-1}W)^\intercal D^{-1}W$. This $r$ is the biggest choice such that $(D^{-1}W)^\intercal D^{-1}W- r I$ is positive semi-definite. Note that $\lambda_1$ can be $0$ but in practice even a $r > \lambda_1$ worked well in our tests. The evolution of $ e^{-(h-r)I}e^{h (D^{-1}W)^\intercal D^{-1}W - r I}$ yields
\begin{align}
\sum_{j = 0}^J \frac{h^j \big( (D^{-1}W)^\intercal D^{-1}W - r I\big)^j}{j!}.
\end{align}
The motivation behind this kernel is that by the subtraction of the identity one wants to prevent pinning. Mathematically speaking, a kernel with small but positive eigenvalues is ideal as it makes the distance- like term $\frac{1}{2h} d_{A(h)}^2(\cdot,\cdot)$ in the minimizing movement interpretation \eqref{eq:min_mov_A} small, which means that larger energy barriers can be overcome and solutions with lower energy can be found. We again take the limiting kernel 
\begin{align}\label{kern:sub_identity}
(D^{-1}W)^\intercal D^{-1}W - r I
\end{align}
for our tests such that the matrix-vector multiplication is expected to be in $\mathcal{O}(\eps_N N \log^2 N)$.
\subsection{Test Setup}
As initialization of our clusters we use a slight variation of Voronoi tessellation~\cite{jacobsvoronoi} which have been used in the auction dynamics \cite{jacobs2018auction}. Laguerre tesselations or sometimes called weighted tessellations have the advantage that they can be used to find an initial clustering that satisfies the volume constraints already. When the initial clustering doesn't fulfill the volume constraints and a quite local kernel like $(D^{-1}W)^\intercal D^{-1}W$ is used, it can happen that not enough nodes get information from a phase with not enough volume such that one is forced to add arbitrary points to it. This can lead to solutions with higher energy which is not favorable. To define Laguerre tesselations in the semi-supervised learning set up assume we have a fidelity set $Y \subset X_N$ whose labels are known. Thus, it is given that $\chi(y) = e_i$ for all $y \in Y_i$ with $Y_i$ the set of all points with label $i$ and $\dot{\cup}_{i = 1, \dots, P} Y_i = Y$. For our points $X_N$ with labeled data $Y$ the $i$-th Laguerre cell is
\begin{align}\label{eq:laguerre}
L_i(Y, m) := \Big\{x \in X_N: \dist(x, y_i) - m_i \leq \dist(x, y_j) - m_j \  \forall j \neq i, y_i \in Y_i, y_j \in Y_j\Big\}.
\end{align}
The distance to the points can be computed efficiently using Dijkstra's algorithm and the weights $m$ such that the volume constraints are satisfied can be computed with our Algorithm \ref{alg:median}.
We use this initialization for all the kernels except for the best rank-$K$ approximation \eqref{kern:best-rank}. For the kernel \eqref{kern:best-rank} we use the approximate heat kernel to get an initial clustering by setting 
\begin{align*}
\delta_Y(x) := \begin{cases}
e_i & \text{ if } x \in Y_i,\\
0 & \text{ if } x \notin Y,
\end{cases} 
\end{align*}
which is comparable to the union of discrete Dirac deltas of the points in $Y$. We then propose to take 
\begin{align}\label{eq:init_by_dif}
\{\chi^0_i = 1\} := \left\{\left(e^{-h \Delta_N}_K \delta_Y\right)_i - m_i \geq \left(e^{-h \Delta_N}_K \delta_Y\right)_j  - m_j\ \forall j\neq i \right\}.
\end{align}
as initial clustering in the spirit of Varadhan~\cite{MR208191}, see also Crane et al.~\cite{crane2017heat}. In the continuous, Euclidean setting with constant density this is equivalent to Laguerre tessellations as $e^{-h \Delta_N}_K \delta_x$ is radially symmetric around $x$. But on the discrete data the heat kernel weights and considers all possible paths between two points and not only the shortest one like the distance function does. In Table \ref{tab:comp_init} the accuracy of this initial clustering is compared to accuracy of the initial clustering by Voronoi regions as described in \cite{jacobsvoronoi}. A comparison between Laguerre and Voronoi cells can be found in Table \ref{tab:comp_init}. 

As stopping criterion we use the relative change in the thresholding energy $E_{N,h}$ of \eqref{eq:discrete_thresholding_energy}. Denoting by $\chi^\ell$ the clustering in the $\ell$-th iteration of the volume constrained MBO scheme, our algorithm is stopped if $\left(E_{N,h}(\chi^{\ell+1}) - E_{N,h}(\chi^\ell)\right)/ E_{N,h}(\chi^{\ell+1}) < \eps$ for  $\eps \approx 10^{-4}$. Note that by the minimizing movement interpretation of Lemma \ref{lem:lagrangeMulti_discrete} it is guaranteed that $\chi^\ell$ converges to a local minimum. We test two diffusion times $h = 1$ and $h = 100$.

We use three different datasets to test our kernels that have also been used in \cite{jacobs2018auction} for their tests.

\textbf{MNIST} is a dataset of grayscale images of size $28\times 28$ of handwritten digits. We combine the $60,000$ training images with the $10,000$ test images to one dataset of $70,000$ images. Otherwise no changes are done to the dataset. 
\textbf{Opt-Digits} is also a dataset of of handwritten digits but with size $8 \times 8$ pixels and every pixel has values between $0$ and $16$.
\textbf{Three Moons} is a synthetic dataset that is constructed by sampling out of three half-circles and adding normal distributed noise. The centers of the circles are at $(0,0), (3,0)$ and $(1.5, 0.4)$ with radii of $1, 1$ and $1.5$ respectively. The first two half-circles are open to the bottom and the third is opened at the top. The two dimensional points are then embedded into $\R^{100}$ and Gaussian noise with mean zero and standard deviation $0.14$ is added. From each half-circle we sampled $500$ points such that the dataset consists of $1500$ points in total. 

As weight matrix we take for easy comparison the weight matrix proposed in \cite{calder2020poisson}. This means we take the $K$-nearest neighbor graph with weights given by
\begin{align*}
w(x,y) = e^{-\frac{4|x-y|^2}{d_K^2(x)}}.
\end{align*}
Here, the euclidean norm was used and $d_K(x) = |x - x_K|$ is a local scaling parameter where $x_K$ denotes the $K$-th nearest neighbour of $x$. For all datasets we used $K = 10$ as in \cite{calder2020poisson}. For best comparison to \cite{calder2020poisson} on MNIST we used their pretrained variational autoencoder (for more information see \cite{calder2020poisson}) to embed the points into a lower dimensional feature space before computing the weights. On Opt-Digits and Three Moons such an embedding was not done. In the end the matrix $W' = (W^\intercal + W)/2$ is used to ensure symmetry.

On all three datasets we once test with one known label per class and once with five known labels per class. We also state the results when temperature is introduced as it was done in \cite{jacobs2018auction}. Temperature is a random perturbation of the diffused values before the thresholding step is applied; it helps to escape local minimas which leads to solutions with lower energy. If temperature is used we use the clustering with lowest energy over a fixed number of $50$ iterations. The volume constraints are chosen exact and it is assumed that the correct number of datapoints in every cluster is known. The results are given in Table \ref{tab:comp_MNIST} for MNIST, Table \ref{tab:comp_opt_digits} for Opt-Digits and Table \ref{tab:comp_tree_moons} for Three Moons. 

\subsection{Results and Comparison to Other Methods}
Overall, our results summerized in Tables \ref{tab:comp_MNIST}, \ref{tab:comp_opt_digits} and \ref{tab:comp_tree_moons} over the three datasets show that there is not a specific kernel that outperforms all other kernels. But some observations follow in order. First, subtracting a scaled identity matrix as done in \eqref{kern:sub_identity} seems to be beneficial to prevent pinning which yields better accuracy although the matrix is not guaranteed to be positive definite. The best rank-$K$ approximation \eqref{kern:best-rank} has its advantages on the Three Moons dataset when only one point per class is labeled. In comparison to the other kernels the rank-$K$ approximation is a global kernel meaning that every node gets information from every other node in the diffusion step. This is helpful to prevent local minimas described by cutting moons with a straight line into two pieces. The kernel has, for $h = 100$, also good accuracy when comparing to the other methods on MNIST dataset. The connection to spectral clustering for $h$ large, analyzed in the Introduction \ref{sec:intro_context}, explains the increase in accuracy when compared to $h=1$. Surprisingly, the kernel \eqref{kern:positive_series} of the positive Taylor expansion yields always worse results then simply using $(D^{-1} W)^\intercal D^{-1} W$. This is unexpected as the added $D^{-1} W$ term included in \eqref{kern:positive_series} should help to minimize the Graph Cut energy which is ideal for clustering. One explanation for this observation could be that $D^{-1} W$ is not positive definite and thus less energy is dissipated in every iteration compared to using the kernel $(D^{-1} W)^\intercal D^{-1} W$. On MNIST and Opt-Digits the $((D^{-1} W)^\intercal D^{-1} W)^2$ kernel yields the best results but is also the most expensive to compute (see previous section). Compared to the other graph based methods (see Table \ref{tab:comp_MNIST}) it also has the highest accuracy of $97.5\%$ when five points of every cluster are labeled. When only one point is labeled only the PoissonMBO \cite{calder2020poisson} and the rank $K$ approximation yield better results.

\subsection{Connection to Assignment Problem}
Although not designed for it, from a theoretical perspective our algorithm is also efficient for the much studied assignment problem. The problem of solving \eqref{alg:volumeMBO} can be reduced to solving the assignment problem as done in \cite{jacobs2018auction}. The other way around, we can also use \eqref{alg:volumeMBO} to solve the assignment problem that reads as follows. Assume you have $N$ workers $\{a_i\}_i$ and $N$ tasks $\{b_i\}_i$ given and you want to assign to each task exactly one worker. Also assume that it costs $c(a,b) \geq 0$ to assign worker $a$ to task $b$. The assignment problem asks to find an assignment $f:\{a_i\}_i \rightarrow \{b_i\}_i$ that minimizes the cost, i.e.
\begin{align*}
\argmin_{f:\{a_i\}_i \rightarrow \{b_i\}_i \text{ bijective}} \sum_{i = 1}^N c(a_i, f(a_i)).
\end{align*}
We can solve this problem by solving
\begin{align*}
\argmax_{\chi:\{a_i\}_i \rightarrow \{0,1\}^N} -\sum_{i = 1}^N \sum_{j = 1}^N \chi_j(a_i) c(a_i, b_j) \quad \text{s.t. } & \sum_{i = 1}^N \chi_i(a_j) = 1\ \forall j = 1, \dots, N,\\
& \sum_{j = 1}^N \chi_i(a_j) = 1\ \forall i = 1, \dots, N,
\end{align*}
and setting $f(a_i) := b_{j(i)}$ where $j(i)$ is the index with $\chi_{j(i)}(a_i) = 1$. Using Algorithm \ref{alg:median} with $P = N$ to solve this problem one gets a running time of $\mathcal{O}(N \log N P^2 + N P^3) = \mathcal{O}(N^4)$. One can even improve this by recognizing that only $P$ hyperplanes change in line $7$ of Algorithm \ref{alg:median}. Thus finding the closest point to the hyperplane in the moved direction can be done by taking the minimum of the previous hyperplanes and the $P$ changed hyperplanes. Using a priority queue to maintain the minimum yields a total running time of $\mathcal{O}(N \log N P^2 + N P^2 \log P) = \mathcal{O}(N^3 \log N)$. The best known algorithm for the assignment problem is the ``Hungarian method'' with running time $\mathcal{O}(N^3)$ \cite{korte, MR0075510, MR0093429}. So our algorithm is asymptotically only the factor $\mathcal{O}(\log(N))$ slower, which is surprising as the algorithm is designed for the case $P \ll N$.

\begin{table}
\begin{tabular}{ |c||c|c| }
\hline
 Method 				& after initialization 		& after MBO \\
\hline
\hline
 Diffusion \eqref{eq:init_by_dif} without volume constraints ($m = 0$)  	& 93.02\%    	&92.51\% 	\\
 \hline
  Diffusion \eqref{eq:init_by_dif} with volume constraints  &   94.62\%  	& 95.29\%   \\
 \hline
  Voronoi 	\cite{jacobsvoronoi}	&   85.92\%  	& 93.22\%   \\
 \hline
 Laguerre \eqref{eq:laguerre} &   87.32\%  	& 93.45\%   \\
 \hline
\end{tabular}
\caption{Comparison of the accuracy of different initialization strategies on the OptDigits dataset. First column is accuracy after the initialization; second column after the MBO scheme with kernel $(D^{-1}W)^\intercal D^{-1}W$, exact volume constraints and initial clustering given by the respective method.}
\label{tab:comp_init}
\end{table}

\begin{table}
\begin{tabular}{ |c||c|c| }
\hline
 Kernel & $1$ & $5$ \\
\hline
\hline
\multirow{2}{*}{$e^{-h \Delta_N}_{20\log N}\text{ c.f. } \eqref{kern:best-rank}, h = 1$}  	& $87.5 (7.2)$    	& 	 $95.0 (2.1)$\\
& $\mathbf{87.0 (6.9)}$   	& $\mathbf{95.2 (2.3)}$ \\
 \hline
 \multirow{2}{*}{$e^{-h \Delta_N}_{20\log N}\text{ c.f. } \eqref{kern:best-rank}, h = 100$}  	& $96.3 (4.6)$   	& 	 $97.1 (0.0) $	 \\
& $\mathbf{95.3 (3.9)}$   	& $\mathbf{97.0 (0.2)}$ \\
 \hline
\multirow{2}{*}{$e^{-h \Delta_N}_{Pos, 2} \text{ c.f. }\eqref{kern:positive_series}$	}	&   $60.7 (4.5)$  	& $79.8 (2.4)$   \\
& $\mathbf{67.12 (5.5)}$   	& $\mathbf{86.9 (2.0)}$ \\
 \hline
\multirow{2}{*}{ $D^{-1}W)^\intercal D^{-1}W - 0.1 I \text{ c.f. } \eqref{kern:sub_identity}$}&   $90.0 (6.5)$  	& $96.7 (0.8)$   	\\
& $\mathbf{91.3 (5.2)}$   	& $\mathbf{96.3 (2.4)}$ \\
 \hline
  $(D^{-1}W)^{\intercal} D^{-1}W$&   $80.3 (6.5)$  	& $95.2 (1.2)$   	\\
  (as in auction dynamics \cite{jacobs2018auction}) & $\mathbf{89.0 (5.9)}$   	& $\mathbf{96.8 (0.4)}$ \\
 \hline
 \multirow{2}{*}{ $\left((D^{-1}W)^{\intercal} D^{-1}W\right)^2$}&   $94.8 (5.6)$  	& $97.5 (0.2)$   	\\
 & $\mathbf{94.8 (4.7)}$   	& $\mathbf{97.2 (1.3)}$ \\
 \hline
 Laplace/ LP \cite{10.5555/3041838.3041953} & $16.1 (6.2)$ & $69.5 (12.2)$ \\
  \hline
 Nearest Neighbor & $55.8 (5.1)$ & $74.1 (2.4)$ \\
  \hline
 Random Walk \cite{10.1007/978-3-540-28649-3_29} & $66.4 (5.3) $ & $84.5 (2.0)$ \\
   \hline
 MBO \cite{6714564} & $19.4 (6.2)$ & $59.2 (6.0)$ \\   \hline
 WNLL \cite{shi2017weighted} & $55.8 (15.2) $ & $94.6 (1.1)$ \\  \hline
 Centered Kernel \cite{JMLR:v19:17-421} & $19.1 (1.9) $ & $35.6 (4.6)$ \\  \hline
 Sparse LP \cite{jung2017semisupervisedlearningsparselabel} & $14.0 (5.5) $ & $16.2 (4.2)$ \\  \hline
 p-Laplace \cite{FLORES202277} & $72.3 (9.1) $ & $91.9 (1.0)$ \\  \hline
 Poisson \cite{calder2020poisson} & $90.2 (4.0)$ & $95.3 (0.7)$ \\  \hline
 PoissonMBO \cite{calder2020poisson} & $96.5 (2.6)$ & $97.2 (0.1)$ \\  \hline
\end{tabular}
\caption{Comparison of the different kernels on MNIST dataset with $1$ and $5$ labels per class averaged over $100$ trials. Bold with temperature as in \cite{jacobs2018auction} and non-bold without temperature.}
\label{tab:comp_MNIST}
\end{table}
\begin{table}
\begin{tabular}{ |c||c|c| }
\hline
 Kernel & $1$ & $5$ \\
\hline
\hline
\multirow{2}{*}{$e^{-h \Delta_N}_{20\log N}\text{ c.f. } \eqref{kern:best-rank}$ }  	& $90.7 (6.7)$   	& 	 $97.1 (1.5) $	 \\
& $\mathbf{90.8 ( 6.5)}$   	& $\mathbf{96.9 (1.6)}$ \\
\hline
\multirow{2}{*}{$e^{-h \Delta_N}_{Pos, 2} \text{ c.f. }\eqref{kern:positive_series}$	}	&   $81.2 (4.7)$  	& $91.7 (1.6)$   \\
& $\mathbf{84.4 ( 4.4)}$   	& $\mathbf{94.4 (1.5)}$ \\
 \hline
\multirow{2}{*}{ $D^{-1}W)^\intercal D^{-1}W - 0.1 I \text{ c.f. } \eqref{kern:sub_identity}$}&   $90.6 (5.1)$  	& $97.2 (1.2)$   	\\
& $\mathbf{92.9 ( 4.0)}$   	& $\mathbf{97.0 (0.9)}$ \\
 \hline
  \multirow{2}{*}{$(D^{-1}W)^{\intercal} D^{-1}W$}&   $88.1 (5.2)$  	& $96.5 (1.4)$   	\\
  & $\mathbf{92.4 (4.0)}$   	& $\mathbf{97.5 (1.0)}$ \\
 \hline
 \multirow{2}{*}{ $\left((D^{-1}W)^{\intercal} D^{-1}W\right)^2$}&   $93.3 (5.1)$  	& $97.9 ( 1.1)$   	\\
 & $\mathbf{95.6 (3.7)}$   	& $\mathbf{98.2 (0.6)}$ \\
 \hline
\end{tabular}
\caption{Comparison of the different kernels on Opt-Digits dataset with $1$ and $5$ labels per class averaged over $100$ trials. Bold with temperature as in \cite{jacobs2018auction} and non-bold without temperature.}
\label{tab:comp_opt_digits}
\end{table}
\begin{table}
\begin{tabular}{ |c||c|c| }
\hline
 Kernel & $1$ & $5$ \\
\hline
\hline
\multirow{2}{*}{$e^{-h \Delta_N}_{\log N}\text{ c.f. } \eqref{kern:best-rank}$}  	& $91.6 (13.0)$   	& 	 $94.5 (0.1)$ \\
& $\mathbf{94.2 (0.6)}$   	& $\mathbf{94.4 (0.3)}$ \\
 \hline
\multirow{2}{*}{$e^{-h \Delta_N}_{Pos, 2} \text{ c.f. }\eqref{kern:positive_series}$	}	&   $79.1 (11.0)$  	& $92.4 (3.9)$   \\
& $\mathbf{82.4 (10.1)}$   	& $\mathbf{97.1 ( 2.3)}$ \\
 \hline
\multirow{2}{*}{ $D^{-1}W)^\intercal D^{-1}W - 0.1 I \text{ c.f. } \eqref{kern:sub_identity}$}&   $88.5 (15.5)$  	& $97.9 (2.8)$   	\\
& $\mathbf{87.2 (12.5)}$   	& $\mathbf{95.2 (0.3)}$ \\
 \hline
  \multirow{2}{*}{$(D^{-1}W)^{\intercal} D^{-1}W$}&   $87.2 (15.3)$  	& $98.2 (3.0)$   	\\
  & $\mathbf{87.8 (13.1)}$   	& $\mathbf{98.0 (0.2)}$ \\
 \hline
 \multirow{2}{*}{ $\left((D^{-1}W)^{\intercal} D^{-1}W\right)^2$}&   $90.0 (15.4)$  	& $97.1 (1.8)$   	\\
 & $\mathbf{91.8 ( 11.4)}$   	& $\mathbf{96.3 (0.3)}$ \\
 \hline
\end{tabular}
\caption{Comparison of the different kernels on Three Moons dataset with $1$ and $5$ labels per class averaged over $100$ trials. Bold with temperature as in \cite{jacobs2018auction} and non-bold without temperature.}
\label{tab:comp_tree_moons}
\end{table}

\section{Code Availability}
Code for Algorithm \ref{alg:median} and Algorithm \ref{alg:lower_upper} as well as the volume-preserving MBO scheme with its different diffusion kernels can be found at \url{https://github.com/fabiuskt/VolumeMBO_OS}. The weight matrices are constructed with the help of the Graph Learning Library \cite{graphlearning}.

\section*{Acknowledgments}
This project has received funding from the Deutsche Forschungsgemeinschaft (DFG, German Research Foundation) 
under Germany's Excellence Strategy -- EXC-2047/1 -- 390685813 and the Research Training Group 2339 IntComSin -- Project-ID 321821685.

\frenchspacing
\bibliographystyle{abbrv}
\bibliography{volumeMBO}
  \end{document}